\documentclass[11pt,a4paper,reqno]{amsart}
 
\usepackage{amsmath,amssymb,amsthm,bbm} 
\usepackage{enumitem} 
\usepackage[overload]{textcase}
 
\usepackage{epsfig,psfrag,color} 

\usepackage{fullpage}

\usepackage{tikz}
\usetikzlibrary{calc}
\usetikzlibrary{shapes,arrows,positioning,decorations.pathreplacing,calc}

\usepackage[colorlinks,bookmarks,linkcolor=black,citecolor=black]{hyperref} 

\usepackage[english,algoruled,lined,noresetcount,norelsize]{algorithm2e}

\setlist{itemsep=1pt,parsep=0pt,topsep=2pt,partopsep=0pt}  
\setenumerate{leftmargin=*} 

\def\abc{\textup{(\alph{*})}}
 
\def\endofFact{\hfill\scalebox{.6}{$\Box$}}

\SetTitleSty{textbf}{}
\SetArgSty{textrm}

\SetKw{Failure}{failure}
\SetKw{Return}{return}
\SetKw{NOT}{not}
\SetKwFor{RepeatForever}{repeat}{forever}{end}

\let\subset\subseteq  
\let\eps\varepsilon 
\let\rho\varrho 
  
\newcommand{\rev}[1]{\overleftarrow{#1}}

\def\cC{\mathcal{C}}

\def\cE{\mathcal{E}}

\def\cH{\mathcal{H}}
\def\cP{\mathcal{P}}
\def\cQ{\mathcal{Q}}
\def\cJ{\mathcal{J}}
\def\cR{\mathcal{R}}

\def\EE{\mathbb{E}}

\newcommand{\Part}{\cP}
\newcommand{\Qart}{\cQ}

\newtheorem{theorem}{Theorem}
\newtheorem{lemma}[theorem] {Lemma}    
\newtheorem{corollary}[theorem] {Corollary}    
   
\newtheorem{proposition}[theorem] {Proposition}

\newtheorem{definition}[theorem] {Definition}  
  
\newtheorem{claim}[theorem] {Claim}  

\theoremstyle{remark}

\newcommand{\oldqed}{}
\newenvironment{claimproof}[1][Proof of Claim]{
  \renewcommand{\oldqed}{\qedsymbol}
  \renewcommand{\qedsymbol}{\endofFact}
  \begin{proof}[#1]
}{
  \end{proof}
  \renewcommand{\qedsymbol}{\oldqed}
} 

\newcommand{\sublem}[1] {{\mbox{\tiny{L\ref{#1}}}}}
\newcommand{\subcor}[1] {{\mbox{\tiny{C\ref{#1}}}}}

\newcommand{\NATS}{\mathbb{N}}

\newcommand{\s}[1]{\left\lvert #1 \right\rvert}

\newcommand{\Pres}{P_\mathrm{res}}
\newcommand{\Palm}{P_\mathrm{almost}}
\newcommand{\Pcover}{P_\mathrm{cover}}
\newcommand{\cross}{\mathrm{Cross}}
\newcommand{\nures}{\nu_\mathrm{res}}

\newcommand{\EMAIL}[1]{  \textit{E\nobreakdash-mail}: \texttt{#1} }

\newcommand{\tpl}[1]{\mathbf{#1}}

\DeclareMathOperator{\Bin}{Bin}

\def\vp#1{}
\renewcommand{\vp}[1]{\footnote{\textcolor{green!40!black}{\textbf{VP: }#1}}} 

\title{Resilience for tight Hamiltonicity}

  \author[P. Allen]{Peter Allen*}

  \author[O. Parczyk]{Olaf Parczyk*}
 
  \author[V. Pfenninger]{Vincent Pfenninger\dag}
 
 \thanks{
 	*
 	Department of Mathematics, London School of Economics, Houghton Street,
 	London, WC2A 2AE, U.\ K.\ \\
 	\EMAIL{p.d.allen@lse.ac.uk, o.parczyk@lse.ac.uk}
 }
 
  \thanks{
    \dag School of Mathematics, University of Birmingham, Edgbaston, Birmingham, B15 2TT, U.\ K.\ \\
   \EMAIL{vxp881@bham.ac.uk}
  }

  \thanks{
  PA was supported by EPSRC, EP/P032125/1.
  OP was supported by the DFG (Grant PA 3513/1-1).
  We would like to thank the Heilbronn Institute for Mathematical Research, and EPSRC (grant number EP/P032125/1) for supporting the workshop `Structure and Randomness in Hypergraphs' where this work was started.
  }

\date{\today}

\begin{document}
	
\begin{abstract}
 We prove that random hypergraphs are asymptotically almost surely resiliently Hamiltonian. Specifically, for any $\gamma>0$ and $k\ge3$, we show that asymptotically almost surely, every subgraph of the binomial random $k$-uniform hypergraph $G^{(k)}\big(n,n^{\gamma-1}\big)$ in which all $(k-1)$-sets are contained in at least $\big(\tfrac12+2\gamma\big)pn$ edges has a tight Hamilton cycle. This is a cyclic ordering of the $n$ vertices such that each consecutive $k$ vertices forms an edge.
\end{abstract}
\maketitle

\section{Introduction}

The study of Hamilton cycles in graphs is one of the oldest topics in graph theory. In extremal graph theory, Dirac~\cite{Dirac} in 1952 proved the sharp result that an $n$-vertex graph with minimum degree at least $\tfrac{n}{2}$ contains a Hamilton cycle. In random graph theory, P\'osa~\cite{Posa} and Korshunov~\cite{Kor76,Kor77} independently showed in the 1970s that Hamilton cycles first appear in the random graph $G(n,p)$ --- that is, the $n$-vertex graph where edges are present independently with probability $p$ --- at a threshold $p=\Theta\big(\tfrac{\log n}{n}\big)$. Koml\'os and Szemer\'edi~\cite{KomSzem} showed that the sharp threshold for Hamiltonicity coincides with that for minimum degree $2$, and Bollob\'as~\cite{Boll} strengthened this by showing a hitting time version: if edges are added one by one, the edge which causes minimum degree $2$ will asymptotically almost surely\footnote{Asymptotically almost surely (a.a.s.)~is with probability tending to $1$ as $n$ tends to infinity.} also cause Hamiltonicity.

Combining these areas, Sudakov and Vu~\cite{SudVu} introduced the term \emph{resilience} (though the same concept appears earlier in work of Alon, Capalbo, Kohayakawa, R\"odl, Ruci\'nski and Szemer\'edi~\cite{ACKRRS}). They proved that for each $\gamma>0$, the random graph $\Gamma=G(n,p)$ is a.a.s.~$\big(\tfrac12+\gamma\big)$-resiliently Hamiltonian whenever $p\gg n^{-1}\log^4n$; that is, every subgraph of $\Gamma$ with minimum degree at least $\big(\tfrac12+\gamma\big)pn$ has a Hamilton cycle. This result is sharp in the minimum degree, for the same reason as Dirac's theorem, but the probability can be improved. This was done over a succession of papers: Lee and Sudakov~\cite{LeeSud} showed that $p$ can be reduced to the threshold $\Omega(n^{-1}\log n)$, and 	very recently Montgomery~\cite{Montgomery} showed the hitting time version of this result (for which one needs to be a little more careful with edge deletion: it is permitted to delete only a $\big(\tfrac12-\gamma\big)$-fraction of the edges at any given vertex).

\medskip

Hamilton cycles in hypergraphs have only much more recently been attacked. There are several natural notions of paths and cycles in hypergraphs: the one that will concern us here is that of tight paths and cycles in $k$-uniform hypergraphs. That is, we work with hypergraphs in which all edges have uniformity $k$. We say that a given linear ordering of some vertices is a \emph{tight path} if each consecutive $k$-set of vertices forms an edge; a given cyclic ordering of some vertices with the same condition forms a tight cycle. The $k=2$ case of this definition reduces to the usual paths and cycles in graphs. For brevity, in what follows we write \emph{$k$-graph} for $k$-uniform hypergraph.

In terms of extremal results, there are again several reasonable questions --- one should place some form of `minimum degree' condition for tight Hamilton cycles, but this can take the form of insisting that every $j$-set of vertices is in sufficiently many edges, where one can choose $j$ between $1$ and $k-1$. This leads to several significantly different problems (and even more if one considers other notions of cycle). We refer the reader to the comprehensive survey of K\"uhn and Osthus~\cite{KuhOst} for details, and focus on the version of minimum degree we want to work with. This is the case $j=k-1$, sometimes called \emph{codegree}. Here, the Hamiltonicity problem is resolved. R\"odl, Ruci\'nski and Szemer\'edi~\cite{RRS1,RRS2}, first for $3$-uniform and then for general uniformity, showed that if $n$ is sufficiently large, any $n$-vertex $k$-graph with minimum codegree at least $\big(\tfrac12+\gamma\big)n$ (i.e.\ every $(k-1)$-set is in at least that many edges) contains a tight Hamilton cycle. For $3$-graphs, they~\cite{RRS3} were also able to give the exact result for sufficiently large $n$ (finding exactly what should replace the error term $\gamma n$).

In random hypergraphs, Dudek and Frieze~\cite{DudFriLoose,DudFriTight} found for several different notions of `cycle' the threshold for Hamiltonicity in the binomial random hypergraph $G^{(k)}(n,p)$, that is the $n$-vertex $k$-graph in which $k$-sets are edges independently with probability $p$. In particular, in~\cite{DudFriTight} they showed by the second moment method that for $k=3$ the threshold is $\omega\big(n^{-1}\big)$, and for $k\ge4$ the sharp threshold is at $en^{-1}$. Narayanan and Schacht~\cite{NarSch} strengthened these results, in particular showing that $en^{-1}$ is also the sharp threshold for $k=3$.

Combining these (and answering a question of Frieze~\cite{FriezeSurvey}), we prove the following corresponding codegree resilience statement.

\begin{theorem}\label{thm:main}
 Given any $\gamma>0$ and $k\ge3$, if $p\ge n^{-1+\gamma}$, we show that $\Gamma=G^{(k)}(n,p)$ a.a.s.\ satisfies the following. Let $G$ be any $n$-vertex subgraph of $\Gamma$ such that $\delta_{k-1}(G)\ge\big(\tfrac12+2\gamma\big) pn$. Then $G$ contains a tight Hamilton cycle.
\end{theorem}

Observe that this theorem is sharp in the minimum degree requirement, but it is presumably not sharp in the probability.
More precisely, when $p = \Omega (\log n/n)$ then a.a.s.\  in~$\Gamma$ there is an $n$-vertex subgraph $G$ such that $\delta_{k-1}(G) \ge (1/2-\gamma) pn$ and $G$ does not contain a tight Hamilton cycle.
When $p=o(\log n/n)$, there a.a.s.\ are $(k-1)$-tuples in $\Gamma$ that are not contained in any edges and, therefore, no $G$ as required by the theorem exists.
For this regime the resilience condition needs to be adjusted, perhaps as explained above for the hitting time results in graphs from~\cite{Montgomery}.
We certainly need $p\ge 2en^{-1}$ for any statement of this kind to be true, otherwise randomly deleting half of the edges from $\Gamma$ would a.a.s.\ destroy the tight Hamiltonicity.

This is the first resilience statement for tight Hamilton cycles in sparse random hypergraphs to the best of our knowledge; however for Berge cycles, Clemens, Ehrenm\"uller and Person~\cite{CEP16} proved a resilience statement which is both tight in the minimum degree and has only a polylogarithmic gap in the probability.
For perfect matchings it was shown by Ferber and Hirschfeld~\cite{Ferber_resilience} that the same codegree resilience as in Theorem~\ref{thm:main} holds with $p = \Omega (\log n/n)$, which is significantly above the threshold for the appearance of perfect matchings, but optimal for the same reasons as discussed above.
More generally, Ferber and Kwan~\cite{FerberKwan} studied the transference of results for perfect matchings in dense hypergraphs into resilience statements in random hypergraphs.

It would be interesting to investigate this transference for other types of Hamilton cycles and other degree conditions.
For example, in the case of $3$-graphs Reiher, R{\"o}dl, Ruci{\'n}ski, Schacht, and Szemer{\'e}di~\cite{reiher2019minimum} show that any $n$-vertex $3$-graph with minimum vertex degree $(\tfrac 59 + \gamma) \binom{n}{2}$ contains a tight Hamilton cycle.
Can this be extended to a resilience statement in random $3$-graphs?
More precisely, can the condition $\delta_2(G) \ge (\tfrac 12 + \gamma) pn$ in Theorem~\ref{thm:main} for $k=3$ be replaced by $\delta_1(G) \ge (\tfrac 59 + \gamma) p\binom{n}{2}$?
The bound on the minimum degree would again be sharp.

\subsection{Ideas of the proof, and outline of the paper}

Our proof strategy for Theorem~\ref{thm:main} uses the \emph{reservoir method}, which was previously used in~\cite{TightCycle} and~\cite{BetterTightCycle}, in a similar way to the use we will make here, to give polynomial-time algorithms that find tight Hamilton cycles in $\Gamma$ itself for broadly similar values of $p$. Very briefly, the reservoir method is as follows.

In a first step, we identify a \emph{reservoir set} $R$, which contains a small (but bounded away from~$0$) fraction of the vertices of $G$. We construct a \emph{reservoir path} $\Pres$, which is a tight path that contains all the vertices of $R$ and in addition for any subset $R'$ of $R$, there is a tight path with the same ends as $\Pres$ whose vertex set is $V(\Pres)\setminus R'$.

In a second step, we extend $\Pres$ to an almost-spanning tight path $\Palm$. In the final step we re-use some vertices of $R$ to extend $\Palm$ further to a structure which is `almost' a tight Hamilton cycle, except that some vertices $R'$ of $R$ are used twice. Finally we apply the reservoir property of $\Pres$ to obtain the desired tight Hamilton cycle.

\bigskip

In~\cite{BetterTightCycle}, in the random hypergraph, there are two main tools needed to put this plan into action. First, for any given ordered $(k-1)$-tuple $\tpl{x}$ of vertices and set $S$ of `unused' vertices which is not too small, there will be lots of ways to start a tight path from $\tpl{x}$ and continuing with vertices of $S$. Second, for any given pair of ordered $(k-1)$-tuples $\tpl{x}$ and $\tpl{y}$, and any given set $S$ of unused vertices which is not too small, it is possible to find a tight path from $\tpl{x}$ to $\tpl{y}$ in $S$.\footnote{To be accurate, these statements will be true for all the sets $S$ that actually appear in the proof, by a careful revealing argument; they are not true for every $S$.}

Neither of these statements is true in the resilience setting. Instead, we make use of hypergraph regularity to help us. In the following section~\ref{sec:tools} we state our main tools, and prove some of them. We first introduce spike paths, which we need to construct our reservoir structure (much as in~\cite{BetterTightCycle}). 

We give the notational setup for hypergraph regularity, and state a sparse, strengthened version of the Strong Hypergraph Regularity Lemma, Lemma~\ref{lem:ssshrl}, which may be of independent interest. We show that the output of this Regularity Lemma is, for $k$-graphs with our minimum degree condition, a structure which is robustly \emph{tightly linked}: this is a version of connectivity appropriate for tight paths.

We show that the random hypergraph has certain nice properties: in particular, once one removes a small fraction of $(k-1)$-tuples, for any remaining $(k-1)$-tuple $\tpl{x}$ and set $S$ which is reasonably small (it cannot contain more than $n/2$ vertices) there are lots of ways to start constructing a tight path from $\tpl{x}$ \emph{avoiding} $S$ (Lemma~\ref{lem:goodtuples}), and if we do so for a sufficiently large (but independent of $n$) number of steps, we reach a positive fraction of all $(k-1)$-tuples. This statement (Lemma~\ref{lem:goodexpansion}) is one of the key points in our proof: most of the time, we can expand in a few steps from any given $(k-1)$-tuple to a positive density of $(k-1)$-tuples (and a similar statement holds for spike paths).

Using Lemma~\ref{lem:goodexpansion}, regularity and tight linkedness, we can prove a Connecting Lemma (Lemma~\ref{lem:connecting}) which states that for any reasonably small set $S$ and most pairs $\tpl{x}$ and $\tpl{y}$ of $(k-1)$-tuples, there is a short tight path from $\tpl{x}$ to $\tpl{y}$ which avoids $S$.

These tools are enough to prove a Reservoir Lemma~\ref{lem:respath}, which (much as in~\cite{BetterTightCycle}) constructs $\Pres$ mentioned above. However again at this point difficulties arise. In the random hypergraph of~\cite{BetterTightCycle}, the vertices outside $\Pres$ have no particular structure. In our setting, $\Pres$ interacts in some rather unpredictable way with the existing structure provided by the Regularity Lemma. To deal with this, we use LP-duality in Lemma~\ref{lem:fractional} to find a fractional matching which will tell us how many vertices we should use in each part of our regular partition in order to obtain $\Palm$. We also at this point run into the difficulty that we can only guarantee expansion from the minimum degree when we are avoiding less than $n/2$ vertices, yet $\Palm$ is supposed to cover almost all of the vertices; it is here that we need the `strengthened' property of our Regularity Lemma.

In Section~\ref{sec:proofmain}, we give the proof of Theorem~\ref{thm:main}, assuming the so far unproved lemmas.

In Section~\ref{sec:proofconnecting} we prove the Connection Lemma, Lemma~\ref{lem:connecting}, and also Lemma~\ref{lem:connectpartition} which shows how we can use the strengthened regularity to continue extending a tight path even when most vertices have been used.

In Section~\ref{sec:proofreservoir} we prove the Reservoir Lemma, Lemma~\ref{lem:respath}.

Finally, we defer the proof of our Regularity Lemma, Lemma~\ref{lem:ssshrl}, together with various more-or-less standard facts about dense hypergraph regularity, to Appendix~\ref{sec:proofreg}. Although some of these results are new and Lemma~\ref{lem:ssshrl} may well be useful in future, the ideas needed to prove them are not new.

\section{Tools}\label{sec:tools}

\subsection{Spike paths}

To build our reservoir structure we need spike paths, which are the following variant of a tight path that changes orientation every $(k-1)$ steps. We will only consider spike paths with a number of vertices divisible by $k-1$.

\begin{definition}[Spike path]
	In an $k$-uniform hypergraph, a spike path with $t$ spikes consists of a sequence of $t$ pairwise disjoint $(k-1)$-tuples $\tpl{a}_1,\dots,\tpl{a}_t$, where $\tpl{a}_i=(a_{i,1},\dots,a_{i,k-1})$ for all $i$, with the property, that the edges $\{ a_{i,k-j},\dots,a_{i,1},a_{i+1,1},\dots,a_{i+1,j} \}$ are present for all $i=1,\dots,t-1$ and $j=1,\dots,k-1$.
	We call $\tpl{a}_i$ the $i$th spike.
\end{definition}

\subsection{Notation}

A \emph{$k$-complex} is a hypergraph $H$ all of whose edges have size at most $k$, which is down-closed, i.e.\ if $e\in E(H)$ and $e'\subseteq e$ then $e'\in E(H)$. The \emph{layers} of a $k$-complex are, for each $0\le i\le k$, the $i$-uniform hypergraph $H^{(i)}$ on the same vertex set, where $E\big(H^{(i)}\big)=\{e\in E(H):|e|=i\}$.

A \emph{$k$-multicomplex} is, informally, a $k$-complex in which multiple edges of any size between $2$ and $k$ are permitted, together with a map \emph{boundary} $\partial$ identifying the $(i-1)$-edges which support a given $i$-edge. Formally, a $k$-multicomplex $H$ consists of a \emph{vertex set} $V(H)$, together with a set of \emph{edges} $E(H)$, a \emph{vertices} map $\mathrm{vertices}:E\to\mathcal{P}(V)$ such that $\mathrm{vertices}(e)$ is a set of size between $0$ and $k$ for each $e\in E(H)$, and a \emph{boundary} map $\partial:E\setminus\{\emptyset\}\to \mathcal{P}(E)$ such that $\partial e$ contains exactly one edge whose vertices are $\mathrm{vertices}(e)\setminus \{v\}$ for each $v\in\mathrm{vertices}(e)$, and no other edges. We further insist on the following consistency condition: if $2\le i\le k$, and $S$ is a set of $i$ edges each with $i-1$ vertices, such that $\big|\bigcup_{f\in S}\partial f\big|>\binom{i}{i-2}$, then there are no edges $e\in H$ such that $\partial e=S$. We say that the \emph{uniformity} of an edge $e$ is $|\mathrm{vertices}(e)|$, and we may write that $e$ is an edge \emph{on the set $\mathrm{vertices}(e)$}, or that $e$ is \emph{a $|\mathrm{vertices}(e)|$-edge}. We will also say, given a set $S$ consisting of $i$ edges of uniformity $(i-1)$, that $e$ is \emph{supported} on $S$ if $\partial e=S$.

Note that the boundary of a $1$-edge is necessarily $\{\emptyset\}$, and that `down-closure' is forced by the condition of the boundary map. To better understand the consistency condition, consider the following. If $e$ is an edge of $H$ with at least two vertices, and $x$ and $y$ are distinct vertices of $e$, let $e_x$ and $e_y$ be the edges in $\partial e$ whose vertices do not contain respectively $x$ and $y$. There is an edge $e_{xy}$ in $\partial e_x$, and an edge $e_{yx}$ in $\partial e_y$, on  $\mathrm{vertices}(e)\setminus\{x,y\}$. The consistency condition is equivalent to insisting that for any $e$, $x$ and $y$ we have $e_{yx}=e_{xy}$.

Observe that a $k$-complex is a $k$-multicomplex, where the vertices of each edge are simply its members as a set, and the boundary map is the usual boundary $\partial e=\big\{e\setminus \{v\}:v\in e\big\}$ (which is in this case the only possible boundary map for the given vertices map). However in general, for a given ground set, edge set and vertices map, there may be several different boundary maps which fit the definition of $k$-multicomplex; these return different multicomplexes. The idea here is that we will need to think of a given edge (say with vertices $\{1,2,3\}$) as containing specific edges with vertices $\{1,2\}$, $\{1,3\}$ and $\{2,3\}$, and the map $\partial$ tells us which edges these are. We should stress that it is possible to have a $k$-multicomplex in which there are two different edges which have the same boundary and vertices, and indeed the multicomplexes we consider in this paper will have this property for edges of uniformity two and above (though for us a $1$-edge will always be the unique $1$-edge on a given vertex).

Given a vector $\mathbf{d}=(d_2,\dots,d_k)$ where $1/d_i\in\NATS$ for each $i$, we say that a $k$-multicomplex $H$ is \emph{$\mathbf{d}$-equitable} if there is exactly one $1$-edge on each vertex, and furthermore for any $2\le i\le k$ and $i$-set $X$ of vertices the following holds. Whenever $S$ is a collection of $i$ edges of uniformity $i-1$ in $H$, one on the vertices $X\setminus\{x\}$ for each $x\in X$, if the union $\bigcup_{f\in S}\partial f$ has exactly $\binom{i}{i-2}$ edges then the number of $i$-edges in $H$ supported on $S$ is exactly $1/d_i$. We refer to $\mathbf{d}$ as the \emph{density vector} of the multicomplex.

Finally, we need a notion of connectedness for multicomplexes.

\begin{definition}[tight link, tightly linked]

Given a $k$-multicomplex $\cR$, and two $(k-1)$-edges $u,v$ of $\cR$, let $\mathbf{u}$ be $u$ together with an ordering $(u_1,\dots,u_{k-1})$ of its vertices, and similarly let $\mathbf{v}$ be $v$ together with an ordering $(v_1,\dots,v_{k-1})$ of its vertices. A \emph{tight link from $\mathbf{u}$ to $\mathbf{v}$ in $\cR$} is the following collection of (not necessarily distinct) vertices and edges of $\cR$.

For each $1\le j\le k-1$, there is a vertex $w_j$. There are $k$-edges $e_{1,u}$ and $e_{1,v}$ of $\cR$, where $e_{1,u}$ is on vertices $\{u_1,\dots,u_{k-1},w_1\}$ and $u\in\partial e_{1,u}$, and $e_{1,v}$ is on vertices $\{v_1,\dots,v_{k-1},w_1\}$ and $v\in\partial e_{1,v}$. For each $2\le j\le k-1$, there are $k$-edges $e_{j,u}$ and $e_{j,v}$ of $\cR$, where $e_{j,u}$ is on vertices $\{u_j,\dots,u_{k-1},w_1,\dots,w_j\}$ and $\partial e_{j-1,u}\cap\partial e_{j,u}\neq\emptyset$, and $e_{j,v}$ is on vertices $\{v_j,\dots,v_{k-1},w_1,\dots,w_j\}$ and $\partial e_{j-1,v}\cap\partial e_{j,v}\neq\emptyset$. Finally $\partial e_{k-1,u}\cap\partial e_{k-1,v}\neq\emptyset$. 
 
 We say that a $k$-multicomplex $\cR$ is \emph{tightly linked} if for any two $(k-1)$-edges in $\cR$, and any orderings of their vertices, $\mathbf{u}$ and $\mathbf{v}$, there is a tight link from $\mathbf{u}$ to $\mathbf{v}$ in $\cR$.
\end{definition}

The precise sequence of vertices and edges is not critical (it is simply a particular structure we can easily construct). However it will be convenient to note that the $k$-edges of a tight link are in fact a spike path with three spikes. Note that there is $\ell\in\mathbb{N}$ and a permutation $\rho$ on $[k-1]$ such that for any $\tpl{u}$ and $\tpl{v}$, if there is a tight link from $\tpl{u}$ to $\tpl{v}$ then there is a homomorphism from the $\ell$-vertex tight path to $\cR$, using only the $k$-edges of the tight link, where the first $k-1$ vertices of the tight path are sent to $\tpl{u}$ in order and the last $k-1$ vertices to the vertices of $\tpl{v}$ in the order $\rho$.

\subsection{Sparse hypergraph regularity}

We need a strengthened version of the Strong Hypergraph Regularity Lemma for sparse hypergraphs. The Strong Hypergraph Regularity Lemma was first proved by R\"odl and Skokan~\cite{RSk} and Gowers~\cite{Gowers}; we use a version due to R\"odl and Schacht~\cite{RSch}, from which we deduce a strengthened version by a standard method. We then use a weak sparse regularity lemma of Conlon, Fox and Zhao~\cite{CFZ} to transfer this strengthened version to a sparse version, following~\cite{ADS}.

In order to state our regularity lemma, we need quite a few definitions. These are either standard definitions for the dense ($p=1$) case, or the natural sparse versions of the same, as taken from~\cite{ABCM}.

Let $\Part$ partition a vertex set $V$ into parts $V_1, \dots, V_s$. We say
that a subset $S \subseteq V$ is \emph{$\Part$-partite} if $|S \cap V_i| \leq 1$ for
every $i \in [s]$ and the \emph{index} of a $\Part$-partite set $S \subseteq V$ is $i(S) := \{i \in [s] : |S \cap V_i| = 1\}$.
For any $A \subseteq [s]$ we write $V_A$ for $\bigcup_{i \in A} V_i$.
Similarly, we say that a hypergraph $H$ is \emph{$\Part$-partite} if
all of its edges are $\Part$-partite. In this case we refer to the parts of
$\Part$ as the \emph{vertex classes} of $H$.
Moreover, we say that 
a hypergraph $H$ is \emph{$s$-partite} if there is some partition $\Part$ of $V(H)$ into $s$ parts for which $H$ is $\Part$-partite.

Let $i \geq 2$, let $H_i$ be any $i$-partite $i$-graph, and let $H_{i-1}$ be any $i$-partite
$(i-1)$-graph, on a common vertex set $V$ partitioned into $i$ common vertex classes. We denote by
$K_i(H_{i-1})$ the $i$-partite $i$-graph on $V$ whose edges are all
$i$-sets in $V$ which are supported on $H_{i-1}$ (i.e.~induce a copy of the complete $(i-1)$-graph~$K_i^{i-1}$
on~$i$ vertices in~$H_{i-1}$). Given $p\in(0,1]$, the \emph{$p$-density of $H_i$ with respect to
$H_{i-1}$} is then defined to be \[ d_p(H_i|H_{i-1}):= \frac{|K_i(H_{i-1})\cap
H_i|}{p|K_i(H_{i-1})|}\]
 if
$|K_i(H_{i-1})|>0$. For convenience we take
$d_p(H_i|H_{i-1}):=0$ if $|K_i(H_{i-1})| = 0$, and we assume $H_1$ is the complete $1$-graph on $V$, whose edge set is $V$. So $d_p(H_i|H_{i-1})$ is the proportion of copies of
$K^{i-1}_i$ in $H_{i-1}$ which are also edges of $H_i$, scaled by $p$. When $H_{i-1}$ is clear from the
context,
we simply refer to $d_p(H_i | H_{i-1})$ as the \emph{relative $p$-density of $H_i$}.
We say that $H_i$ is \emph{$(d_i,\eps,p)$-regular with respect
to~$H_{i-1}$} if we have $d_p(H_i|H'_{i-1}) = d_i \pm \eps $ for every subgraph $H'_{i-1}$ of $H_{i-1}$ such that $|K_i(H'_{i-1})| > \eps |K_i(H_{i-1})|$.  Given an $i$-graph $G$
whose vertex set contains that of $H_{i-1}$, we say that $G$ is
\emph{$(d_i,\eps,p)$-regular with respect to~$H_{i-1}$} if the $i$-partite subgraph of
$G$ induced by the vertex classes of $H_{i-1}$ is $(d_i,\eps,p)$-regular with respect to
$H_{i-1}$. Finally, we say $G$ is \emph{$(\eps,p)$-regular with respect to $H_{i-1}$} if there exists $d_i$ such that $G$ is $(d_i,\eps,p)$-regular with respect to $H_{i-1}$. Similarly as before, when $H_{i-1}$ is clear from the context, we refer to the relative density of this $i$-partite subgraph of $G$ with respect to $H_{i-1}$ as the \emph{relative $p$-density of~$G$}.

Now let $H$ be an $s$-partite $k$-complex on
vertex classes $V_1, \dots, V_s$, where $s \geq k \geq 3$. Recall that,
since~$H$ is a complex, if $e \in H$ and $e' \subseteq e$ then $e' \in H$. So if $e \in
H^{(i)}$ for some $2 \leq i \leq k$, then the vertices of $e$ induce a copy of
$K^{i-1}_i$ in $H^{(i-1)}$. We say that~$H$ is \emph{$(d_k,\dots,d_2,\eps_k,\eps,p)$-regular} if
\begin{enumerate}[label=\abc]
 \item for any $2 \leq i \leq k-1$ and any $A \in
  \binom{[s]}{i}$, the induced subgraph $H^{(i)}[V_A]$ is
  $(d_i,\eps,1)$-regular with respect to~$H^{(i-1)}[V_A]$, and
 \item for any $A \in \binom{[s]}{k}$, the induced subgraph 
  $H^{(k)}[V_A]$ is $(d_k, \eps_k,p)$-regular with respect to~$H^{(k-1)}[V_A]$.
\end{enumerate}
So each constant $d_i$ approximates the relative density of each subgraph
$H^{(i)}[V_A]$ for $A \in \binom{[s]}{i}$. 
For a $(k-1)$-tuple $\mathbf{d} = (d_k, \dots, d_2)$ we write
$(\mathbf{d},\eps_k,\eps,p)$-regular to mean
$(d_k,\dots,d_2,\eps_k,\eps,p)$-regular.

The definition of a $(\mathbf{d},\eps_k,\eps,p)$-regular complex $H$ is the `right' generalisation of an $\eps$-regular pair $(X,Y)$ in dense graphs to sparse hypergraphs. The Szemer\'edi Regularity Lemma states that there is a partition of the vertices of any graph into boundedly many parts such that most pairs of parts are regular; now our aim is to define a generalisation of `partition' in order to say that we can partition any  $k$-uniform hypergraph $G$ such that most $k$-sets lie in regular complexes. As one can guess from the phrasing, the $k$-layer of each complex will consist of (all) edges of $G$ supported by the complex. The lower layers will be in the `partition', and we now set up the notation to define this.

Fix $k \geq 3$, and let $\Part$ partition a vertex set $V$ into parts $V_1,
\dots, V_t$.
For any $A \subseteq [t]$, we denote by $\cross_A(\Part)$ the
collection of $\Part$-partite subsets $S \subseteq V$ of index $i(S) = A$.
Likewise, we denote by $\cross_j(\Part)$ the union of $\cross_A$ for each $A \in
\binom{[t]}{j}$, so $\cross_j(\Part)$ contains all $\Part$-partite subsets $S
\subseteq V$ of size $j$. When $\Part$ is clear from the context, we write simply $\cross_A$ and $\cross_j$.
For each $2 \leq j \leq k-1$ and $A \in \binom{[t]}{j}$ let $\Part_A$ be a
partition of $\cross_A$. For consistency of notation we also define the trivial partitions $\Part_{\{s\}} := \{V_s\}$ for $s \in [t]$ and $\Part_\emptyset := \{\emptyset\}$. Let $\Part^*$ consist of the partitions $\Part_A$ for each $A\in\binom{[t]}{j}$ and each $0 \le j\le k-1$. We say that $\Part^*$ is a \emph{$(k-1)$-family of partitions
on $V$} if whenever $S, T \in \cross_A$ lie in the same part of $\Part_A$ and $B
\subseteq A$, then $S \cap \bigcup_{j\in B} V_j$ and $T \cap
\bigcup_{j\in B} V_j$ lie in the same part of $\Part_B$. In other words, given $A \in \binom{[t]}{j}$, if we specify one part of each $\Part_B$ with $B \in \binom{A}{j-1}$, then we obtain a subset of $\cross_A$ consisting of all $S\in\cross_A$ whose $(j-1)$-subsets are in the specified parts. We say that this subset of $\cross_A$ is the subset \emph{supported} by the specified parts of $\Part_B$. In general, we say that a $j$-set $e$ is \emph{supported} by a collection $S$, with $|S|=j$, of $(j-1)$-graphs if exactly one $(j-1)$-subset of $e$ is in each member of $S$, and we say a set of $j$-edges $E$ is supported by $S$ if each edge of $E$ is supported by $S$.

Thus the partitions $\Part_B$ give a natural partition of $\cross_A$, and we are saying that $\Part_A$ must refine it.

We refer to the parts
of each member of $\Part^*$ as \emph{cells}. Also, we refer to $\Part$ as the \emph{ground partition} of $\Part^*$, and the parts
of $\Part$ (i.e.~the vertex classes $V_i$) as the \emph{clusters} of $\Part^*$. For
each $0 \leq j \leq k-1$ let $\Part^{(j)}$ denote the partition of $\cross_j$
formed by the parts (which we call \emph{$j$-cells}) of each of the partitions 
$\Part_A$ with $A \in \binom{[t]}{j}$ (so in particular $\Part^{(1)} = \Part$).

Observe that a $(k-1)$-family of partitions $\Part^*$ naturally form the edges of a $k$-multicomplex, whose vertex set is the (set of parts of the) ground partition, whose edges of uniformity $j\le k-1$ are the $j$-cells, with the vertices map identifying the $j$ parts of the ground partition which contain a given $j$-cell, and where the boundary operator $\partial e$ identifies the $(|e|-1)$-cells supporting $e$. So far we have described a $(k-1)$-multicomplex; we extend this to a $k$-complex by adding, for each set $S$ of $k$ edges of uniformity $k-1$ which can be a boundary (i.e.\ which is such that $\big|\bigcup_{f\in S}\partial f\big|=\binom{k}{k-2}$ ) one edge of uniformity $k$ whose boundary is $S$. When we refer to the \emph{multicomplex of the family of partitions $\cP^*$} we mean this multicomplex. Note that we have defined the word `support' both in terms of multicomplexes and in terms of a family of partitions: but these definitions are consistent, i.e.\ that a given $j$-cell is supported by some $(j-1)$-cells means the same thing whether one reads `support' in terms of the family of partitions or its multicomplex.

For any $0 \leq j \leq k-1$, any $A \in
\binom{[t]}{j}$ and any $Q' \in \cross_A$, let $C_{Q'}$ denote the cell of
$\Part_{A}$ which contains $Q'$. 
Then the fact that $\Part^*$ is a family of partitions implies that 
for any $Q \in \cross_k$ the union 
$\cJ(Q) := \bigcup_{Q' \subsetneq Q} C_{Q'}$ 
of cells containing subsets of $Q$ is a
$k$-partite $(k-1)$-complex.
We say that the $(k-1)$-family of partitions $\Part^*$ is \emph{$(t_0,
t_1, \eps)$-equitable} if 
\begin{enumerate}[label=\abc]
  \item\label{equitfam:a} $\Part$ partitions~$V$ into~$t$ clusters of equal size, where
  $t_0 \leq t \leq t_1$,
  \item\label{equitfam:b} for each $2 \leq j \leq k-1$, $\Part^{(j)}$ partitions
  $\cross_j$ into at most $t_1$ cells,
  \item\label{equitfam:c} there exists ${\mathbf{d}}=( d_{k-1},
  \dots, d_2)$ such that for each $2 \leq j \leq k-1$ we have $d_j \geq
  1/t_1$ and $1/d_j \in \NATS$, and for every $Q \in \cross_k$ the $k$-partite
  $(k-1)$-complex 
$\cJ(Q)$ 
is $({\bf d},\eps,\eps,1)$-regular.
\end{enumerate}
Note that conditions~\ref{equitfam:a} and~\ref{equitfam:c} imply that $\cJ(Q)$ is a $(1, t_1,
\eps)$-equitable $(k-1)$-complex (with the same density vector $\mathbf{d}$) for any $Q \in \cross_k$.

Next, for any $\Part$-partite set $Q$ with $2 \leq |Q| \leq k$, define $\hat{P}(Q; \Part^*)$
to be the $|Q|$-partite $(|Q|-1)$-graph on $V_{i(Q)}$ with edge set $\bigcup_{Q' \in
\binom{Q}{|Q|-1}} C_{Q'}$. We refer to $\hat{P}(Q; \Part^*)$ as a \emph{$|Q|$-polyad}; when the family of partitions $\Part^*$ is clear from the context, we write simply $\hat{P}(Q)$ rather than $\hat{P}(Q; \Part^*)$. 
Note that the condition for $\Part^*$ to be a $(k-1)$-family of partitions can then be rephrased as saying that if $2 \leq |Q| \leq k-1$ then the cell $C_Q$ is supported on $\hat{P}(Q)$, and in the multicomplex corresponding to $\cP^*$ we have edges corresponding to the cells of each uniformity from $1$ to $k-1$ inclusive, together with edges corresponding to the $k$-polyads supported by $\cP^*$. 
As shown in~\cite[Claim~32]{ABCM}, if $\Part^*$ is $(t_0, t_1, \eps)$-equitable for sufficiently small $\eps$, then for any $2 \leq j \leq k-1$ and any $Q \in \cross_j$ the number of $j$-cells of $\Part^*$ supported on $\hat{P}(Q)$ is precisely equal to $1/d_j$. More specifically, if $\big(d_j^{-1}-1\big)(d_j+\eps)<1$, and $\big(d_j^{-1}+1\big)(d_j-\eps)>1$, then by definition necessarily there are exactly $d_j^{-1}$ cells supported; it suffices to choose $\eps\ll d_j^{2}$
to ensure these two inequalities. In other words, the multicomplex corresponding to $\Part^*$ is $\mathbf{d}$-equitable.

Now let~$G$ be a $k$-graph on~$V$, and let $\Part^*$ be a $(k-1)$-family of
partitions on~$V$. Let $Q\in\cross_k$, so the polyad $\hat{P}(Q)$ is a $k$-partite $(k-1)$-graph.
We say that~$G$
is \emph{$(\eps_k,p)$-regular with respect to~$\Part^*$} if there are at most
$\eps_k \binom{|V|}{k}$ sets $Q \in \cross_k$ for which $G$ is not
$(\eps_k,p)$-regular with respect to the polyad $\hat{P}(Q)$. That is,
at most an $\eps_k$-proportion of subsets of $V$ of size $k$ yield polyads with respect to which $G$ is not regular (though some subsets of $V$ of size $k$ do
not yield any polyad due to not being members of $\cross_k$).

At this point we have the setup to state the Strong Hypergraph Regularity Lemma, which says that for any $k$-uniform hypergraph $G$ there is a $(k-1)$-family of partitions $\Part^*$, which is $(t_0,t_1,\eps)$-equitable for some $t_1$ independent of $|V(G)|$, such that $G$ is regular with respect to $\Part^*$. However for this paper we need a stronger version, which is not standard (the dense graph version, called the Strengthened Regularity Lemma, is due to Alon, Fischer, Krivelevich and Szegedy~\cite{AFKS}, and it is folklore that the hypergraph version we now state should exist). To that end, given two families of partitions $\Part^*$ and $\Qart^*$ on the same vertex set, we say that $\Part^*$ \emph{refines} $\Qart^*$ if every cell of $\Part^*$ is a subset of some cell of $\Qart^*$.

\begin{definition}
	Given a $k$-uniform hypergraph $G$, we call a pair of families of partitions $(\cP^*_c,\cP^*_f)$ on $V(G)$ a $(t_0,t_1,t_2,\eps_k,\eps,f_k,f,p)$-strengthened pair for $G$ if the following are true.
	\begin{enumerate}[label=\textup{(S\arabic*)}]
	 \item\label{S:refine} $\cP^*_f$ refines $\cP^*_c$.
	 \item\label{S:cequit} $\cP^*_c$ is $(t_0,t_1,\eps)$-equitable.
	 \item\label{S:creg} $G$ is $(\eps_k,p)$-regular with respect to $\cP^*_c$.
	 \item\label{S:fequit} $\cP^*_f$ is $(t_0,t_2,f)$-equitable.
	 \item\label{S:freg} $G$ is $(f_k,p)$-regular with respect to $\cP^*_f$.
	 \item\label{S:dens} For all but at most $\eps_k^2\binom{|V(G)|}{k}$ elements $Q$ of $\cross_k(\cP_c)$, we have $d_p\big(G\big|\hat{\cP}(Q,\cP_c^*)\big)=d_p\big(G\big|\hat{\cP}(Q,\cP_f^*)\big)\pm\eps_k$.
	\end{enumerate}
\end{definition}

We refer to $\cP^*_c$ as the \emph{coarse partition} and $\cP^*_f$ as the \emph{fine partition}. Slightly extending the usual definition, we say a $k$-polyad $\hat{P}(Q;\cP^*_c)$ is \emph{irregular} (with respect to $G$) if any one of the following three things occurs:
\begin{enumerate}[label=(\roman*)]
 \item $G$ is not $(\eps_k,p)$-regular with respect to $\hat{P}(Q;\cP^*_c)$,
 \item for more than an $\eps_k$-fraction of the $k$-sets $Q'$ supported on $\hat{P}(Q;\cP^*_c)$, $G$ is not $\big(f_k,p\big)$-regular with respect to $\hat{P}(Q';\cP^*_f)$, or
 \item for more than an $\eps_k$-fraction of the $k$-sets $Q'$ supported on $\hat{P}(Q;\cP^*_c)$, we have $d_p\big(G\big|\hat{P}(Q';\cP^*_f)\big)\neq d_p\big(G\big|\hat{P}(Q;\cP^*_c)\big)\pm\eps_k$.
\end{enumerate}
If a polyad of $\cP^*_c$ is not irregular, we say it is \emph{regular}.

We will always choose $f_k$ such that $f_k\le\eps_k^2$, and $\eps$ small enough that every $k$-polyad supports very close to the same number of $k$-edges. Under this assumption, it is straightforward to check that at most a $4\eps_k$-fraction of polyads in $\cP^*_c$ are irregular (we will prove this in Appendix~\ref{sec:proofreg}, Proposition~\ref{prop:fewirreg}).

We need one more definition. Given any (not necessarily distinct) subsets $E_1,\dots,E_k$ in $\binom{[n]}{k-1}$, we say a $k$-set $S\subset [n]$ is \emph{rainbow} for the $E_i$ if there is an injective labelling of the $(k-1)$-subsets of $S$ with the numbers $1,\dots,k$ such that the $(k-1)$-subset labelled $i$ is in $E_i$. We write $K_k(E_1,\dots,E_k)$ for the set of rainbow $k$-sets in $[n]$. We say that a graph $G$ on $[n]$ is \emph{$(\eta,p)$-upper regular} if the following holds. For any $E_1,\dots,E_k$, we have
\[\big|E(G)\cap K_k(E_1,\dots,E_k)\big|\le p\big|K_k(E_1,\dots,E_k)\big|+p\eta n^k\,.\]

Finally, we are in a position to state our strengthened sparse version of the Strong Hypergraph Regularity Lemma. Informally, what this says is that we can find $\cP^*_c$ and $\cP^*_f$ which are simultaneously a strengthened pair for $s$ edge-disjoint graphs, for any (fixed) regularity $\eps_k$ of $\cP^*_c$, where $\eps$ and $f$ can be as small as desired depending on the number of parts in $\cP^*_c$ and $\cP^*_f$ respectively, and furthermore the regularity $f_k$ of $\cP^*_f$ can depend arbitrarily on the number of parts of $\cP^*_c$.

\begin{lemma}[Strengthened Sparse Strong Hypergraph Regularity Lemma]\label{lem:ssshrl}
	Given integers $k\ge2$ and $t_0$ and $s$, real $\eps_k>0$ and functions $\eps,f_k,f:\mathbb{N}\to(0,1]$, there exists a real $\eta>0$ and integers $T$ and $n_0$ such that the following holds for all $n\ge n_0$ with $T!|n$. Let $V$ be a vertex set of size $n$, suppose that $G_1,\dots,G_s$ are $k$-uniform hypergraphs on $V$, and suppose $\cQ^*$ is a family of partitions on $V$ which is $(1,t_0,\eta)$-equitable. Suppose furthermore that for each $1\le i\le s$ there is a real $p_i\in(0,1]$ such that $G_i$ is $(\eta,p_i)$-upper regular. Then there are integers $t_1,t_2$ with $t_0\le t_1\le t_2\le T$, and families of partitions $\cP^*_c$ and $\cP^*_f$, both refining $\cQ^*$, such that for each $1\le i\le s$, the pair $(\cP^*_c,\cP^*_f)$ is a $\big(t_0,t_1,t_2,\eps_k,\eps(t_1),f_k(t_1),f(t_2),p_i\big)$-strengthened pair for $G_i$.
\end{lemma}

We prove this lemma in Appendix~\ref{sec:proofreg}. Note that the case $k=2$ will not be used here; and in this setting the `families of partitions' are simply vertex set partitions and the functions $\eps$ and $f$ play no r\^ole.

Given a $(t_0,t_1,t_2,\eps_k,\eps,f_k,f,p)$-strengthened pair $(\cP^*_c,\cP^*_f)$ for $G$, recall that $\cP^*_c$ has the structure of a multicomplex. We denote by $\cR_{\eps_k}(G;\cP^*_c,\cP^*_f)$ the \emph{$\eps_k$-reduced multicomplex of $G$ with respect to $(\cP^*_c,\cP^*_f)$}, which is the (unique) maximal submulticomplex of $\cP^*_c$ which has the following properties.
\begin{enumerate}[label=(RG\arabic*)]
 \item\label{rg:k} Every $k$-edge of $\cR_{\eps_k}(G;\cP^*_c,\cP^*_f)$ is regular.
 \item\label{rg:lower} For each $1\le i\le k-1$, each $i$-edge of $\cR_{\eps_k}(G;\cP^*_c,\cP^*_f)$ is in the boundary of at least 
 \[\Big(1-2^{i+2}\eps_k^{1/k}\Big)t\prod_{j=2}^{i+1}d_j^{-\binom{i}{j-1}} \quad\text{ if $i<k-1$, and }\quad\Big(1-2^{k+1}\eps_k^{1/k}\Big)t\prod_{j=2}^{k-1}d_j^{-\binom{k-1}{j-1}}\quad\text{ if $i=k-1$}\]
 $(i+1)$-edges of $R_{\eps_k}(G;\cP^*_c,\cP^*_f)$.
\end{enumerate}

The existence and uniqueness of the reduced multicomplex are trivial: we obtain it by simply iteratively removing from the multicomplex $\cP^*_c$ edges which either fail one of~\ref{rg:k} or~\ref{rg:lower}, or from whose boundary we removed edges (so that they are no longer supported and cannot be in the multicomplex). It is easy, but not quite trivial, to show that most of the vertices of $\cR_{\eps_k}(G;\cP^*_c,\cP^*_f)$ (i.e. the parts of $\cP_c$) are also $1$-edges of $\cR_{\eps_k}(G;\cP^*_c,\cP^*_f)$. Now given $d>0$, we let $\cR_{\eps_k,d}(G;\cP^*_c,\cP^*_f)$ be the (unique) submulticomplex of $\cR_{\eps_k}(G;\cP^*_c,\cP^*_f)$ obtained by removing all $k$-edges corresponding to polyads whose relative $p$-density is less than $d$.
We call $\cR_{\eps_k,d}(G;\cP^*_c,\cP^*_f)$ the \emph{$(\eps_k,d)$-reduced multicomplex of $G$ with respect to $(\cP^*_c,\cP^*_f)$}.

In Appendix~\ref{sec:proofreg} we show the following lemma.

\begin{lemma}
	\label{lem:goodconnected}
	Given $k\in\mathbb{N}$ and $d>0$ suppose that $t_0\in\mathbb{N}$ is sufficiently large. Given any constants $\delta,\eps_k,\nu>0$, any function $\eps:\mathbb{N}\to(0,1]$ which tends to zero sufficiently fast, any $t_1,t_2\in\mathbb{N}$, any $0<f_k\le\eps_k^2$ and any $f>0$, there exists $\eta>0$ such that the following holds for any sufficiently large $n$ and any $p>0$. Suppose $G$ is an $n$-vertex hypergraph which is $(\eta,p)$-upper regular and every $(k-1)$-set in $V(G)$ is contained in at least $\delta p n$ edges. Suppose that $(\cP^*_c,\cP^*_f)$ is a $(t_0,t_1,t_2,\eps_k,\eps(t_1),f_k,f,p)$-strengthened pair for $G$.
	
  Let $\cR=\cR_{\eps_k,d}(G;\cP^*_c,\cP^*_f)$ be the $(\eps_k,d)$-reduced multicomplex of $G$, and suppose that $\cP^*_c$ has $t$ clusters and density vector $\mathbf{d}=(d_{k-1},\dots,d_2)$. Then $\cR$ contains at least $\big(1-4\eps_k^{1/k}\big)t$ $1$-edges, and every $(k-1)$-edge of $\cR$ is contained in at least
	\[\big(\delta-2d-2^{k+2}\eps_k^{1/k}\big)t\cdot \prod_{i=2}^{k-1}d_i^{-\binom{k-1}{i-1}}\]
	$k$-edges of $\cR$.
	
	Finally, if $\delta>\tfrac12+2d+2^{k+2}\eps_k^{1/k} + \nu$, then any induced subcomplex of $\cR$ on at least $(1-\nu)t$ $1$-edges is tightly linked.
\end{lemma}

The next lemma, often called the Dense Counting Lemma, is a straightforward generalisation of the well-known graph Counting Lemma (in contrast to the so-called Sparse Counting Lemma, which is much harder; the difference being that in the Dense Counting Lemma the parameter $\eps$ of regularity is much smaller than all the density parameters).
We state the special case of counting  $(k-1)$- and $k$-cliques in $(k-1)$-uniform hypergraphs.
The version that we need works with parts of different sizes, but this can be easily derived from the version with parts of the same size from~\cite[Theorem 6.5]{KRS}.

\begin{lemma}[Dense Counting Lemma]\label{lem:DCL}
	For all integers $k\ge2$ and constants $\alpha,\gamma,d_0>0$, there exists $\eps>0$ such that the following holds. Let $\mathbf{d}=(d_{k-1},\dots,d_2)$ be a vector of real numbers with $d_i\ge d_0$ for each $2\le i\le k-1$, and let $G$ be a $k$-partite $(k-1)$-complex which is $(\mathbf{d},\eps,\eps,1)$-regular 
	 and has parts $V_1,\dots,V_k$ of size at least $m \ge \alpha^{-1} \eps^{-1} $. Then for $V_i' \subseteq V_i$ of size $|V_i'| \ge \alpha |V_i|$ for $i=1,\dots,k$ the number of copies of the $k$-vertex complete $(k-1)$-complex in $G[V_1' , \dots , V_k']$ is
	\[\big(1\pm\gamma\big) \prod_{i=1}^k |V_i'| \prod_{i=2}^{k-1}d_i^{\binom{k}{i}}\,,\]
	and the number of copies of the $(k-1)$-vertex complete $(k-1)$-complex in $G[V_1' , \dots , V_{k-1}']$ is
	\[\big(1\pm\gamma\big) \prod_{i=1}^{k-1} |V_i'| \prod_{i=2}^{k-1}d_i^{\binom{k-1}{i}}\,.\]
\end{lemma}

Note that with $\alpha=1$ this is the Dense Counting Lemma with parts of the same size.
We give the proof for this generalisation in Appendix~\ref{sec:proofreg}.
If we do not remove too many vertices from the $1$-cells we still have a regular complex with slightly different parameters.
We will use this to prove Lemma~\ref{lem:DCL}, but also need it in our arguments.

\begin{lemma}[Regular Restriction Lemma~{\cite[Lemma 28]{ABCM}}]
	\label{lem:RRL}
	For all integers $k\ge2$ and constants $\alpha,d_0>0$, there exists $\eps>0$ such that the following holds. Let $\mathbf{d}=(d_{k-1},\dots,d_2)$ be a vector of real numbers with $d_i\ge d_0$ for each $2\le i\le k-1$, and let $G$ be a $k$-partite $(k-1)$-complex with parts $V_1,\dots,V_k$ of size $m\ge\eps^{-1}$ which is $(\mathbf{d},\eps,\eps,1)$-regular. 
	Choose any $V_i' \subseteq V_i$ of size at least $\alpha m$ for $i=1,\dots,k$.
	Then the induced subcomplex $G[V_1' , \dots , V_k']$ is $(\mathbf{d},\sqrt{\eps},1)$-regular.
\end{lemma}

We will also need the following two lemmas that follow from the Dense Counting Lemma and the Regular Restriction Lemma.
The first allows us to control the `degree' of most tuples within the $(k-1)$-complex.
For $k=3$ this basically says that most edges are contained in the correct number of triangles.

\begin{lemma}[Degree Counting Lemma]
	\label{lem:DCL-variant}
	For all integers $k\ge2$ and constants $\alpha,\gamma,d_0>0$, there exists $\eps>0$ such that the following holds.
	Let $\mathbf{d}=(d_{k-1},\dots,d_2)$ be a vector of real numbers with $d_i\ge d_0$ for each $2\le i\le k$, and let $G$ be a $k$-partite $(k-1)$-complex with parts $V_1,\dots,V_k$ of size $m\ge \alpha^{-1} \eps^{-1}$ which is $(\mathbf{d},\eps,\eps,1)$-regular.
	Choose any $V_i' \subseteq V_i$ of size at least $\alpha m$ for $i=1,\dots,k$ and let $G'=G[V_1' , \dots , V_k']$.
	Then at least a $(1-\gamma)$-fraction of the $(k-1)$-tuples in $G[V_1' , \dots , V_{k-1}']$ is contained in
	\[\big(1\pm\gamma\big)|V_k'| \prod_{i=2}^{k-1} d_i^{\binom{k-1}{i-1}}\]
	copies of the $k$-vertex complete $(k-1)$-complex in $G'$.
	Furthermore, the $\gamma$-fraction of $(k-1)$-tuples in most copies of the $k$-vertex complete $(k-1)$-complex in $G'$ contain in total at most
	\[\tfrac52 \gamma \prod_{i=1}^k |V_i'| \prod_{i=2}^{k} d_i^{\binom{k}{i-1}}\]
	copies of the $k$-vertex complete $(k-1)$-complex in $G'$.
	Similarly, at least a $(1-\gamma)$-fraction of the $(k-2)$-tuples in $G[V_1' , \dots , V_{k-2}']$ is contained in
	\[\big(1\pm\gamma\big)|V_{k-1}'| \prod_{i=2}^{k-1} d_i^{\binom{k-2}{i-1}}\]
	copies of the $(k-1)$-vertex complete $(k-1)$-complex in $G'$ together with a vertex from $V_{k-1}'$.
\end{lemma}

The second looks a bit more complicated, but we only need the variant with all parts of the same size.
For $k=3$ this implies that if many vertices have high degree into two different $2$-cells, then they will also support many triangles.

\begin{lemma}[Minimum Degree Lemma]
	\label{lem:MDL}
	For all integers $k\ge 3$ and constants $\gamma,\delta, d_0>0$, there exists $\eps>0$ such that the following holds.
	Let $\mathbf{d}=(d_{k-1},\dots,d_2)$ be a vector of real numbers with $d_i\ge d_0$ for each $2\le i\le k-1$, and let $G$ be a $k$-partite $(k-1)$-complex with parts of size $m\ge\eps^{-1}$ which is $(\mathbf{d},\eps,\eps,1)$-regular.
	Moreover, with integers $a,b,c$ such that $a+b-c=k$, let $A$ be part of an $a$-cell, $B$ be part of a $b$-cell, and $C$ be part of a $c$-cell such that the tuples from $C$ have degree $(\delta\pm\gamma) m^{a-c} \prod_{i=2}^{k-1}d_i^{\binom{a}{i}-\binom{c}{i}}$ into $A$. Suppose that every edge of $B$ contains an edge of $C$, and that $|B|\ge\gamma m^{b} \prod_{i=2}^{k-1}d_i^{\binom{b}{i}}$.
	Then there are
	\[ |B| (\delta\pm2\gamma)m^{a-c} \prod_{i=2}^{k-1}d_i^{\binom{k}{i}-\binom{b}{i}}\]
	copies of the $k$-vertex complete $(k-1)$-complex in $G$ supported by $A$ and $B$.
\end{lemma}

All of these lemmas are broadly standard, and hence we prove them in the appendix.

\subsection{Properties of the random hypergraph}

We use the following standard versions of the Chernoff bound.

\begin{theorem}
 Let $X$ be a random variable with distribution $\Bin(n,p)$. Then for any $\eps>0$ we have
 \[\Pr(X\ge pn+\eps n)\le\exp(-D(p+\eps||p)n)\quad \text{ and }\quad \Pr(X\le pn-\eps n)\le\exp(-D(p-\eps||p)n)\,,\]
 where $D(x||y)$ is the Kullback-Leibler divergence. From this it follows
 \[\Pr(|X-pn|>\eps pn)<2\exp\big(-\tfrac{\eps^2pn}{3}\big)\quad \text{ for any $\eps\le\tfrac32$}\]
 and if $t\ge 6pn$ we have
 \[Pr(X\ge pn+t)<\exp(-t)\,.\]
\end{theorem}

\begin{lemma}\label{lem:upperreg}
 Given $\eta>0$, $k \in \mathbb{N}$ there exists $C$ such that if $p\ge\tfrac{C}{n}$, then $\Gamma=G^{(k)}(n,p)$, and all its subgraphs, are a.a.s.\ $(\eta,p)$-upper regular. 
\end{lemma}
\begin{proof}
	Observe that if $\Gamma=G^{(k)}(n,p)$ is $(\eta,p)$-upper-regular, then automatically all its subgraphs are also. We assume without loss of generality that $\eta<1$, and set $C=18k\eta^{-3}$.
	
	Given any $E_1,\dots,E_k\subset \binom{[n]}{k}$, we aim to estimate the probability that $E_1,\dots,E_k$ witness the failure of $G^{(k)}(n,p)$ to be $(\eta,p)$-upper regular. The expected number of edges of $G^{(k)}(n,p)$ which appear on the sets $K_k(E_1,\dots,E_k)$ is $p\big|K_k(E_1,\dots,E_k)\big|$, and the distribution is binomial, so we may apply the Chernoff bound.
	
	If $\big|K_k(E_1,\dots,E_k)\big|\le\tfrac16\eta n^k$, then failure to be $(\eta,p)$-upper regular means that the number of $k$-edges appearing on $K_k(E_1,\dots,E_k)$ is at least seven times the expected number; by the Chernoff bound the probability of this event is less than $\exp(-p\eta n^k)<\exp(-kn^{k-1})$.
	
	If $\big|K_k(E_1,\dots,E_k)\big|\ge\tfrac16\eta n^k$, then the probability that more than $(1+\eta)p\big|K_k(E_1,\dots,E_k)\big|$ edges appear is at most
	\[\exp\big(-\frac{\eta^2p\big|K_k(E_1,\dots,E_k)\big|}{3}\big)\le\exp\big(-\frac{\eta^2\tfrac{C}{n}\eta n^k}{18}\big)=\exp\big(-kn^{k-1}\big)\,.\]
	Since there are at most $2^{\binom{n}{k-1}}$ choices for each $E_i$, by the union bound the probability that $G^{(k)}(n,p)$ is not $(\eta,p)$-upper regular is at most
	\[2^{k\binom{n}{k-1}}\exp\big(-kn^{k-1}\big)\]
	which tends to zero as $n$ tends to infinity.
\end{proof}

Given a set $S\subseteq V(\Gamma)$, we say a $(k-1)$-set $x$ is \emph{$(\eps,p,1)$-good for $S$} if we have
\[\big|\big\{s\in S:x\cup\{s\}\in E(\Gamma)\big\}\big|=p|S|\pm\eps p n\,.\]
For each $\ell\ge 2$, we say inductively that a $(k-1)$-set $x$ is \emph{$(\eps,p,\ell)$-good for $S$} if it is $(\eps,p,\ell-1)$-good for $S$ and there are at most $\eps p n$ edges of $\Gamma$ which contain $x$ and in addition contain a set which is not $(\eps,p,\ell-1)$-good for $S$.

\begin{lemma}
	\label{lem:goodtuples}
 Given $\eps>0$, $k \in \mathbb{N}$ there exists $C$ such that if $p\ge\tfrac{C\log n}{n}$, then $\Gamma=G^{(k)}(n,p)$ a.a.s.\ has the following property. For each set $S\subset V(\Gamma)$, and each $1\le \ell\le \tfrac1C \log \log n$, there are at most $o(n)$ $(k-1)$-sets in $V(\Gamma)$ outside $S$ which are not $(\eps,p,\ell)$-good for $S$.
\end{lemma}

\begin{proof}
 Given $S$, we first estimate the number of $(k-1)$-sets $x$ which are outside $S$ and not $(\eps,p,1)$-good for $S$.
 
 If $|S|<\tfrac16\eps n$, then failure of a given $x$ to be $(\eps,p,1)$-good for $S$ means $x$ forms an edge with at least $7p|S|$ vertices in $S$, the probability of which is by the Chernoff bound at most $\exp(-\eps p n)$, which for large enough $C$ is smaller than $n^{-k}$. If on the other hand $|S|\ge\tfrac16pn$, then the probability that $x$ does not form an edge with $(1\pm\eps)p|S|$ vertices of $S$ is at most $2\exp\big(\tfrac{-\eps^2 p|S|}{3}\big)<n^{-k}$ for large enough $C$. We see that in either case, the probability that $x$ is not $(\eps,p,1)$-good for $S$ is at most $n^{-k}$. Now if $x$ and $x'$ are two different $(k-1)$-sets outside $S$, then the events of $x$ and of $x'$ being not $(\eps,p,1)$-good for $S$ are independent, so again using the Chernoff bound we can estimate the likelihood of many sets being bad for $S$. The expected number of bad sets for $S$ is at most $n^{k-1}\cdot n^{-k}=n^{-1}$. 
 Therefore, for any $t \ge 1$, we can bound the probability that there are $t$ or more bad $(k-1)$-sets for $S$ by
\[ \exp\left(-D(n^{-k}+tn^{1-k} || n^{-k} ) n^{k-1} \right)  \le \exp\left(-\frac{t \log n}{2}  \right)\, . \]
 In particular, taking $t=4n/\log n$ and using the union bound, the probability that there exists a set $S$ for which more than $4n/\log n$ $(k-1)$-sets are not $(\eps,p,1)$-good is at most $2^{-n}$. Suppose that $\Gamma$ is such that this good event occurs, and in addition that every $(k-1)$-set of vertices of $\Gamma$ is contained in at most $2pn$ edges of $\Gamma$.
 
 Let $K=2k\eps^{-1}$. Now given $S$ and $\ell\ge 1$, we claim that the number of $(k-1)$-sets outside $S$ which are not $(\eps,p,\ell)$-good for $S$ is at most $4n\cdot K^{\ell-1}/\log n$. We prove this by induction on $\ell$; the base case $\ell=1$ is the assumption on $\Gamma$. Suppose $\ell\ge 2$, and that the number of $(k-1)$-sets outside $S$ which are not $(\eps,p,\ell-1)$-good for $S$ is at most $4n\cdot K^{\ell-2}/\log n$. For each $(k-1)$-set $x$ outside $S$ which is not $(\eps,p,\ell-1)$-good for $S$, we assign to each $(k-1)$-set $y$ such that $x\cup y$ is an edge of $\Gamma$ one unit of badness. Observe that the total number of units of badness assigned is at most $(k-1)\cdot 2pn\cdot 4n\cdot K^{\ell-2}/\log n$. On the other hand, a set $y$ which is $(\eps,p,\ell-1)$-good for $S$ can only fail to be $(\eps,p,\ell)$-good for $S$ if it is assigned at least $\eps p n$ units of badness. It follows that the total number of such sets is at most $2(k-1)\eps^{-1} \cdot 4n K^{\ell-2}/\log n$, and so the number of $(k-1)$-sets outside $S$ which are not $(\eps,p,\ell)$-good for $S$ is at most
 \[2(k-1)\eps^{-1} \cdot 4n K^{\ell-2}/\log n+4n\cdot K^{\ell-2}/\log n\le 4n\cdot K^{\ell-1}/\log n\,,\]
 as desired. In particular, this formula is in $o(n)$ for all $1\le\ell\le\tfrac1C \log\log n$ with $C$ large enough.
\end{proof}

For a given $(k-1)$-tuple, we will find many paths starting from there.
To get expansion we need to ensure that they have many different end-tuples.

\begin{lemma}
	\label{lem:goodexpansion}
	For $\gamma>0$, $k \geq 3$, any fixed integer $\ell > \frac{k-1}{\gamma}+k-1$, and any $\mu>0$ a.a.s.\ in~$\Gamma = G^{(k)}(n,p)$ with $p=n^{-1+\gamma}$ the following holds.
	For any $(k-1)$-tuple $\tpl{x}$ in $V(\Gamma)$ and a set $\mathcal{P}$ of at least $(\mu p n)^\ell$ tight paths in $\Gamma$ with $\ell+(k-1)$ vertices and rooted at $\tpl{x}$, the number of end $(k-1)$-tuples of the paths in $\mathcal{P}$ is at least $\frac{\mu^{2 \ell}}{8(2 \ell)!} n^{k-1}$.
	Moreover, when $(k-1) | \ell$, the same holds for spike paths rooted at $\tpl{x}$.
\end{lemma}

To prove Lemma~\ref{lem:goodexpansion} we need a concentration result of Kim and Vu \cite{KimVu}. We first give some definitions and then state the result.
Let $m$ be a positive integer and $H$ be a hypergraph with $|V(H)| = m$ and each edge has at most $r$ vertices. Let $p \in [0,1]$ and let $X_i, i \in V(H)$ be independent random variables with $\mathbb{P}[X_i = 1] = p$ and $\mathbb{P}[X_i = 0] = 1-p$.
We define the random variable 
\[
Y_H = \sum_{f \in E(H)} \prod_{i \in f} X_i.
\]
For each subset $A \subseteq V(H)$, we define the $A$-truncated subgraph $H(A)$ of $H$ to be the subgraph of $H$ with $V(H(A)) = V(H) \setminus A$ and $E(H(A)) = \{f \subseteq V(H(A)) \colon f \cup A \in E(H)\}$. Hence
\[
Y_{H(A)} = \sum_{\substack{f \in E(H) \\ A \subseteq f}} \prod_{i \in f \setminus A} X_i.
\]
Now, for $0 \leq i \leq r$, we set $\mathcal{E}_i(H) = \max_{A \subseteq V(H), |A| = i} \mathbb{E}[Y_{H(A)}]$. Note that $\mathcal{E}_0(H) = \mathbb{E}[Y_H]$. Finally, we let $\mathcal{E}(H) = \max_{0 \leq i \leq r} \mathcal{E}_i(H)$ and $\mathcal{E}'(H) = \max_{1 \leq i \leq r} \mathcal{E}_i(H)$. 

\begin{theorem}[Kim-Vu polynomial concentration \cite{KimVu}]
	\label{thm:Kim-Vu}
	In this setting we have 
	\[
	\mathbb{P}[|Y_H - E(Y_H)| > a_r (\mathcal{E}(H)\mathcal{E}'(H))^{1/2}\lambda^r)] = O(\exp(-\lambda + (r-1)\log m))
	\]
	for any $\lambda > 1$ and $a_r = 8^r r!^{1/2}$.
\end{theorem}

Moreover, we will need the following definitions. Let $k \geq 3$, $\ell \geq k-1$, and $\gamma >0$.
We define $D_\ell$ to be the $k$-graph obtained from two vertex-disjoint tight paths on $\ell + k -1$ vertices by identifying the end $(k-1)$-tuples. Let $\mathcal{D}_\ell$ be the set of hypergraphs obtained from $D_\ell$ by additionally identifying some (or none) of the not yet identified vertices from the first tight path with such vertices from the second without completely collapsing it into a tight path. More precisely, we let $U = \{u_1, \dots, u_{\ell + k-1}\}$ and $W = \{w_1, \dots, w_{\ell + k -1}\}$ be two sets of vertices that are disjoint except that $\tpl{x} = (u_1, \dots, u_{k-1}) = (w_1, \dots, w_{k-1})$ and $\tpl{y} = (u_{\ell + k-1}, \dots, u_{\ell +1}) = (w_{\ell + k-1}, \dots, w_{\ell +1})$. Then $D_\ell$ is the hypergraph with vertex set $U \cup W$ and edge set
\[
\{\{u_{i}, \dots, u_{i+k-1}\} \colon i \in [\ell]\} \cup \{\{w_{i}, \dots, w_{i+k-1}\} \colon i \in [\ell]\}.
\]
For $0 \leq j \leq \ell - (k-1)$, we denote by $\mathcal{D}_\ell^{j}$ the graphs obtained from $D_\ell$ by taking sets $I_1, I_2 \subseteq \{k, \dots, \ell\}$ each of size $j$ and a bijection $\sigma \colon I_1 \rightarrow I_2$ and  identifying $u_i$ with $w_{\sigma(i)}$ for all $i \in I_1$, where, if $j = \ell -(k-1)$, then we do not allow $\sigma$ to be the identity (since that would collapse $D_\ell$ into a tight path). We say that such a graph $F \in \mathcal{D}_\ell^{j}$ is rooted at $\tpl{x}$.
Finally, we let $\mathcal{D}_\ell = \bigcup_{0 \leq j \leq \ell -(k-1)} \mathcal{D}_\ell^j$.

We now prove the following lemma which we will use to prove Lemma~\ref{lem:goodexpansion}.
\begin{lemma}
	\label{lem:D_ell}
	For $\gamma > 0$, $k \geq 3$, and any fixed integer $\ell > \frac{k-1}{\gamma}+k-1$ a.a.s.\ in~$\Gamma = G^{(k)}(n,p)$ with $p = n^{-1+\gamma}$ the following holds. For all $(k-1)$-tuples $\tpl{x}$ in $V(\Gamma)$, the number of copies of elements of $\mathcal{D}_\ell$ in $\Gamma$ that are rooted at $\tpl{x}$ is at most $2p^{2\ell}n^{2\ell - (k-1)}$.
\end{lemma}
\begin{proof}
	Fix a $(k-1)$-tuple $\tpl{x}$ in $V(\Gamma)$ and an integer $\ell > \frac{k-1}{\gamma} +k-1$. Let $F \in \mathcal{D}_\ell$ and consider the complete $k$-graph $K_n^{(k)}$ on $n$ vertices. We define a hypergraph $H_F$ as follows. Let $V(H_F) = E(K_n^{(k)})$ and let 
	\[
	E(H_F) = \left\{\mathcal{F} \in \binom{E(K_n^{(k)})}{e(F)} \colon \mathcal{F} \text{ spans a copy of $F$ in $K_n^{(k)}$ rooted at $\tpl{x}$}\right\}.
	\]
	Note that, since $e(F) \leq 2\ell$, each edge in $H_F$ has size at most $2 \ell$. For each $e \in V(H_F)$ let $X_e$ be the random variable for which $X_e = 1$ if $e$ is an edge of $\Gamma$ and $X_e = 0$ otherwise. Note that $\mathbb{P}[X_e =1] = p$. It is easy to see that with these definitions $Y_{H_F}$ is the number of copies of $F$ in $\Gamma$ rooted at $\tpl{x}$. Since $e(D_\ell) = 2 \ell$, $v(D_\ell) = 2\ell$, and $k-1$ vertices are rooted, we have $\binom{n}{2\ell -(k-1)}p^{2\ell} \leq \mathbb{E}[Y_{H_{D_\ell}}] \leq p^{2\ell}n^{2\ell-(k-1)}$, in particular $\mathbb{E}[Y_{H_{D_\ell}}] = \Theta(p^{2\ell}n^{2\ell -(k-1)})$.
	\begin{claim}
		For $F \in \mathcal{D}_\ell \setminus \{D_\ell\}$, we have
		\[
		\mathbb{E}\left[Y_{H_F}\right] = o\left(\mathbb{E}\left[Y_{H_{D_\ell}}\right]\right).
		\]
	\end{claim}
	\begin{claimproof}
		We split the proof into two cases depending on the integer $j$ for which we have $F \in \mathcal{D}_\ell^j$.
		
		First suppose that $F \in \mathcal{D}_\ell^j$ for some $j \in \{1, \dots, \ell -2(k-1)\}$. Note that $v(F) = 2\ell -j$. We claim that $e(F) \geq 2\ell - j$. This can be seen as follows. Recall that $F$ is obtained from $D_\ell$ by identifying $j$ additional vertices from the first tight path in $D_\ell$ with vertices from the second. This leaves $\ell -(k-1)-j \geq k-1$ unidentified vertices in the first tight path. In addition to the $\ell$ edges in the second path, $F$ contains an edge ending in each of the unidentified vertices and one more additional edge starting with each of the last $k-1$ unidentified vertices (these edges cannot end in an unidentified vertex, so there is no double counting). Thus
		\[
		e(F) \geq \ell + \ell -(k-1) -j +(k-1) = 2 \ell -j.
		\]
		Hence, since $k-1$ vertices are rooted,
		\begin{align*}
		\mathbb{E}[Y_{H_F}] \leq p^{2\ell-j}n^{2\ell-j-(k-1)} = p^{2\ell}n^{2\ell-(k-1)}(pn)^{-j} = p^{2\ell}n^{2\ell-(k-1)}n^{-j\gamma} = o\left(\mathbb{E}\left[Y_{H_{D_\ell}}\right]\right).
		\end{align*}
		
		Now suppose that $F \in \mathcal{D}_\ell^j$ for some $j \in \{\ell -2(k-1)+1, \dots, \ell -(k-1)\}$.
		As in the previous case, in addition to the $\ell$ edges in the second path, $F$ contains an edge ending in each of the $\ell -(k-1) - j$ unidentified vertices. Thus $e(F) \geq 2\ell -(k-1) -j$.
		Hence, since $v(F) = 2\ell -j$, $k-1$ vertices are rooted, and $j > \ell - 2(k-1)$, we have
		\begin{align*}
		\mathbb{E}[Y_{H_F}] &\leq p^{2\ell-j- (k-1)}n^{2\ell-j-(k-1)} = p^{2\ell}n^{2\ell-(k-1)}(pn)^{-j}p^{-(k-1)} \\
		&=p^{2\ell}n^{2\ell-(k-1)} n^{-j\gamma +(k-1)-\gamma (k-1)} < p^{2\ell}n^{2\ell-(k-1)} n^{-\gamma (\ell - 2(k-1))+(k-1)-\gamma (k-1)} \\
		&=p^{2\ell}n^{2\ell-(k-1)} n^{-\gamma (\ell - (k-1))+(k-1)}
		= o\left(\mathbb{E}\left[Y_{H_{D_\ell}}\right]\right),
		\end{align*}
		since $\ell > \frac{k-1}{\gamma}+k-1$.
	\end{claimproof}
	Combining the claim with our bound on $\mathbb{E}[Y_{H_{D_\ell}}]$ we obtain
	\begin{align}
	\label{eq:Ebound}
	\sum_{F \in \mathcal{D}_\ell} \mathbb{E}[Y_{H_F}] \leq \frac{3}{2}p^{2\ell}n^{2\ell-(k-1)}.
	\end{align}
	Next we show that, for all $F \in \mathcal{D}_\ell$, the random variable $Y_{H_F}$ is concentrated around its expectation.
	\begin{claim}
		For all $F \in \mathcal{D}_\ell$, we have 
		\[
		\mathbb{P}\left[|Y_{H_F} - \mathbb{E}[Y_{H_F}]| > \frac{p^{2\ell}n^{2\ell-(k-1)}}{2|\mathcal{D}_\ell|}\right] =O\left(\exp(-n^{\gamma/10\ell})\right).
		\]
	\end{claim}
	\begin{claimproof}
		Let $F \in \mathcal{D}_\ell$. We first show that $\mathcal{E}(H_F) = \mathbb{E}[Y_{H_F}]$ and $\mathcal{E}'(H_F) \leq n^{-\gamma/2}\mathbb{E}[Y_{H_F}]$. Let $1 \leq i \leq 2\ell$ and $A \subseteq V(H_F) = E(K_n^{(k)})$ with $|A| = i$. Note that the number of vertices of $K_n^{(k)}$ covered by $A$ is $v_A = \left\lvert \bigcup A \right\rvert \geq k + i - 1$ since each edge beyond the first covers at least one additional vertex. Moreover, note that $Y_{H_F(A)}$ is the number of copies of $F$ in $\Gamma + A$ that are rooted at $\tpl{x}$ and contain $A$. 
		Thus
		\begin{align*}
		\mathbb{E}[Y_{H_F(A)}] &\leq n^{v(F) - v_A}p^{e(F)-i} \leq n^{v(F) - (k-1)} p^{e(F)}(np)^{-i} 
		\\ &= n^{v(F) - (k-1)}p^{e(F)}n^{-i\gamma} \leq n^{-\gamma/2}\mathbb{E}[Y_{H_F}], 
		\end{align*}
		since $\mathbb{E}[Y_{H_F}] = \Theta(n^{v(F)-(k-1)}p^{e(F)})$ and $i \geq 1$.
		Hence $\mathcal{E}(H_F) = \mathbb{E}[Y_{H_F}]$ and $\mathcal{E}'(H_F) \leq n^{-\gamma/2}\mathbb{E}[Y_{H_F}]$. This implies $(\mathcal{E}(H_F)\mathcal{E}'(H_F))^{1/2} \leq n^{-\gamma/4} \mathbb{E}[Y_{H_F}] \leq n^{-\gamma/4}p^{2\ell}n^{2\ell-(k-1)}$.
		Therefore, with 
		\[
		\lambda_F = \left(\frac{p^{2\ell}n^{2\ell-(k-1)}}{2|\mathcal{D}_\ell|a_{2\ell}(\mathcal{E}(H_F)\mathcal{E}'(H_F))^{1/2}}\right)^{1/2\ell} \geq \left(\frac{n^{\gamma/4}}{2|\mathcal{D}_\ell|a_{2\ell}}\right)^{1/2\ell} \geq n^{\gamma/9\ell},
		\]
		we have, by Theorem~\ref{thm:Kim-Vu}, 
		\begin{align*}
		\mathbb{P}\left[|Y_{H_F} - \mathbb{E}[Y_{H_F}]| > \frac{p^{2\ell}n^{2\ell-(k-1)}}{2|\mathcal{D}_\ell|}\right]
		&= \mathbb{P}[|Y_{H_F} - \mathbb{E}[Y_{H_F}]| > a_{2\ell}(\mathcal{E}(H_F)\mathcal{E}'(H_F))^{1/2} \lambda_F^{2\ell}] \\
		& = O\left(\exp\left(-\lambda_F+(2\ell -1)\log\binom{n}{k}\right)\right) \\
		& = O\left(\exp(-n^{\gamma/9\ell} + (2\ell -1) k \log n)\right)
		= O\left(\exp(-n^{\gamma/10\ell})\right).
		\end{align*}
	\end{claimproof}
	Now let $Z_{\tpl{x}}$ be the number of copies of elements of $\mathcal{D}_\ell$ in $\Gamma$ rooted at $\tpl{x}$. Note that $Z_{\tpl{x}} = \sum_{F \in \mathcal{D}_\ell} Y_{H_F}$. We have 
	\begin{align*}
	\mathbb{P}[Z_{\tpl{x}} > 2p^{2\ell}n^{2\ell - (k-1)}] &\leq 
	\mathbb{P}\left[\left\lvert\sum_{F \in \mathcal{D}_\ell}(Y_{H_F} - \mathbb{E}[Y_{H_F}])\right\rvert > 2p^{2\ell}n^{2\ell - (k-1)} - \sum_{F \in \mathcal{D}_\ell} \mathbb{E}[Y_{H_F}]\right] \\
	&\overset{\text{(\ref{eq:Ebound})}}{\leq} \mathbb{P}\left[\sum_{F \in \mathcal{D}_\ell}\left\lvert Y_{H_F} - \mathbb{E}[Y_{H_F}]\right\rvert > \frac{1}{2}p^{2\ell}n^{2\ell - (k-1)}\right] \\
	&\leq \sum_{F \in \mathcal{D}_\ell} \mathbb{P}\left[|Y_{H_F} - \mathbb{E}[Y_{H_F}]| > \frac{p^{2\ell}n^{2\ell-(k-1)}}{2|\mathcal{D}_\ell|}\right] \\
	&= O\left(|\mathcal{D}_\ell| \exp(n^{-{\gamma/10\ell}})\right).
	\end{align*}
	Finally, the result follows by the union bound over all $(k-1)$-tuples $\tpl{x}$ in $V(\Gamma)$.
\end{proof}

We are now ready to prove Lemma~\ref{lem:goodexpansion}.
\begin{proof}[Proof of Lemma~\ref{lem:goodexpansion}]
	Let $\tpl{x}$ be a $(k-1)$-tuple in $V(\Gamma)$ and $\mathcal{P}$ be a set of at least $(\mu pn)^\ell$ tight paths in $\Gamma$ with $\ell + (k-1)$ vertices and rooted at $\tpl{x}$. Let $Q$ be the set of end-tuples we reach from $\tpl{x}$ with paths in $\mathcal{P}$. For each $\tpl{q} \in Q$, let $\mathcal{P}_{\tpl{q}}$ be those paths in $\mathcal{P}$ that end in $\tpl{q}$. Note that, for $\tpl{q} \in Q$ and distinct elements $P, P' \in \mathcal{P}_{\tpl{q}}$, we have $P \cup P' \in \mathcal{D}_\ell$. 
	Thus for each $\tpl{q} \in Q$, there are at least $\frac{1}{(2\ell)!}\binom{|\mathcal{P}_{\tpl{q}}|}{2}$ copies of elements of $\mathcal{D}_\ell$ in $\Gamma$ rooted at $\tpl{x}$ and ending in $\tpl{q}$ (we divide by $(2\ell)!$ since there are at most $(2\ell)!$ ways the union of two paths in $\mathcal{P}$ could result in the same copy of an element of $\mathcal{D}_\ell$). Hence the number of copies of elements of $\mathcal{D}_\ell$ in $\Gamma$ rooted at $\tpl{x}$ is at least 
	\[
	\sum_{\tpl{q} \in Q}\frac{1}{(2\ell)!} \binom{|\mathcal{P}_{\tpl{q}}|}{2} \geq \frac{|Q|}{(2\ell)!}\binom{\frac{(\mu pn)^\ell}{|Q|}}{2} \geq \frac{(\mu pn)^{2\ell}}{4(2\ell)!|Q|},
	\]
	where the penultimate inequality follows by Jensen's inequality.
	Thus, by Lemma~\ref{lem:D_ell}, we have a.a.s.\
	\[
	2p^{2\ell}n^{2\ell-(k-1)} \geq \frac{(\mu pn)^{2\ell}}{4(2\ell)!|Q|}
	\]
	and thus 
	\[
	|Q| \geq \frac{\mu^{2\ell}}{8(2\ell)!}n^{k-1}.
	\]
	Moreover, an analogous argument shows the result for spike paths.
\end{proof}
Together with the definition of tuples that are $(\varepsilon,p,\ell)$-good for $S$, Lemma~\ref{lem:goodexpansion} implies the following.

\begin{corollary}
	\label{cor:goodpaths}
	For any $\gamma>0$ and any $0<\eps \le \tfrac14\gamma$, and integers $s$, $k \ge 3$, and $\ell > \frac{k-1}{\gamma} + k-1$, there exists $\nu>0$ such that in $\Gamma=G^{(k)}(n,p)$ a.a.s.\ the following holds when $p=n^{-1+\gamma}$. Let $G\subseteq\Gamma$ satisfy $\delta_{k-1}(G)\ge\big(\tfrac12+\gamma\big)pn$.
	Let $S,S' \subseteq V(\Gamma)$ be sets with $|S|\le \tfrac12n$ and $|S'| \le s$. Let $\tpl{x}$ be a $(k-1)$-tuple, which is $(\eps,p,\ell)$-good for $S$. Then there are at least $\nu n^{k-1}$ different $(k-1)$-tuples $\tpl{y}$, such that there exists a tight path in $G$ of length $\ell$ with ends $\tpl{x}$ and $\tpl{y}$ and no vertices of the path in $S \cup S'$ except for possibly some of the vertices in $\tpl{x}$.
	Moreover, when $(k-1) | \ell$, the same holds for spike paths in $G$ of length $\ell$.
\end{corollary}

\begin{proof}
 We only prove the statement for tight paths as it is easy to see that the proof can be adapted for spike paths. We set $\mu=\tfrac14\gamma$, and $\nu=\tfrac{\mu^{2\ell}}{8(2\ell)!}$. Suppose that the good event of Lemma~\ref{lem:goodexpansion}, with input $\gamma$, $k$, and $\mu$, holds for $\Gamma=G^{(k)}(n,p)$.
 
 Given $G$ and $\tpl{x}$ as in the lemma statement, let $\tpl{x} = \big(x_1, \dots, x_{k-1}\big)$. We construct tight paths $x_1 \dots x_{\ell + k-1}$ rooted at $\tpl{x}$ by choosing vertices $x_{k}, \dots, x_{\ell + k-1}$ one by one as follows.
	For each $k \leq i \leq \ell + k-1$, we choose $x_i$ such that $x_i \not\in S \cup S' \cup \{x_1, \dots, x_{i-1}\}$ and  $\{x_{i-k+1}, \dots, x_i\} \in E(G)$. If $i < \ell + k-1$, we insist in addition that $\{x_{i-k+2}, \dots, x_i\}$ is $(\eps, p, \ell - (i-k+1))$-good for $S$. Since $\tpl{x}$ is $(\eps, p, \ell)$-good for $S$, for each $k \leq i \leq \ell + k-1$, the number of choices for each $x_i$, such that $\{ x_{i-k+1},\dots,x_i \}$ is an edge of $G$, $x_i \not\in S \cup S' \cup \{ x_1,\dots,x_{i-1} \}$, and $\{x_{i-k+2},\dots,x_i\}$ is $(\eps,p,\ell - (i-k+1))$-good, is at least 
	\begin{align*}
		\big(\tfrac12+\gamma\big)pn-(p\s{S} + \eps pn) - s -\ell-(k-1) - \eps pn \ge (\gamma-2 \eps)pn-s-\ell-(k-1)\ge\tfrac14\gamma p n\,.
	\end{align*}
	Let $\mathcal{P}$ be the set of tight paths constructed in this way; then we have $|\mathcal{P}|\ge\big(\tfrac14\gamma p n\big)^\ell=(\mu p n)^\ell$. Since the good event of Lemma~\ref{lem:goodexpansion} holds, the number of end $(k-1)$-tuples of these paths is at least $\tfrac{\mu^{2\ell}}{8(2\ell)!}n^{k-1}=\nu n^{k-1}$, as desired.
\end{proof}

\subsection{Connecting lemma}

The next lemma will enable us to connect two $(k-1)$-tuples, which are $(\eps',p,\ell)$-good for some set $S$, by a path of length at most $\ell$ avoiding $S$.

\begin{lemma}
\label{lem:connecting}
Given $k \ge 3$, and $\gamma>0$, there exists an integer $\ell$
such that for any integer $s$ the following holds.
For any $d,\eta>0$, any $0<\eps' \le \tfrac14 \gamma$, any integer $t_0$, any small enough $\nu, \eps_k>0$, any functions $\eps,f,f_k:\mathbb{N}\to(0,1]$ which tend to zero sufficiently fast, and any large enough $t_1,t_2\in\mathbb{N}$, the following holds a.a.s.\ in~$\Gamma=G^{(k)}(n,p)$ with $p \ge n^{-1+\gamma}$.
Suppose $G \subseteq \Gamma$ is an $n$-vertex $k$-graph with $\delta_{k-1}(G)\ge\big(\tfrac12+\gamma\big)pn$, that $(\cP^*_c,\cP^*_f)$ is a $(t_0,t_1,t_2,\eps_k,\eps(t_1),f_k(t_1),f(t_2),p)$-strengthened pair for $G$, and that $t$ is the number of $1$-cells in $\cP_c^*$.
Let $\cR' \subseteq \cR =\cR_{\eps_k,d}(G;\cP^*_c,\cP^*_f)$ be an induced subcomplex of the $(\eps_k,d)$-reduced multicomplex of $G$ on at least $(1-\nu)t$ $1$-edges and assume that it is tightly linked.
Further, let $S \subseteq V(G)$ be such that $|S| \le \tfrac{1}{2} n$ and it intersects all $1$-cells of $\cR'$ in at most an $(1-\eta) $-fraction.
Then for any two $(k-1)$-tuples $\tpl{x}$ and $\tpl{y}$, which are $(\eps',p,\ell)$-good for $S$,
and any set $S'$ of size at most $s$, there exists a tight path of length $\ell$ with ends $\tpl{x}$ and $\tpl{y}$.
\end{lemma}

To prove this we use the following lemma, which allows us to connect a fraction of any good $(k-1)$-cell to a fraction of an adjacent good $(k-1)$-cell, where adjacency is with respect to regular polyads.

\begin{lemma}
\label{lem:connectpartition}
Given $k \ge 3$ and $\gamma>0$, there exists an integer $\ell$ such that for any integer $s$ the following holds.
For any $d,\eta,\nu>0$, any $t_0\in\mathbb{N}$, any small enough $\eps_k>0$, any functions $\eps,f:\mathbb{N}\to(0,1]$ which tend to zero sufficiently fast, any integers $t_2 \ge t_1 \ge t_0$, and any small enough $f_k>0$, the following holds a.a.s.\ in $\Gamma=G^{(k)}(n,p)$ with $p \ge n^{-1+\gamma}$.
Suppose $G \subseteq \Gamma$ is an $n$-vertex $k$-graph, that $(\cP^*_c,\cP^*_f)$ is a $(t_0,t_1,t_2,\eps_k,\eps(t_1),f_k,f(t_2),p)$-strengthened pair for $G$, let $H:=\hat{P}(Q;\cP^*_c)$ be a regular polyad in the reduced complex $\cR_{\eps_k,d}(G;\cP^*_c,\cP^*_f)$, $V_1,\dots,V_k$ its underlying $1$-cells, and $S \subseteq V(G)$ is a set intersecting each of these in at most an $(1-\eta)$-fraction.
Further, let $E_1$ and $E_k$ be the two $(k-1)$-cells of $H$ missing $V_1$ and $V_k$, respectively.

Let the tuples in $E_1$ and $E_k$ be ordered according to $V_1,\dots,V_k$.
Then there is $\overline{E}_k\subset E_k$ with $|\overline{E}_k|\ge (1-\nu)|E_k|$, such that for any $\tpl{x}\in\overline{E}_k$ and any set $S'$ of at most $s$ vertices there is a tight path from $\tpl{x}$ to $\tpl{y}$ of length $\ell$ with internal vertices not in $S \cup S'$ for a $(1-\nu)$-fraction of the tuples $\tpl{y} \in {E_1}$.
\end{lemma}

With this lemma and Corollary~\ref{cor:goodpaths} it is straightforward to prove Lemma~\ref{lem:connecting}.
We will prove both, Lemma~\ref{lem:connecting} and~\ref{lem:connectpartition}, in Section~\ref{sec:proofconnecting}.

\subsection{Fractional matchings}

While the clusters of a regular partition are all the same size, and are still about the same size after we remove the reservoir set, the reservoir path may intersect the clusters in very different amounts. When we extend the reservoir path to an almost-spanning path, this means we need to use different numbers of vertices in the different clusters. To guide the construction of the almost-spanning path, the following lemma returns a fractional matching in the cluster graph such that the total weight on each cluster is at most the fraction of vertices still to use in that cluster, and the total weight of the fractional matching is very close to $\tfrac{1}{k}$ times the fraction of vertices in total still to use.

\begin{lemma}
	\label{lem:fractional}
	Let $H$ be an $m$-vertex $k$-complex, and let $w:V(H)\to[0,1]$ be a weight function. Given $\eps>0$, suppose that $H$ has at least $(1-\eps)m$ edges of size $1$, and that for each $1\le i\le k-2$, each $i$-edge of $H$ is contained in at least $(1-\eps)m$ edges of size $i+1$. Finally suppose that each $(k-1)$-edge of $H$ is contained in at least $\big(\tfrac12+\gamma\big)m$ edges of size $k$, and suppose $\sum_{v\in V(H)}w(v)\ge(1-\gamma)m$. Then there is a weight function $w^*:E\big(H^{(k)}\big)\to[0,1]$ such that for each $v\in V(H)$ we have $\sum_{e\ni v}w^*(e)\le w(v)$ and $\sum_{e\in E\big(H^{(k)}\big)}w^*(e)\ge\Big(\sum_{v\in V(H)}w(v)-\eps m\Big)\cdot\tfrac{1}{k}$.
\end{lemma}
\begin{proof}
 Consider the linear program
 \[\text{maximise }\sum_{e\in E(H^{(k)})}w^*(e)\quad\text{ subject to }\quad\sum_{e\ni v}w^*(e)\le w(v)\text{ for each $v \in V(H)$ and }w^*(e)\ge0\,.\]
 The dual program has variables $y:V(H)\to[0,1]$ such that for each $e\in E\big(H^{(k)}\big)$ we have $\sum_{v\in e}y(v)\ge 1$, where we minimise $\sum_{v\in V(H)}y(v)w(v)$. Suppose that $y$ is a feasible solution to the dual program.
 
 We order $V(H)$ according to decreasing $y$. We find a $k$-edge of $H$ as follows. We take the last $v_1$ such that $\{v_1\}$ is a $1$-edge of $H$. Then for each $2\le i\le k$ in succession, we choose the last vertex $v_i$ such that $\{v_1,\dots,v_i\}$ is an $i$-edge of $H$.
 
 For each $1\le i\le k-1$, since by construction $\{v_1,\dots,v_{i-1}\}$ is an $(i-1)$-edge of $H$, there are at most $\eps m$ choices of $v_i$ which do not give an $i$-edge of $H$, and in particular $v_i$ will be at or after position $(1-\eps)m$ in the order. Finally since $\{v_1,\dots,v_{k-1}\}$ is a $(k-1)$-edge of $H$, necessarily $v_k$ will be at position at or after $\big(\tfrac12+\gamma\big)m$ in the order.
 
 Suppose that the vertex $v$ of $H$ at position $(1-\eps)m$ in the order satisfies $y(v)=a$, and let $y(v_k)=b$. Then we have $(k-1)a+b\ge \sum_{i=1}^ky(v_i)\ge 1$, where the second inequality is since $y$ is feasible for the dual program. On the other hand, let $\alpha$ denote the sum of $w(u)$ over vertices $u$ equal to or earlier in the order than $v_k$, and let $\beta$ denote the sum of $w(u)$ over vertices $u$ after $v_k$ but not after $v$ (where $v$ is at position $(1-\eps)m$ in the order). Then we have $\sum_{v\in V(H)}w(v)y(v)\ge \alpha b+\beta a$.
 
 We view this as an optimisation problem: given $0\le a\le b\le 1$ such that $(k-1)a+b\ge1$, minimise $\alpha b+\beta a$. Trivially we can assume the minimum occurs for $(k-1)a+b=1$, and since $k\ge2$ and $\alpha>\beta$, the unique minimum occurs when $a=b=\tfrac{1}{k}$.
 
 Thus we have $\sum_{v\in V(H)}w(v)y(v)\ge(\alpha+\beta)\cdot\tfrac{1}{k}$ for any feasible solution $y$ to the dual program, so the value of the dual program is at least $(\alpha+\beta)\cdot\tfrac{1}{k}$. By the Duality Theorem for linear programming, the value of the primal program is the same. Finally since $w(v)\in [0,1]$ we have $\sum_{v\in V(H)}w(v)\le\alpha+\beta+\eps m$, and the lemma follows. 
\end{proof}

\subsection{Reservoir path}

\begin{definition}[Reservoir path]
A \emph{reservoir path} $\Pres$ with a \emph{reservoir set} $R\subsetneq V(\Pres)$ is an $k$-uniform hypergraph with two $(k-1)$-tuples $\tpl{v}$ and $\tpl{w}$, such that for \emph{any} $R'\subseteq R$, $\Pres$ contains a tight path with the vertex set $V(\Pres)\setminus R'$ and end-tuples $\tpl{v}$ and $\tpl{w}$.
\end{definition}

\begin{lemma}[Reservoir Lemma]\label{lem:respath}
	Given $k \ge 3$, $\gamma>0$, and $\ell' \in \mathbb{N}$, there exist an integer $c$, such that for $0<\eps' \le \tfrac14\gamma$, $0<d \le \tfrac18 \gamma$, large enough $t_0$, small enough $\nu,\eps_k>0$, any functions $\eps,f_k,f:\mathbb{N}\to(0,1]$ which tend to zero sufficiently fast and any integers $t_2 \ge t_1 \ge t_0$ the following holds a.a.s.\ in~$\Gamma=G^{(k)}(n,p)$ with $p \ge n^{-1+\gamma}$.
	Suppose $G \subseteq \Gamma$ with $\delta_{k-1}(G) \ge (\frac{1}{2}+\gamma) pn$, that $(\cP^*_c,\cP^*_f)$ is a $\big(t_0,t_1,t_2,\eps_k,\eps(t_1),f_k(t_1),f(t_2),p\big)$-strengthened pair for $G$, and that $t$ is the number of $1$-cells in $\cP^*_c$.
	Let $\cR=\cR_{\eps_k,d}(G;\cP^*_c,\cP^*_f)$ be the $(\eps_k,d)$-reduced multicomplex of $G$ and let $S$ be the union of the $1$-cells that are not in $\cR$.
 	Then given $R \subseteq V(G)$ with $|R| \le \nu n$ there exists a reservoir path $\Pres$ in $G$ with reservoir set $R$ and ends $\tpl{v}$ and $\tpl{w}$, such that  $\tpl{v}, \tpl{w}$ are $(\eps',p,\ell')$-good for $S \cup V(\Pres)$ and $|V(\Pres)| \le c |R|$.
\end{lemma}

We prove Lemma~\ref{lem:respath} in Section~\ref{sec:proofreservoir}.

\section{Proof of Theorem~\ref{thm:main}}
\label{sec:proofmain}

\begin{proof}
 Given $\gamma>0$ and $k\ge3$, let $\ell_\sublem{lem:connectpartition} \ge \tfrac{k-1}{\gamma}+k$ be returned by Lemma~\ref{lem:connectpartition} for input $k$, $0$, and $\gamma$.
 Similarly, let $\ell_\sublem{lem:connecting}$ be given by Lemma~\ref{lem:connecting} with input $k$ and $\tfrac12\gamma$.
 Let $\nu_\subcor{cor:goodpaths}$ be returned by Corollary~\ref{cor:goodpaths} for input $\gamma$, $\eps=\tfrac14\gamma$, $s=k$, $k$, and $\ell_\sublem{lem:connectpartition}$ and let $\nu_\sublem{lem:connectpartition}=\tfrac14\nu_\subcor{cor:goodpaths}$.
 Let $c$ be the integer returned by Lemma~\ref{lem:respath} for input $\gamma$, $k$ and $\ell_\sublem{lem:connectpartition}$ and then let $\tfrac18\gamma \ge d>0$.
 Let $t_0\ge k! \nu_\subcor{cor:goodpaths}^{-2}$ be sufficiently large for Lemma~\ref{lem:respath} with input as above, $\eps'=\tfrac14\gamma$ and $d$, for Lemma~\ref{lem:connecting} with input as above and $\eta=\tfrac12$, $\eps'=\tfrac18\gamma$, and $s=3 \ell_\sublem{lem:connecting} + 3k$, and for Lemma~\ref{lem:goodconnected} with input $k$ and $d$.

 We then choose $\nures<\tfrac{\gamma}{8c}$ such that $2\nures$ is sufficiently small for Lemma~\ref{lem:respath} with the given input and $\nu_\sublem{lem:connecting}>0$
 is small enough for Lemma~\ref{lem:connecting} with the given input. We let $\eta_\sublem{lem:connectpartition}=10^{-6}\ell_\sublem{lem:connecting}^{-1}\nures \nu_\sublem{lem:connecting}$.
 Next we choose $\eps_k\le 10^{-6}k^{-k}\nu_\subcor{cor:goodpaths}^k\nures^k\eta_\sublem{lem:connectpartition}^k$ small enough for Lemma~\ref{lem:connectpartition} with input as above and $s=\ell_\sublem{lem:connectpartition}$, $d$, $\eta_\sublem{lem:connectpartition}$, $\nu_\sublem{lem:connectpartition}$, and $\eps'=\tfrac18\gamma$, for Lemma~\ref{lem:respath} with input as above, and also such that $2\nures^{-k}\eps_k$ is small enough for Lemma~\ref{lem:connecting} with input as above.
 
 We choose functions $\eps,f_k,f:\mathbb{N}\to(0,1]$ such that $\sqrt{\eps}$, $2\nures^{-k}f_k$ and $\sqrt{f}$ are all smaller than $\eps_k$, small enough for each of Lemmas~\ref{lem:connecting},~\ref{lem:connectpartition} and~\ref{lem:respath} with the above inputs, and that for each $t$, both $\eps(t)$ and $f(t)$ are small enough for Lemma~\ref{lem:RRL} with input $k$, $\alpha=\tfrac12\nures$ and $d_0=\tfrac{1}{2t}$.
 Let $\eta_\sublem{lem:ssshrl}$ and $T_\sublem{lem:ssshrl}$ be returned by Lemma~\ref{lem:ssshrl} for input $k$, $t_0$, $s=1$, $\eps_k$, $\eps$, $f_k$, $f$.

 Given $n$, let $p\ge n^{-1+\gamma}$. Let $\tilde{L}$ be a set of at most $T_\sublem{lem:ssshrl}!-1$ vertices in $[n]$ such that $n-|\tilde{L}|$ is divisible by $T_\sublem{lem:ssshrl}!$. Suppose that $\tilde{\Gamma}=G^{(k)}(n,p)$ and its induced subgraph $\Gamma=\tilde{\Gamma}-\tilde{L}$ are in the good events of Corollary~\ref{cor:goodpaths}, Lemmas~\ref{lem:connectpartition} and~\ref{lem:respath} with inputs as above and Lemma~\ref{lem:goodtuples} with input $\tfrac14\gamma\nures$ and $k$.
 Suppose that $\Gamma$ and all its subgraphs are $(\eta_\sublem{lem:ssshrl},p)$-upper regular, which by Lemma~\ref{lem:upperreg} holds a.a.s. In addition, suppose that $\Gamma$ satisfies the following: if $R$ is a set of vertices chosen independently with probability $\nures$ from $V(\Gamma)$, then a.a.s.~$\Gamma[R]$ is in the good event of Lemma~\ref{lem:connecting} with input as above.
 Note that this last event occurs a.a.s.\ for the following reason: if we first choose $R$ randomly then expose the edges of $\Gamma$, a.a.s.\ we obtain a set $R$ of size $\big(1\pm\tfrac12\big)\nures n$, and given this Lemma~\ref{lem:connecting} states that a.a.s.~$\Gamma[R]$ will be in the good event. Thus the probability of obtaining a pair $(R,\Gamma)$ such that $\Gamma[R]$ is not in the good event of Lemma~\ref{lem:connecting} is $o(1)$, and it follows that, for any $\iota>0$, the probability of choosing $\Gamma$ such that $\big($ $\Gamma[R]$ has probability at least $\iota$ of not being in the good event of Lemma~\ref{lem:connecting} $\big)$, is $o(1)$.
 
 Given $\tilde{G}\subseteq\Gamma$ with $\delta_{k-1}(\tilde{G})\ge\big(\tfrac12+2\gamma\big)pn$, we remove $\tilde{L}$ to obtain an induced subgraph $G$ with $T_{\sublem{lem:ssshrl}}!|v(G)$. Observe that $\delta_{k-1}(G)\ge\big(\tfrac12+\gamma\big)pn$. We apply Lemma~\ref{lem:ssshrl} to $G$, with input as above, to obtain a $\big(t_0,t_1,t_2,\eps_k,\eps(t_1),f_k(t_1),f(t_2),p\big)$-strengthened pair $(\cP^*_c,\cP^*_f)$, where $t_0\le t_1\le t_2\le T_\sublem{lem:ssshrl}$. Let $t$ be the number of clusters of $\cP_c$; by definition we have $t_0\le t\le t_1$. Applying Lemma~\ref{lem:goodconnected}, we see that the $(\eps_k,d)$-reduced multicomplex $\cR$ of $G$, with respect to this strengthened pair, has at least $\big(1-4\eps_k^{1/k}\big)t$ $1$-edges, every $(k-1)$-edge of $\cR$ is contained in at least
 \[\big(\tfrac12+\gamma-2d-2^{k+2}\eps_k^{1/k}\big)t\prod_{i=2}^{k-1}d_i^{-\binom{k-1}{i-1}}\ge\big(\tfrac12+\tfrac12\gamma\big)t\prod_{i=2}^{k-1}d_i^{-\binom{k-1}{i-1}}\]
 $k$-edges, and every induced subcomplex of $\cR$ on at least $\big(1-\gamma+2d+2^{k+2}\eps_k^{1/k}\big)t<\big(1-\tfrac12\gamma\big)t$ vertices is tightly linked.
 
 We choose a subset $R$ of $[n]$ by selecting vertices uniformly at random with probability $\nures$. A.a.s.\ we have $|R|=\big(1+o(1)\big)\nures n$. By Chernoff's inequality and the union bound, a.a.s.\ for each $V$ which is a part of either $\cP_c$ or $\cP_f$, we have $|V\cap R|=\big(1\pm o(1)\big)\nures|V|$. Furthermore, for each $S$ which is the neighbourhood in $\tilde{G}$ or in $\Gamma$ of some $(k-1)$-set of vertices, we have $|S\cap R|=\big(1\pm o(1)\big)\nures|S|$ (recall that any such set $S$ has size at least $\tfrac12pn\ge n^{\gamma/2}$). Finally, by our assumption on $\Gamma$, we have a.a.s.\ that $\Gamma[R]$ is in the good event of Lemma~\ref{lem:connecting} with input as above. Suppose that $R$ is such that all of these likely events occur.
 
 We apply Lemma~\ref{lem:respath}, with inputs as above, to find a reservoir path $\Pres$ in $G$ with reservoir set $R$ whose ends are $\tpl{v}_\mathrm{res}$ and $\tpl{w}_\mathrm{res}$, such that $\tpl{v}_\mathrm{res}$ and $\tpl{w}_\mathrm{res}$ are both $(\tfrac14\gamma,p,\ell_\sublem{lem:connectpartition})$-good for $S\cup V(\Pres)$, where $S$ is the union of all $1$-cells not in $\cR$, and such that $\big|V(\Pres)\big|\le c|R|\le\tfrac18\gamma n$. 
 
  We now aim to extend $\Pres$, from its end $\tpl{w}_\mathrm{res}$, to a path $\Palm$ covering almost all vertices of $G$. To begin with, let $\cR'$ denote the complex on $V(\cR)$ obtained by letting $e'$ be an edge of $\cR'$ whenever there is an edge $e$ of $\cR$ such that $\mathrm{vertices}(e)=e'$. Thus the $1$-edges of $\cR$ and $\cR'$ are identical, and it follows inductively from the definition of an ($\eps_k,d)$-reduced multicomplex that for each $1\le i\le k-2$, each $i$-edge of $\cR'$ is contained in at least $\big(1-2^{i+2}\eps_k^{1/k}\big)t$ $(i+1)$-edges, and each $(k-1)$-edge of $\cR'$ is contained in at least $\big(\tfrac12+\tfrac12\gamma\big)t$ $k$-edges. We define a weight function $\omega$ on $V(\cR')$ as follows. Given a cluster $V_i\in V(\cR')$, if $|V_i\setminus V(\Pres)|<2\eta_\sublem{lem:connectpartition}\tfrac{n}{t}$, we set $\omega(V_i)=0$. Otherwise, we set
 \[\omega(V_i)=\frac{|V_i\setminus V(\Pres)|-2\eta_\sublem{lem:connectpartition}\tfrac{n}{t}}{(1-\nures)\tfrac{n}{t}}\,.\]
 Note that since $|V_i\setminus V(\Pres)|\le|V_i\setminus R|\le \big(1+o(1)\big)(1-\nures)\tfrac{n}{t}$, this weight function takes values in $[0,1]$. Furthermore, we have
 \[\sum_{V_i\in V(\cR')}\omega(V_i)=\frac{n-|V(\Pres)|- 2\eta_\sublem{lem:connectpartition} n}{(1-\nures)\tfrac{n}{t}}>\big(1-\tfrac12\gamma\big)t\,.\]
 This is the required setup to apply Lemma~\ref{lem:fractional}, with input $2^{k+2}\eps_k^{1/k}$ and $\tfrac12\gamma$. The result is a weight function $\omega^*:E(\cR'^{(k)})\to[0,1]$ such that for each $V_i\in V(\cR')$ we have $\sum_{e\ni V_i}\omega^*(e)\le\omega(V_i)$, and
 \begin{align*}
  \sum_{e\in\cR'^{(k)}}\omega^*(e)&\ge\tfrac{1}{k}\big(\sum_{V_i\in V(\cR')}\omega(V_i)-2^{k+2}\eps_k^{1/k}t\big)>\tfrac{1}{k}\cdot\frac{n-|V(\Pres)|- 2\eta_\sublem{lem:connectpartition} n-2^{k+3}\eps_k^{1/k}n}{(1-\nures)\tfrac{n}{t}}\\
  &>\tfrac{1}{k}\cdot\frac{n-|V(\Pres)|- 3\eta_\sublem{lem:connectpartition} n}{(1-\nures)\tfrac{n}{t}}\,.
 \end{align*}
 
 Recall that $\cR$ is tightly linked, even if an arbitrary set of $\tfrac12\gamma t$ vertices is removed. If a cluster of $\cP_c$ has $\omega$-weight zero, then it contains at least $\tfrac{n}{2t}$ vertices of $\Pres$, so there are at most $\tfrac{2c\nures n\cdot 2t}{n}=4c\nures t\le\tfrac12\gamma t$ clusters with $\omega$-weight zero. In particular, the submulticomplex of $\cR$ induced by removing clusters of $\omega$-weight zero is tightly linked.
 
  We next construct a path $\Palm$ extending $\Pres$ from $\tpl{w}_\mathrm{res}$ as follows. Recall that $\tpl{w}_\mathrm{res}$ is $\big(\tfrac14\gamma,p,\ell_\sublem{lem:connectpartition} \big)$-good for $S \cup V(\Pres)$. To begin with, we use Corollary~\ref{cor:goodpaths} to obtain a collection of $(k-1)$-tuples, of size at least $\nu_\subcor{cor:goodpaths}n^{k-1}$, each of which is the end-tuple of a path of length $\ell_\sublem{lem:connectpartition}$ starting at $\tpl{w}_\mathrm{res}$ whose vertices, other than those in $\tpl{w}_\mathrm{res}$, are disjoint from $V(\Pres)$. Note that all these tuples are by construction outside $V(\Pres)$ and so also outside $R$. By definition of a strengthened pair and $(\eps_k,d)$-reduced multicomplex, and choice of $t_0$ and $\eps_k$, at least half of these end-tuples are contained in $(k-1)$-cells of $\cP^*_c$ which are in $\cR$. In particular, by averaging there is a $(k-1)$-cell of $\cR$, with clusters in a given order, $f_0$, such that at least a $\tfrac12\nu_\subcor{cor:goodpaths}$-fraction of these end-tuples are in $f_0$ in the given order. Let $P_0=\Pres$, and let $Q_0$ denote the set of $(k-1)$-tuples in $f_0$ which are ends of paths of length $\ell_\sublem{lem:connectpartition}$ starting from $\tpl{w}_0:=\tpl{w}_\mathrm{res}$ whose vertices outside $\tpl{w}_0$ are disjoint from $P_0$.
  
  We order arbitrarily the $k$-edges of $\cR'$ with positive $\omega^*$-weight, and for each $j$, let $g_j$ be a $k$-edge of $\cR$ whose vertices are the same as the $j$th $k$-edge of $\cR'$; we let $\omega^*(g_j)$ be given by $\omega^*$ at the $j$th edge of $\cR'$. We now create a sequence $e_1,\dots$ of $k$-edges of $\cR$ as follows. To begin with, we choose a tight link in $\cR$ from $f_0$ to a $(k-1)$-tuple in $g_1$ using only clusters of positive weight, and we let the first edges $e_1,\dots$ be the edges of a homomorphic copy of a minimum length tight path following this tight link. We then repeat $g_1$ in the sequence
  \[\Big\lceil\frac{k(1-\nures)\tfrac{n}{t}\cdot\omega^*(g_1)}{\ell_\sublem{lem:connectpartition}} \Big\rceil\]
  times, follow a tight link to $g_2$, and so on. When we follow a tight link, we always do so such that the edges $e_1,\dots$ form a homomorphic copy of a tight path in $\cR$, using only vertices whose weight according to $\omega$ is positive; this is possible since the vertices of each $g_j$ have weight at least $\omega^*(g_j)>0$, and since the positive-weight induced submulticomplex of $\cR$ is tightly linked. Note that the number of repetitions of $g_1$ fixes the ordered $(k-1)$-cell in the boundary of $g_1$ from which we follow a tight link to $g_2$, and so on.
  
  Since $\cR$ is a bounded size multicomplex --- it contains in total at most $t_1^k\cdot t_1^{\binom{k}{2}}\dots t_1^{\binom{k}{k-1}}\le t_1^{2^k}$ edges of size $k$
  --- the total number of edges $e_i$ used in following tight links is at most $4k^3\cdot t_1^{2^{k+1}}$.
 
 We now use the following procedure repeatedly for $i\ge1$. We are given $P_{i-1}$ which is a path from $\tpl{v}_\mathrm{res}$ to $\tpl{w}_{i-1}$, an ordered $(k-1)$-cell $f_{i-1}$ of $\cR$ (which is contained in $e_{i-1}$ and also in $e_i$), and a set $Q_{i-1}$ of $(k-1)$-tuples in $f_{i-1}$ which are ends of paths of length $\ell_\sublem{lem:connectpartition}$ from $\tpl{w}_{i-1}$ whose vertices outside $\tpl{w}_{i-1}$ are disjoint from $P_{i-1}$. We suppose $Q_{i-1}$ contains at least a $2\nu_\sublem{lem:connectpartition}$-fraction of the $(k-1)$-tuples in $f_{i-1}$. We let $f_i$ be the $(k-1)$-cell in the boundary of $e_i$ on the last $k-1$ clusters of $e_i$, with the order inherited from $e_i$.
 
 By Lemma~\ref{lem:connectpartition}, with input as above, and $S=V(P_{i-1})$, and choice of $\nu_\sublem{lem:connectpartition}$ there is a tuple $\tpl{w}_i$ in $Q_{i-1}$ such that the following holds. Let $P_i$ denote the extension of $P_{i-1}$ to $\tpl{w}_i$ by adding a path of length $\ell_\sublem{lem:connectpartition}$ witnessing $\tpl{w}_i\in Q_{i-1}$; let $S'$ be the vertices $V(P_i)\setminus V(P_{i-1})$. There is a set $Q_i$ of $(1-\nu_\sublem{lem:connectpartition})$-fraction of the tuples of $f_i$, each of which is the end of a path of length $\ell_\sublem{lem:connectpartition}$ from $\tpl{w}_i$, whose vertices outside $\tpl{w}_i$ are disjoint from $V(P_{i-1})$ and from $S'$. Note that this is the setup required to iterate the application of Lemma~\ref{lem:connectpartition}, provided that we ensure that at no stage does $S=V(P_{i-1})$ intersect any cluster of $e_i$ in more than a $(1-\eta_\sublem{lem:connectpartition})$-fraction. This is guaranteed for the following reason. Given a cluster $V_j$ of $\cP_c$, if $\omega(V_j)=0$ then $V_j$ is not a vertex of any $e_i$. If on the other hand $\omega(V_j)>0$, then the total number of vertices used in $V_j$ is at most
 \[\ell_\sublem{lem:connectpartition} \cdot 4k^3\cdot t_1^{2^{k+1}} + 2\ell_\sublem{lem:connectpartition}\cdot\binom{t_1}{k-1} + \tfrac{\ell_\sublem{lem:connectpartition}}{k} \cdot \frac{k(1-\nures)\tfrac{n}{t}}{\ell_\sublem{lem:connectpartition}} \cdot \sum_{g_i\ni V_j}\omega^*(g_i)\,,\]
 where the first term counts vertices used in following tight links, the second accounts for the rounding up in the weighting at each edge $g_i$ containing $V_j$ and the (at most) one vertex per $g_i$
extra since the tight path may use one more vertex in some clusters than others (since $\tfrac{\ell_\sublem{lem:connectpartition}}{k}$ may not be an integer). Note that these first two terms are bounded from above by a constant. Since $\sum_{g_i\ni V_j}\omega^*(g_i)\le\omega(V_j)$, we see that the number of vertices used in $V_j$ is at most
 \[O(1)+ \tfrac{\ell_\sublem{lem:connectpartition}}{k} \cdot \frac{k(1-\nures)\tfrac{n}{t}}{\ell_\sublem{lem:connectpartition}} \cdot \omega(V_j) = O(1)+|V_j\setminus V(\Pres)|-2\eta_\sublem{lem:connectpartition}\tfrac{n}{t} \le \big|V_j\setminus V(\Pres)\big|-\eta_\sublem{lem:connectpartition}\tfrac{n}{t}\,,\]
 and in particular at all times at least $\eta_\sublem{lem:connectpartition}\tfrac{n}{t}$ vertices remain in $V_j$. We let $\Palm$ denote the final tight path from $\tpl{v}_\mathrm{res}$ to $\tpl{w}_\mathrm{alm}$ obtained by this procedure.
 
 Observe that, just counting repetitions of the $g_i$, the total number of vertices $\big|V(\Palm)\setminus V(\Pres)\big|$ is at least
 \begin{align*}
  \ell_\sublem{lem:connectpartition} \cdot\tfrac{k(1-\nures)\tfrac{n}{t}}{\ell_\sublem{lem:connectpartition}}\cdot \sum_{e\in\cR'^{(k)}}\omega^*(e) & > k(1-\nures)\tfrac{n}{t}\cdot \tfrac{1}{k}\cdot\frac{n-|V(\Pres)|- 3\eta_\sublem{lem:connectpartition} n}{(1-\nures)\tfrac{n}{t}}\\
  &= n-|V(\Pres)|-3\eta_\sublem{lem:connectpartition}n\,.
 \end{align*}
 It follows that $n-\big|V(\Palm)\big|\le 3\eta_\sublem{lem:connectpartition}n$. Let $L=\big(V(G)\setminus V(\Palm)\big)\cup\tilde{L}$. Recall that $\tilde{L}$ is the set of at most $T_\sublem{lem:ssshrl}!-1$ vertices we removed from $\tilde{G}$ in order to guarantee the required divisibility condition.
 
 Our final task is to extend $\Palm$, re-using some vertices of $R$, to cover the vertices of $L$ and connect the ends.
 Critically, observe that $|L|$ is much smaller than $|R|$, and that by assumption on $\Gamma$ and $R$, the good event of Lemma~\ref{lem:connecting} holds for $\Gamma[R]$, for the input given at the start of the proof. Recall that $G[R]$ has minimum codegree at least $\big(\tfrac12+\tfrac12\gamma\big)p|R|$. Let $\cP^*_{cr}$ and $\cP^*_{fr}$ denote the families of partitions obtained from $\cP^*_c$ and $\cP^*_f$ respectively by reducing each cell to only those elements contained in $R$. By Lemma~\ref{lem:RRL} and choice of $\eps_k$, $\eps$, $f_k$ and $f$, $(\cP^*_{cr},\cP^*_{fr})$ is a $(t_0,t_1,t_2,2\nures^{-k}\eps_k,\sqrt{\eps(t_1)},2\nures^{-k}f_k,\sqrt{f(t_2)},p)$-strengthened pair for $G[R]$. Let $\cR_r$ denote the multicomplex obtained from $\cR$ by replacing the cells of $\cP^*_c$ with those of $\cP^*_{cr}$. Note that $\cR_r$ is still the $(\eps_k,d)$-reduced multicomplex for this strengthened pair, so it is contained in the $(2\nures^{-k}\eps_k,d)$-reduced multicomplex.
 
 Let $S_{-1}=\emptyset$.
 We now construct for $i=0,1,\dots$ two disjoint tight paths $P_{v,i}$ and $P_{w,i}$ and $S_i=V(P_{v,i}) \cup V(P_{v,i})$, where one end of $P_{v,i}$ is $\rev{\tpl{v}_\mathrm{res}}$ and the other, $\tpl{v}_{i}$, 
 is $(\tfrac14\gamma\nures,p,\ell_\sublem{lem:connecting})$-good for $S_{i-1}$, and $P_{v,i}$ contains $i$ vertices of $L$ and all other vertices, except those of $\tpl{v}_\mathrm{res}$, are in $R$.
 Similarly one end of $P_{w,i}$ is $\tpl{w}_\mathrm{alm}$ and the other, $\tpl{w}_{i}$, is $(\tfrac14\gamma\nures,p,\ell_\sublem{lem:connecting})$-good for $S_{i-1}$, and $P_{w,i}$ contains $i$ vertices of $L$, not in $P_{v,i}$, and all other vertices, except those of $\tpl{w}_\mathrm{alm}$, are in $R$. We do this as follows. To begin with, we find a tight path $P_{v,0}$ of length $k-1$, one of whose end tuples is $\rev{\tpl{v}_\mathrm{res}}$ and the other of which, $\tpl{v}_{0}$, is contained in $R$.
 Recall that every $(k-1)$-set in $V(G)$ contains at least $\big(\tfrac12+\tfrac12\gamma\big)p|R|$ edges of size $k$ with the extra vertex in $R$, so in particular we can greedily build the required path of length $k-1$. We construct $P_{w,0}$ from $\tpl{w}_\mathrm{alm}$ to $\tpl{w}_{0}$ similarly.
 Observe that, by definition, both $\tpl{v}_{0}$ and $\tpl{w}_{0}$ are $(\tfrac14\gamma\nures,p,\ell_\sublem{lem:connecting})$-good for $S_{-1}$.
 
 Now suppose $i\ge1$ and that we have constructed tight paths $P_{v,i-1}$ and $P_{w,i-1}$ as above, whose ends $\tpl{v}_{i-1}$ and $\tpl{w}_{i-1}$ are both $(\tfrac14\gamma\nures,p,\ell_\sublem{lem:connecting})$-good for $S_{i-2}$, and we have $|S_{i-1}|\le 4 (i-1) \ell_\sublem{lem:connecting}$.
 We first extend $P_{v,i-1}$ to a path $P_{v,i}$ as follows.
 We choose any $u \in L\setminus S_{i-1}$ and vertices $v_1,\dots,v_{k-2}$ from $R \setminus S_{i-1}$ such that the tuple $(u,v_1,\dots,v_{k-2})$ is $(\tfrac14\gamma\nures,p,\ell_\sublem{lem:connecting}+k-1)$-good for $S_{i-1}$.
 This step always succeeds, as $o(n)$ of these tuples are not $(\tfrac14\gamma\nures,p,\ell_\sublem{lem:connecting}+k-1)$-good for $S_{i-1}$, by the good event of Lemma~\ref{lem:goodtuples} assumed above.
 Then we can easily choose additional vertices $v_{k-1}, u_1,\dots,u_{k-1}$ from $R \setminus S_{i-1}$ such that for $j=1,\dots,k$ there is a $k$-edge $\{ u_j,\dots,u_{k-1},u,v_{1},\dots,v_{j-1} \}$ and the tuples $\tpl{v}_{i}=(v_{1},\dots,v_{k-1})$ and $\tpl{u} = (u_{1},\dots,u_{k-1})$ are $(\tfrac14\gamma\nures,p,\ell_\sublem{lem:connecting})$-good for $S_{i-1}$.
 For example, there are at least $(\tfrac12+\tfrac14\gamma)p|R|$ edges $\{u,v_1,\dots,v_{k-1} \}$ in $G$ with $v_{k-1} \in R$, of which at most $p|S| + \tfrac14\gamma\nures pn \le \tfrac14 p|R|$ have $v_{k-1} \in S_{i-1}$ and at most $\tfrac14\gamma\nures pn$ are such that $(v_1,\dots,v_{k-1})$ is not $(\tfrac14\gamma\nures,p,\ell_\sublem{lem:connecting}+k-2)$-good for $S_{i-1}$.
 
 Next, with Lemma~\ref{lem:connecting}, we connect $\tpl{v}_{i-1}$ to $\tpl{u}$ with a tight path of length $\ell_\sublem{lem:connecting}$ and internal vertices not in $S_{i-1}$.
 Note that here we added the set $S'$ containing the vertices $V(P_{v,i-1}) \setminus V(P_{v,i-2})$,  $V(P_{w,i-1}) \setminus V(P_{w,i-2})$, and $\{ u,v_1,\dots,v_{k-1} \}$ and that $|S'| \le 2 \ell_\sublem{lem:connecting} + 2k$.
 To see that the conditions of Lemma~\ref{lem:connecting} are satisfied, recall that $|S_{i-1}|\le 4 (i-1) \ell_\sublem{lem:connecting} \le 6 \eta_\sublem{lem:connectpartition} \ell_\sublem{lem:connecting} n$.
 By the choice of $\eta_\sublem{lem:connectpartition}$, this is at most $\frac14|R|$ and there can bet at most $\nu_\sublem{lem:connecting}t$ $1$-cells in $\cR_r$ which intersect $S_{i-1}$ in at least a $\tfrac12$-fraction.
 We then let $P_{v,i}$ be the path obtained by concatenating $P_{v,i-1}$, the path from $\tpl{v}_{i-1}$ to $\tpl{u}$, and the path from $\tpl{u}$ via $u$ to $\tpl{v}_{i}$.
 
 If there remain uncovered vertices in $L$, we repeat the same procedure to extend $P_{w,i-1}$ to $P_{w,i}$, where in the last step we also add the vertices from $V(P_{v,i}) \setminus V(P_{v,i-1})$ to $S'$ and get $|S'| \le 3 \ell_\sublem{lem:connecting} + 3k$.
 Note that afterwards with $S_i = V(P_{v,i})\cup V(P_{w,i})$ we have $|S_i|\le 4 i \ell_\sublem{lem:connecting}$ and all conditions of $P_{v,i}$ and $P_{w,i}$ needed for the next iterations are satisfied.
 We stop this procedure as soon as all vertices of $L$ are used; we let $P_v$ denote the final $P_{v,i}$ with end tuple $\tpl{v}=\tpl{v}_{i}$, and $P_w$ denote either $P_{w,i}$ or $P_{w,i-1}$ (depending on whether $|L|$ is even or odd, respectively) with end tuple $\tpl{w}$ either $\tpl{w}_{i}$ or $\tpl{w}_{i-1}$, respectively. Finally we make a last use of Lemma~\ref{lem:connecting} to find a tight path in $R$ whose interior vertices are disjoint from $V(P_v)\cup V(P_w)$, and whose ends are $\rev{\tpl{v}}$ and $\tpl{w}$. This is possible for the same reasons as above. Concatenating these three tight paths, we obtain a tight path $\Pcover$ whose end tuples are $\tpl{w}_\mathrm{alm}$ and $\tpl{v}_\mathrm{res}$, such that $L\subset V(\Pcover)$, and such that all interior vertices of $\Pcover$ are contained in $L\cup R$.
 
 Let $R'$ denote the set of vertices $V(\Pcover)\cap R$. By the reservoir property of $\Pres$, there is a tight path $\Pres^*$ whose end tuples are identical to $\Pres$ and whose vertex set is $V(\Pres)\setminus R'$. We replace $\Pres$ with $\Pres^*$ in $\Palm$ to obtain a tight path $\Palm^*$ whose end tuples are identical to those of $\Palm$ and whose vertex set is $V(\Palm)\setminus R'=V(\tilde{G})\setminus (L\cup R')$. Concatenating $\Palm^*$ and $\Pcover$, we obtain the desired tight Hamilton cycle in $\tilde{G}$. 
\end{proof}

\section{Connecting within the partition}
\label{sec:proofconnecting}

In this section we prove Lemma~\ref{lem:connecting} and~\ref{lem:connectpartition}.
For the first, the strategy is to expand from the tuples $\tpl{x}$ and $\tpl{y}$ using Corollary~\ref{cor:goodpaths} and then connecting two of the many ends that we found with Lemma~\ref{lem:connectpartition} by following a tight link given by Lemma~\ref{lem:goodconnected}.

\begin{proof}[Proof of  Lemma~\ref{lem:connecting}]
Let $k \ge 3$ and $\gamma>0$.
Further let $\ell_\subcor{cor:goodpaths}$ be the smallest integer exceeding $\tfrac{k-1}{\gamma}+k-1$, and let $\ell_\sublem{lem:connectpartition}$ be given by Lemma~\ref{lem:connectpartition} on input $k$ and $\gamma$.
Let $h$ be the shortest length of a tight path which admits a homomorphism to the edges of a tight link with first $k-1$ vertices going to the start $(k-1)$-tuple of the tight link in order and last $k-1$ vertices going to the end $(k-1)$-tuple of the tight link; let $\rho$ be such that they are in the order $\rho$.
We let $\ell := h \ell_\sublem{lem:connectpartition} + 2 \ell_\subcor{cor:goodpaths} + 2$.
Further let $s$ be any integer, and set $s_\subcor{cor:goodpaths} = s + \ell_\sublem{lem:connectpartition}$ and $s_\sublem{lem:connectpartition}=\ell$.
Then let $d,\eta>0$, $0<\eps' \le \tfrac12\gamma$, and $\nu_\subcor{cor:goodpaths}$ be given by Corollary~\ref{cor:goodpaths} on input with $\gamma$, $\eps'$, $s_\subcor{cor:goodpaths}$, $k$, and $\ell_\subcor{cor:goodpaths}$.
Next, let $\eta_\sublem{lem:connectpartition}, \nu_\sublem{lem:connectpartition},\nu>0$ and an integer $t_0$ be such that
\begin{align*}
\eta_\sublem{lem:connectpartition} < \tfrac12\eta, \quad 6 \nu_\sublem{lem:connectpartition} < \eta_{\sublem{lem:connectpartition}}^{k-1}, \quad (k-1)! \nu_{\sublem{lem:connectpartition}} \le \tfrac12 \nu_{\subcor{cor:goodpaths}}, \quad \nu \le \tfrac14 \nu_{\subcor{cor:goodpaths}}, \quad \text{and} \quad 
\tfrac{k}{t_0} \le \tfrac18 \nu_{\subcor{cor:goodpaths}}.
\end{align*}
Then let $\eps_k>0$, functions $\eps, f : \mathbb{N} \mapsto (0,1])$, integers $t_1 ,t_2$, and $f_k>0$ be such that Lemma~\ref{lem:connectpartition} is applicable with input as above and $s_\sublem{lem:connectpartition}$, $d$, $\eta_\sublem{lem:connectpartition}$, $\nu_\sublem{lem:connectpartition}$, and $t_0$.
We additionally require that $\eps(t_1)$ is small enough for 
Lemma~\ref{lem:DCL} with input $k$, $\alpha = \eta_{\sublem{lem:connectpartition}}$, $\gamma_c=\tfrac12$, and $d_0=\tfrac{1}{t_0}$.
 
Given $n$, let $p \ge n^{-1+\gamma}$.
Suppose that $\Gamma=G^{(k)}(n,p)$ is in the good events of Corollary~\ref{cor:goodpaths} and Lemma~\ref{lem:connectpartition}.
Let $G \subseteq \Gamma$ be an $n$-vertex $k$-graph with $\delta_{k-1}(G) \ge (\tfrac{1}{2}+\gamma) pn$.
Further, suppose that $(\cP^*_c,\cP^*_f)$ is a $(t_0,t_1,t_2,\eps_k,\eps(t_1),f_k,f(t_2),p)$-strengthened pair for $G$, let $\cR=\cR_{\eps_k,d}(G;\cP^*_c,\cP^*_f)$ be the $(\eps_k,d)$-reduced multicomplex of $G$, and let $t$ be the number of $1$-cells.
Then let $\cR' \subseteq \cR$ be an induced subcomplex of $\cR$ on at least
$(1-\nu)t$ $1$-edges, assume that it is tightly linked,  and fix the coarse density vector $\tpl{d}=(d_{k-1},\dots,d_2)$ with $d_i \ge t_1^{-1}$ for $i=1,\dots,k-1$.
Further, let $S \subseteq V(G)$ with $|S| \le \tfrac{1}{2} n$ be such that it intersects every $1$-cell of $\cR'$ in at most an $(1-\eta)$-fraction.
Next, let $\tpl{x}=(x_1,\dots,x_{k-1})$ and $\tpl{y}=(y_1,\dots,y_{k-1})$ be two $(k-1)$-tuples, which are $(\eps',p,\ell_{\subcor{cor:goodpaths}})$-good for $S$ and also fix a set of vertices $S'$ of size at most $s$.

From now on we will solely work in $\cR'$.
To avoid clashes when constructing the paths we arbitrarily split the vertices in $V(G) \setminus (S \cup S')$ into two sets $T_{\tpl{x}}$ and $T_{\tpl{y}}$ such that each of them intersects every $1$-cell in the same number of vertices and, in particular, more than an $\eta_{\sublem{lem:connectpartition}}$-fraction.
We then define $S_{\tpl{x}}:= S \cup T_{\tpl{y}}$ and $S_{\tpl{y}}:= S \cup T_{\tpl{x}}$.
For any coarse $(k-1)$-cell $E$ we get by Lemma~\ref{lem:DCL} that $|E| \le \frac{3}{2} \left( \frac{n}{t} \right)^{k-1} \prod_{i=2}^{k-1}d_i^{\binom{k-1}{i}}$.
Similarly, applying Lemma~\ref{lem:DCL} to the $1$-cells restricted to $T_{\tpl{x}}$ and $T_{\tpl{y}}$ respectively, we get that $|E \setminus S_{\tpl{x}}^{k-1}|,|E \setminus S_{\tpl{y}}^{k-1}| \ge \frac{1}{2} \left( \eta_{\sublem{lem:connectpartition}} \frac{n}{t} \right)^{k-1} \prod_{i=2}^{k-1}d_i^{\binom{k-1}{i}}$.
Therefore, we have
\begin{align}
\label{eq:fraccells}
\frac{|E \setminus S_{\tpl{x}}^{k-1}|}{|E|} , \frac{|E \setminus S_{\tpl{y}}^{k-1}|}{|E|} \ge \frac{1}{3} \eta_{\sublem{lem:connectpartition}}^{k-1} > 2\nu_{\sublem{lem:connectpartition}} \quad \text{for all } (k-1)\text{-cells } E.
\end{align}

We apply Corollary~\ref{cor:goodpaths} with $S$, $S''=S' \cup \{ y_1,\dots,y_{k-1}\}$, and $\tpl{x}$ to obtain a set $X$ of $\nu_{\subcor{cor:goodpaths}} n^{k-1}$ different $(k-1)$-tuples $\tpl{x'}$ that are reachable from $\tpl{x}$ by a path of length $\ell_{\subcor{cor:goodpaths}}$ with no vertices in $S \cup S''$ but possibly some of $\tpl{x}$.
We now want to estimate how many of the tuples from $X$ we cannot use.
By assumption, in $\cR'$ there are at most $\nu t$ $1$-cells missing, which accumulate to at most $\nu t \tfrac nt n^{k-2} \le \tfrac14 \nu_{\subcor{cor:goodpaths}} n^{k-1}$ tuples.
Also, there can be at most $(k-1) \tfrac{n}{t} n^{k-2} \le \tfrac18 \nu_{\subcor{cor:goodpaths}} n^{k-1}$ tuples which are not crossing with respect to the partition.
Therefore, we have at least $\tfrac12 \nu_{\subcor{cor:goodpaths}} n^{k-1} \ge (k-1)! \nu_{\sublem{lem:connectpartition}} n^{k-1}$ usable tuples in $X$ and there exists a coarse $(k-1)$-cell $E_0$, such that at least a $\nu_{\sublem{lem:connectpartition}}$-fraction of this cell is contained in $X$ and they all have the same ordering.

As $\cR'$ is tightly linked, we can fix a $k$-edge $H_{\tpl{x}}:=\hat{P}(Q_{\tpl{x}},\cP_c^*)$ in $\cR'$ that contains $E_0$.
By Lemma~\ref{lem:connectpartition} applied with $S_{\tpl{x}}$ to $H_{\tpl{x}}$ and another $(k-1)$-cell $E_1$ from $H_{\tpl{x}}$, using~\eqref{eq:fraccells}, there exists a $(k-1)$-tuple $\tpl{x}_0 \in (X \cap E_0)$ such that for a $(1-\nu_{\sublem{lem:connectpartition}})$-fraction of the tuples $\tpl{x}' \in E_1$ there exists a tight path of length $\ell_{\sublem{lem:connectpartition}}$ from $\tpl{x}_0$ to $\tpl{x}'$ with internal vertices not in $S_{\tpl{x}} \cup S''$ for any set $S''$ of at most $s_{\sublem{lem:connectpartition}}$ vertices.

We repeat the above for $\tpl{y}$ with $S_{\tpl{y}}$ and $S''$ the union of $S'$ with the set of all vertices on the path from $\tpl{x}$ to $\tpl{x}_0$, where $|S''| \le s + \ell_{\sublem{lem:connectpartition}} = s_{\subcor{cor:goodpaths}}$.
From this we obtain $Y$, $E_{h+1}$, $E_{h}$, $H_{\tpl{y}}$, and $\tpl{x}_{h+1} \in Y \cap E_{h+1}$ with the analogous properties as above.
As the $1$-cells of $E_0$ are incident to at most $(k-1) \tfrac{n}{t} n^{k-2} \le \tfrac18 \nu_{\subcor{cor:goodpaths}} n^{k-1}$ $(k-1)$-tuples we can assume that $E_0$ and $E_{h+1}$ use a disjoint set of $1$-cells and, therefore, the paths starting in $\tpl{x}_0$ and $\tpl{x}_{h+1}$ do not overlap.

We will now use that, by assumption, $\cR'$ is tightly linked to connect $\tpl{x}_0$ to $\tpl{x}_{h+1}$. We apply tight linkedness with $\tpl{u}$ being the clusters of $E_1$ in the order induced by the order of $\tpl{x}_0$ on $H_{\tpl{x}}$, and $\rho^{-1}\big(\tpl{v})$ defined similarly on $H_{\tpl{y}}$.
By definition of $h$ and $\rho$, there exists a sequence of $h$ edges $E_1,\dots,E_{h}$ that give a homomorphism of a tight path, from $\tpl{u}$ to $\tpl{v}$ in that order, and the edges $E_{i}$ and $E_{i+1}$ are contained in a $k$-edge of $\cR' $ for $i=1,\dots,h-1$.

We will connect $\tpl{x}_0$ to $\tpl{x}_{h+1}$ by following $E_1,\dots,E_{h}$.
Assume that for some $i=0,\dots,h-2$ we have a tuple $\tpl{x}_i$ in $E_i$, tight paths from $\tpl{x}$ to $\tpl{x}_i$ and from $\tpl{y}$ to $\tpl{x}_{h+1}$ with no vertices in $S \cup S'$ but possibly some of $\tpl{x}$ and $\tpl{y}$, and denote the set of vertices of these paths by $S''$.
Further assume that for at least a $(1-\nu_{\sublem{lem:connectpartition}})$-fraction of the $(k-1)$-tuples $\tpl{x}'$ from $E_{i+1}$ there exists a tight path of length at most $\ell_{\sublem{lem:connectpartition}}$ from $\tpl{x}_i$ to $\tpl{x}'$ with no internal vertices in $S_{\tpl{x}} \cup S''$.
Then we apply Lemma~\ref{lem:connectpartition} with $S_\tpl{x} \cup S''$ to the $k$-edge of $\cR'$ containing $E_{i+1}$ and $E_{i+2}$ to obtain with~\eqref{eq:fraccells} that there exists a $(k-1)$-tuple $\tpl{x}_{i+1}$ from $E_{i+1} \setminus (S_{\tpl{x}} \cup S'')^{k-1}$ with a path from $\tpl{x}_i$ to $\tpl{x}_{i+1}$ of length $\ell_{\sublem{lem:connectpartition}}$ with no internal vertices in $S_{\tpl{x}} \cup S''$, such that for at least a $(1-\nu_{\sublem{lem:connectpartition}})$-fraction of the $(k-1)$-tuples $\tpl{x}'$ from $E_{i+2}$ there exists a tight path of length at most $\ell_{\sublem{lem:connectpartition}}$ from $\tpl{x}_{i+1}$ to $\tpl{x}'$ with no internal vertices in $S_{\tpl{x}} \cup S''$ for any set $S''$ of at most $s_{\sublem{lem:connectpartition}}$ vertices.
This implies the condition above for $i+1$ and, therefore, we can advance to the next step.

For the final step let $S''$ be the vertices on the tight paths from $\tpl{x}$ to $\tpl{x}_{h-1}$ and from $\tpl{y}$ to $\tpl{x}_{h+1}$ that we now have and note that $|S''| \le s_{\sublem{lem:connectpartition}}$.
Observe, that for a $(1-2\nu_{\sublem{lem:connectpartition}})$-fraction of the $(k-1)$-tuples $\tpl{x}'$ from $E_{h-1}$ there exist a tight path from $\tpl{x}_{h-1}$ and $\tpl{x}_{h+1}$ to $\tpl{x}'$ avoiding $S_{\tpl{x}} \cup S''$ and $S_{\tpl{y}} \cup S''$ respectively.
Then, by~\eqref{eq:fraccells}, there exist an $\tpl{x}_{h}$ in $E_{h-1} \setminus (S \cup S' \cup S'')^{k-1}$ and two paths from $\tpl{x}_{h-1}$ to $\tpl{x}_{h}$ and from $\tpl{x}_{h+1}$ to $\tpl{x}_{h}$ of length $\ell_{\sublem{lem:connectpartition}}$ that only overlap in $\tpl{x}_{h}$.
This finishes the path from $\tpl{x}_0$ to $\tpl{x}_{h+1}$ and, therefore, we have a path from $\tpl{x}$ to $\tpl{y}$, of length $\ell$ with no internal vertices from $S \cup S'$.
\end{proof}

The proof of Lemma~\ref{lem:connectpartition} is fairly long and intricate. Before explaining it, however, let us sketch an easier version. Suppose that $k=2$ (i.e.\ we are dealing with graphs, not hypergraphs) and rather than having two clusters which are adjacent in the reduced graph, we have a path of $\ell+1$ clusters $V_1,\dots,V_\ell,V_{\ell+1}$ in the reduced graph. We want to show that for most vertices $x\in V_1$, there is a path from $x$ to $y$ for most $y\in V_{\ell+1}$. To begin with, we look at the fine parts within $V_\ell$. We discard those fine parts which do not form $(f_2,\tfrac12d,p)$-regular pairs with most fine parts in $V_{\ell+1}$; by definition of the reduced graph, there are few such, and we let $X_\ell$ be the remaining subset of $V_\ell$. Next, for each $i=\ell-1,\dots,1$ we discard from $V_i$ those vertices with fewer than $(d-\eps_2)p|X_{i+1}|$ neighbours in $X_{i+1}$ to obtain $X_i$. Again, by regularity we discard few vertices at each step, so $X_1$ is most of $V_1$. Now if we choose any $x\in X_1$, we claim there is a path from $x$ to $y$ for most $y\in V_{\ell+1}$.

To see this, note that there are many paths which start at $x$ and go out to $X_\ell$: we can construct these paths greedily starting from $x$, and we have at least $\tfrac12dp|V_i|$ choices in each $X_i$. By Lemma~\ref{lem:goodexpansion} and choice of $\ell$, there are linearly many different endvertices of these paths in $X_\ell$. We call this the \emph{coarse expansion}. However the number of these endvertices will be much \emph{smaller} than $\eps_2|X_\ell|$, so we cannot use the coarse regularity to say anything about the set of endvertices. This is where we need the fine partition: we can ensure the fine regularity constant $f_2$ is so small that the number of endvertices is much larger than $f_2|X_\ell|$. By averaging, there is a fine part $Z$ contained in $X_\ell$ which contains a set $R_0$ of endvertices, where $|R_0|\ge f_2|Z|$. Now $Z$ forms a $(f_2,\tfrac12d,p)$-regular pair with most fine parts in $V_{\ell+1}$. For any such fine part $Z'$, by $(f_2,\tfrac12d,p)$-regularity, the set $\overline{R_1}$ of vertices in $Z'$ which we cannot reach, i.e.\ which do not send an edge to $R_0$, is of size at most $f_2|Z'|$. In other words, we have found, for most fine parts $Z'$ in $V_{\ell+1}$, a path from $x$ to most vertices of $Z'$; that is the desired paths to most vertices of $V_{\ell+1}$. We call this second step the \emph{fine expansion}.

\medskip

It is fairly easy to see that this strategy still works with sets $S$ and $S'$ to avoid. It is also not very hard to modify it to work with one regular pair rather than a path of regular pairs: we split off a small fraction of each cluster to use for the coarse expansion (and we do not reuse this part for the fine expansion). What is not, however, so easy is to make this argument work for $k$-graphs for $k\ge 3$. The coarse expansion step works much as described above, but the fine expansion requires more care. If we are given $k=3$ and a regular polyad on parts $(X,Y,Z)$, and a significant fraction of the $XY$ $2$-cell are marked as end-tuples of tight paths from some given $\tpl{x}$, then we cannot necessarily conclude that almost all pairs in the $YZ$ $2$-cell are end-tuples of tight paths from $\tpl{x}$. We can only conclude this for those pairs whose vertex in $Y$ is also in many marked pairs. However this does then imply that most vertices of $Z$ are in $YZ$ pairs which form an edge with a marked pair, and taking another step, using another regular polyad $(Y,Z,W)$, we can finally argue that most $ZW$ pairs are end-tuples of tight paths from $\tpl{x}$; so the $ZW$ pairs play the same role as $Z'$ in the argument sketched above. For higher uniformity, we generalise this argument; in uniformity $k$, we need $k-1$ steps.

Before we prove Lemma~\ref{lem:connectpartition}, we give the following lemma, which deals with the fine expansion mentioned above. We will also reuse it in proving Lemma~\ref{lem:respath}.

\begin{lemma}\label{lem:finconnect}
Given $k\ge3$ and $\delta$, $d_k>0$, for all sufficiently small $f'_k>0$ we have: given $d_0>0$, for all sufficiently small $f'>0$ and all sufficiently large $m$ the following holds.

Given a set $V$ of vertices, suppose that we have a ground partition $\cP=\{X_0,\dots,X_{2k-3}\}$ with $|X_i|=m$ for each $i$, and for each $2\le i\le k-1$ a $\cP$-partite $i$-graph $G_i$ on $V$ such that for each $Y\subset\{0,\dots,2k-3\}$ the graph $G_i\big[\prod_{y\in Y}X_y\big]$ is $(d_i,f',1)$-regular with respect to $G_{i-1}$ (where we assume $E(G_1)=V$). Furthermore suppose that we have a $\cP$-partite $k$-graph $G_k$, such that $G_k[X_j,\dots,X_{j+k-1}]$ is $(d_k,f'_k,p)$-regular with respect to $G_{k-1}$ for each $0\le j\le k-2$. Suppose that $d_i\ge d_0$ for each $2\le i\le k-1$. Suppose that for each $2\le i\le k$, all the edges of $G_i$ are supported by $G_{i-1}$.

Suppose that we are given a set $R_0\subset G_{k-1}[X_0,\dots,X_{k-2}]$ of size at least
\[\delta m^{k-1}\prod_{\ell=2}^{k-1}d_\ell^{\binom{k-1}{\ell}}\,.\]
Let $R_{k-1}\subset G_{k-1}[X_{k-1},\dots,X_{2k-3}]$ be those $(k-1)$-edges which are the end-tuples of some tight path in $G_k$ with one vertex in each of $X_0,\dots,X_{2k-3}$ and whose start $(k-1)$-tuple is in $R_0$.

Then we have $|R_{k-1}|\ge (1-\delta) m^{k-1}\prod_{\ell=2}^{k-1}d_\ell^{\binom{k-1}{\ell}}$.
\end{lemma}
\begin{proof}
Given $k$ and $\delta,d_k>0$, we set $\eta=4^{-10k}\delta$, and we set $\gamma=\tfrac{1}{1000k^2}\eta^{2k^2}$. Suppose $0<f'_k<\tfrac{1}{100k}\gamma^3d_k$, and given $d_0$, let $f'>0$ be sufficiently small for all the below applications of Lemmas~\ref{lem:DCL},~\ref{lem:DCL-variant} and~\ref{lem:MDL} with input $k$, $\alpha=1$, $\gamma$, $\delta$, and $d_0$ as required.

For each $1\le j\le k-1$, we let $R_j$ be those $(k-1)$-edges which are the end-tuples of some tight path in $G_k$ with one vertex in each of $X_j,\dots,X_{k+j-2}$ and whose start $(k-1)$-tuple is in $R_0$. For each $1\le j\le k-1$, let $\overline{R_j}=E\big(G_{k-1}[X_j,\dots,X_{j+k-2}]\big)\setminus R_j$. That is, $\overline{R_j}$ is the part of the $j$th $(k-1)$-cell which we \emph{cannot} reach from $R_0$.

By definition, for each $1\le j\le k-1$, there is no edge of $G_k$ which contains both a $(k-1)$-set in $R_{j-1}$ and one in $\overline{R_j}$. Since $G_k[X_{j-1},\dots,X_{j+k-2}]$ is $(d_k,f'_k,p)$-regular with respect to $G_{k-1}$, we conclude that the number of copies of $K^{(k-1)}_k$ in $G_{k-1}$ which contain both an edge of $R_{j-1}$ and one of $\overline{R_j}$ is smaller than an $f'_k$-fraction of all the copies of $K^{(k-1)}_k$ in $G_{k-1}[X_{j-1},\dots,X_{j+k-2}]$, i.e.\ it is at most
\begin{equation}\label{eq:finconnect:bound}
 2 f'_k m^{k}\prod_{i=2}^{k-1}d_i^{\binom{k}{i}}\,.
\end{equation}

The remainder of the proof of this lemma will consist of repeatedly using this fact, together with counting in $G_{k-1}$, to argue that $R_{k-1}$ is necessarily large. We will not need to use $G_k$ (or $(d_k,f'_k,p)$-regularity) again. We do this by the following induction.

 Let $U_{0,0}=\{\emptyset\}$ (which we think of as the $0$-edge in a complex).

\begin{claim}\label{cl:reachR}
There exist sets with the following properties. For each $1\le j\le k-1$, $U_{j,j}$ is a subgraph of $G_j[X_{k-1},\dots,X_{k+j-2}]$, with
\[|U_{j,j}|=\big(1-j\eta\big)m^{j}\prod_{\ell=2}^jd_\ell^{\binom{j}{\ell}}\pm 1\,.\]
For each $0\le j< k-1$ and each $j<i\le k-1$, the set $U_{j,i}$ is a subgraph of $G_i[X_{k+j-i-1},\dots,X_{k+j-2}]$.
For each $0\le j\le k-2$ and each $j\le i\le k-2$, each edge of $U_{j,i}$ is contained in
\[\eta^{j+1} m\prod_{\ell=2}^{i+1}d_\ell^{\binom{i}{\ell-1}}\pm 1\]
edges of $U_{j,i+1}$.

Each edge of $U_{j,k-1}$ is an edge of $R_j$ for each $0 \leq j \leq k-1$.
\end{claim}
\begin{claimproof}
 We begin with the base case $j=0$. Let $U'_{0,k-1}=R_0$. For each $k-2\ge i\ge 1$ successively, we let $U'_{0,i}$ contain all the $i$-edges with one vertex in each of $X_{k-1-i},\dots,X_{k-2}$ which lie in
 \[\text{between}\quad\eta m\prod_{\ell=2}^{i+1}d_\ell^{\binom{i}{\ell-1}}\quad\text{and}\quad 2m\prod_{\ell=2}^{i+1}d_\ell^{\binom{i}{\ell-1}}\]
 edges of $U'_{0,i+1}$. We now check $|U'_{0,i}|$ is sufficiently large. By the third part of Lemma~\ref{lem:DCL-variant}, there are at most $\gamma m^i\prod_{\ell=2}^id_\ell^{\binom{i}{\ell}}$ elements of $G_i$ that violate the upper bound, which by the second part of Lemma~\ref{lem:DCL-variant}
 are in total contained in at most $3\gamma m^{i+1}\prod_{\ell=2}^{i+1}d_\ell^{\binom{i+1}{\ell}}$ edges of $U'_{0,i+1}$. By definition and by Lemma~\ref{lem:DCL}, the total number of edges of $U'_{0,i+1}$ containing edges of $G_i$ violating the lower bound is at most $(1+\gamma)\eta m^{i+1}\prod_{\ell=2}^{i+1}d_\ell^{\binom{i+1}{\ell}}$, and therefore there are at least
 \[|U'_{0,i+1}|-2\eta m^{i+1}\prod_{\ell=2}^{i+1}d_\ell^{\binom{i+1}{\ell}}\]
 edges of $U'_{0,i+1}$ which contain edges of $U'_{0,i}$. We check inductively that this is for each $i$ at least $\tfrac12|U'_{0,i+1}|$, and hence obtain
 \[|U'_{0,i}|\ge\frac{\tfrac12|U'_{0,i+1}|}{2m\prod_{\ell=2}^{i+1}d_\ell^{\binom{i}{\ell-1}}}\ge 4^{i+1-k}\delta m^i\prod_{\ell=2}^id_\ell^{\binom{i}{\ell}}\,.\]
 By choice of $\eta$, in particular we have $|U'_{0,1}|>\eta m$.

Now we let $U_{0,1}\subset U'_{0,1}$ be some set of size $\eta m\pm 1$, and for each $i\ge 1$ successively and each edge $e$ of $U_{0,i}$, we put into $U_{0,i+1}$ some
\[\eta m\prod_{\ell=2}^{i+1}d_\ell^{\binom{i}{\ell-1}}\pm 1\]
edges in $U'_{0,i+1}$ which contain $e$. These sets witness that Claim~\ref{cl:reachR} holds for $j=0$.
 
 Now suppose $1\le j\le k-1$. Given sets $U_{j-1,j-1},\dots,U_{j-1,k-1}$ as in the claim statement for $j-1$, let $C$ consist of those edges in $U_{j-1,k-2}$ which are contained in at least
 \[\gamma m\prod_{\ell=2}^{k-1}d_\ell^{\binom{k-2}{\ell-1}}\]
 members of $\overline{R_j}$. Let $A=R_{j-1}$ and let $B$ consist of those members of $\overline{R_j}$ which contain an edge of $C$. By Lemma~\ref{lem:MDL} (with input $\gamma/2$) and choice of $f'_k$, if $|C|\ge \gamma m^{k-2}\prod_{\ell=2}^{k-2} d_\ell^{\binom{k-2}{\ell}}$,
 then the number of copies of $K^{(k-1)}_k$ containing both an edge of $R_{j-1}$ and of $\overline{R_j}$ is in contradiction to~\eqref{eq:finconnect:bound}. We conclude
 \[|C|<\gamma m^{k-2}\prod_{\ell=2}^{k-2} d_\ell^{\binom{k-2}{\ell}}\,.\]
 By definition, we have
 \[\big|U_{j-1,k-2}\big|=\Bigg(\prod_{i=j-1}^{k-3}\big(\eta^{j} m\prod_{\ell=2}^{i+1}d_\ell^{\binom{i}{\ell-1}}\pm1\big)\Bigg)\cdot\big(1-(j-1)\eta\big)m^{j-1}\prod_{\ell=2}^jd_\ell^{\binom{j}{\ell}}\,,\]
 and hence
 \[\big|U_{j-1,k-2}\setminus C\big|\ge(1-(j-1)\eta)(1-2\gamma\eta^{-k^2})\eta^{(k-j-1)j}m^{k-2}\prod_{\ell=2}^{k-2}d_\ell^{\binom{k-2}{\ell}}\,.\]
 We let $U'_{j,k-1}$ be the set of all edges in $R_j$ which contain an edge of $U_{j-1,k-2}$. Suppose that a given edge $e\in U_{j-1,k-2}\setminus C$ is contained in at least $(1-\gamma)m\prod_{\ell=2}^{k-1}d_\ell^{\binom{k-2}{\ell-1}}$ edges of $G_{k-1}$ together with a vertex of $X_{k+j-2}$. Then by definition of $C$, at least $(1-2\gamma)m\prod_{\ell=2}^{k-1}d_\ell^{\binom{k-2}{\ell-1}}$ of these are edges of $R_j$. By the third part of Lemma~\ref{lem:DCL-variant} there are at most $\gamma m^{k-2} \prod_{\ell=2}^{k-2} d_\ell^{\binom{k-2}{\ell}}$ edges in $U_{j-1,k-2}$ that are contained in less than $(1-\gamma)m\prod_{\ell=2}^{k-1}d_\ell^{\binom{k-2}{\ell-1}}$ edges of $G_{k-1}$ and we conclude
 \begin{align}
  \nonumber\big|U'_{j,k-1}\big|&\ge (1-2\gamma)m\prod_{\ell=2}^{k-1}d_\ell^{\binom{k-2}{\ell-1}}\cdot (1-(j-1)\eta)(1-3\gamma\eta^{-k^2})\eta^{(k-j-1)j}m^{k-2}\prod_{\ell=2}^{k-2}d_\ell^{\binom{k-2}{\ell}}\\
  \label{eq:finconnect:sizeUjk-1}&\ge\big(1-5\gamma\eta^{-k^2}\big)(1-(j-1)\eta)\eta^{(k-j-1)j}m^{k-1}\prod_{\ell=2}^{k-1}d_\ell^{\binom{k-1}{\ell}}\,.
 \end{align}
 If $j=k-1$, we set $U_{k-1,k-1}=U'_{k-1,k-1}$ and by choice of $\gamma$ we are done, so we now suppose $1\le j\le k-2$.
 
 We next let for each $i=k-2,\dots,j$ in succession the set $U'_{j,i}$ consist of all those edges of $G_i$ with one vertex in each of $X_{k+j-i-1},\dots,X_{k+j-2}$ which contain at least
 \[\eta^{j+1} m\prod_{\ell=2}^{i+1}d_\ell^{\binom{i}{\ell-1}}\]
 edges of $U'_{j,i+1}$. Finally we let $U_{j,j}=U'_{j,j}$, and for each $j+1\le i\le k-1$, we create $U_{j,i}$ by, for each $e\in U_{j,i-1}$, putting in $\eta^{j+1} m\prod_{\ell=2}^{i+1}d_\ell^{\binom{i}{\ell-1}}\pm 1$ edges of $U'_{j,i}$ which contain $e$.
 
 We claim that these sets witness the $j$ case of Claim~\ref{cl:reachR}. To show this, we need to show $|U_{j,j}|$ is sufficiently large, and for that purpose we establish bounds on $|U'_{j,i}|$ for each $k-2\ge i\ge j$ in succession.
 
 Given $i$, consider the set of edges $Y_i$ in $G_i[X_{k+j-i-1},\dots,X_{k+j-2}]$ which contain an element of $U_{j-1,i-1}$. By the third part of Lemma~\ref{lem:DCL-variant} and because $U_{j-1,i-1}$ is sufficiently large, we have
 \[|Y_i|=(1\pm\gamma)\big|U_{j-1,i-1}|m\prod_{\ell=2}^id_\ell^{\binom{i-1}{\ell-1}}=(1\pm2\gamma)\eta^{(i-j)j}m^i\prod_{\ell=2}^id_\ell^{\binom{i}{\ell}}\,.\]
 By Lemma~\ref{lem:MDL} with input $\gamma/2$, for any subset $Y'$ of $Y_i$ with $|Y'|\ge\tfrac12|Y_i|$, the number of edges of $G_{i+1}$ which contain an edge of $Y'$ and an edge of $U_{j-1,i}$ is
 \[|Y'|\big(\eta^j\pm\gamma)m \prod_{\ell=2}^{i+1}d_\ell^{\binom{i}{\ell-1}}\,.\]
 Observe that by construction every edge of $U'_{j,i+1}$ contains an edge of $Y_i$ and an edge of $U_{j-1,i}$. Let $Y'$ contain all the edges of $U'_{j,i}$, and if necessary additional edges to match the lower bound of Lemma~\ref{lem:MDL}. Then putting the bound of Lemma~\ref{lem:MDL} together with the definition of $U'_{j,i}$, we have
 \begin{align*}
  |U'_{j,i+1}|&\le |U'_{j,i}|(\eta^j+\gamma)m\prod_{\ell=2}^{i+1}d_\ell^{\binom{i}{\ell-1}}+|Y_i\setminus U'_{j,i}|\eta^{j+1} m\prod_{\ell=2}^{i+1}d_\ell^{\binom{i}{\ell-1}}\\
  &\le \Big((1+\gamma \eta^{-j})\tfrac{|U'_{j,i}|}{|Y_i|}+\eta\big(1-\tfrac{|U'_{j,i}|}{|Y_i|}\big)\Big)\eta^{j}m|Y_i|\prod_{\ell=2}^{i+1}d_\ell^{\binom{i}{\ell-1}}\\
  &\le \Big(1+\gamma\eta^{-j}-(1-\eta)\big(1-\tfrac{|U'_{j,i}|}{|Y_i|}\big)\Big)\eta^{(i+1-j)j}m^{i+1}\prod_{\ell=2}^{i+1}d_\ell^{\binom{i+1}{\ell}}\,.
 \end{align*}
 Comparing the last line, in the case $i=k-2$, with~\eqref{eq:finconnect:sizeUjk-1}, we see
 \[(1-5\gamma\eta^{-k^2})\big(1-(j-1)\eta\big)\le \Big(1+\gamma\eta^{-j}-(1-\eta)\big(1-\tfrac{|U'_{j,k-2}|}{|Y_{k-2}|}\big)\Big)\]
 and hence
 \begin{align*}
  |U'_{j,k-2}|&\ge |Y_i|\Big(1-(j-1)\eta-2j\eta^2-9\gamma\eta^{-k^2}\Big)\\
  &\ge\Big(1-(j-1)\eta-2j\eta^2-10\gamma\eta^{-k^2}\Big)\eta^{(k-2-j)j}m^{k-2}\prod_{\ell=2}^{k-2}d_\ell^{\binom{k-2}{\ell}}\,.
 \end{align*}
 Repeating the same argument for $i=k-3,k-4,\dots,j$ in succession, we end up with
 \[|U'_{j,j}|\ge \Big(1-(j-1)\eta-2(k-1-j)j\eta^2-5\cdot 2^{k-1-j}\gamma\eta^{-k^2}\Big)m^{j}\prod_{\ell=2}^{j}d_\ell^{\binom{j}{\ell}}\,,\]
 which by choice of $\eta$ and $\gamma$ is as required. 
\end{claimproof}

The lemma follows from the case $j=k-1$ of Claim~\ref{cl:reachR} directly, by choice of $\eta$.
\end{proof}

The proof of Lemma~\ref{lem:connectpartition} proceeds as follows. We start by setting up the parameters and the graph together with the properties that we assume.
Afterwards, we choose the fine $(k-1)$-cells $\cC$ (Properties~\ref{fr:polyad} and~\ref{fr:cells}) that we would like to use later and show that these make up most of the coarse $(k-1)$-cell (Claim~\ref{claim:goodfinecells}).
Then we select a small fraction $\overline{\cC}$ of these $1$-cells, which gives a sufficiently large vertex set $\overline{V}$ for the coarse expansion.
Having this setup, we can use the coarse partition and $\overline{V}$ to expand (Claim~\ref{claim:highdgreecells}).
After this we reached a significant fraction of a coarse $(k-1)$-cell, carefully chosen to also give significant fraction of a well-behaved fine $(k-1)$-cell. We then use Lemma~\ref{lem:finconnect} to argue that from a significant fraction of this one fine cell, we can reach almost all of almost all fine $(k-1)$-cells, as required.

\begin{proof}[Proof of  Lemma~\ref{lem:connectpartition}]
\textbf{Setting the parameters.}
For $k \ge 3$ and $\gamma>0$, let $\ell$ be the smallest integer exceeding $\frac{k-1}{\gamma}+2k$, such that $\ell \equiv 1 \pmod{k}$
and let $s \ge 3k$.
Let $d,\eta,\nu > 0$, where w.l.o.g.~we can assume $d,\eta,\nu \le 1$, and $t_0$ be an integer.
We let $0<\gamma_c \le \frac{\nu}{10^{\ell}} t_0^{-2k2^k}$ and $\eps_k>0$ such that
\begin{align}
\label{eq:def_epsk}
(2k)^{3k} \eps_k^{1/2k} \le \frac{\nu^k \eta}{10^{\ell}}  \,.
\end{align}

Next, we let $\alpha_c=\tfrac{\nu}{10k}\eps_k^{1/k}$ and $\eps : \mathbb{N} \mapsto (0,1]$ tend to zero sufficiently fast, such that for any $t_1 \ge t_0$ we have that $\eps:=\eps(t_1)< 2^{-2k 2^k} t_1^{-1} \eps_k$ is small enough for Lemma~\ref{lem:DCL} on input $k$, $\alpha_c$, $d_0=t_0^{-1}$, and $\gamma_c$
and Lemma~\ref{lem:DCL-variant} on input $k$, $\alpha_c$, $\gamma_c$, $d_0=t_0^{-1}$.

We then define
\begin{align}
\label{eq:delta}
\mu=\frac{d \eps_k^{1/k} \nu}{t_1 20 k^2} \, t_1^{-2^k}, \quad
\delta = \frac{\mu^{2 \ell}}{10 (2\ell)!} \, \frac{20k^2 t_0}{\eps_k^{1/2k} \nu } \,  t_1^{-2^k},
\end{align}
and note that $\delta \le \nu$.
We let
\begin{align}
\label{eq:def_gammaf}
0<\gamma_f \le \frac{\delta}{10^k} \frac{\nu}{3^k} t_1^{-2k2^k}
\end{align}
and
\begin{align}
\label{eq:def_fk}
0<f_k \le \frac{\delta^3 \nu}{30^{3k}}\, .
\end{align}
Suppose $f_k$ is small enough for Lemma~\ref{lem:finconnect} on input $k$, $\tfrac12\delta$, $\tfrac12d$.

Next, we let $f : \mathbb{N} \mapsto (0,1]$ tend to zero sufficiently fast, such that for $t_2 \ge t_1$ we have that $f:=f(t_2)<\eps t_2^{-1}$ is small enough for Lemma~\ref{lem:DCL} on input $k$, $\alpha=1$,  $d_0=t_1^{-1}$, and $\gamma_f$ and Lemma~\ref{lem:RRL} on input $k$, $\alpha=\eps_k^{1/2k}$, and $d_0=t_1^{-1}$.
Additionally assume that $\sqrt{f}$ is small enough for Lemma~\ref{lem:finconnect} with input as above, $f_k$, and $d_0=\tfrac{1}{2t_2}$.
For convenience we summarise the relative order of the parameters (besides $\gamma$, $k$, $\ell$, $s$) in a simplified form
\[f \ll f_k, \gamma_f, t_2^{-1} \ll \delta, \eps \ll \eps_k, \gamma_c, t_1^{-1} \ll t_0^{-1}, d, \eta, \nu, \eps' \, .\]

Given $n$ we let $p \ge n^{-1+\gamma}$ and $\Gamma=G^{(k)}(n,p)$.
We assume that $\Gamma$ is in the good event
of Lemma~\ref{lem:goodexpansion} with $\gamma$, $k$, $\ell_0=\ell-3k+4$, and $\mu$.

We are given $G \subseteq \Gamma$ and the reduced complex $\cR_{\eps_k,d}(G;\cP^*_c,\cP^*_f)$ of a $(t_0,t_1,t_2,\eps_k,\eps^2,f_k,f^2,p)$-strengthened pair $(\cP^*_c,\cP^*_f)$ for $G$.
Fix the density vectors $\mathbf{d}=(d_{k-1},\dots,d_2)$ and $\mathbf{d_f}=(d_{k-1}^f,\dots,d_2^f)$ with $d_i \ge t_1^{-1}$ and $d_i^f \ge t_2^{-1}$ for $i=1,\dots,k$ and let $t$ and $t_f$ be the number of coarse and fine $1$-cells respectively.

Take a $k$-set $Q_0$ such that $H:=\hat{P}(Q_0;\cP^*_c)$ is in $\cR_{\eps_k,d}(G;\cP^*_c,\cP^*_f)$. For convenience, we equip the complex $H$ with a $k$-level consisting of all $k$-sets supported by the $(k-1)$-edges of $H$. By definition, the relative $p$-density of $G$ with respect to $H$ is at least $d$. Note that every coarse $1$-cell contains exactly $t_f/t$ fine $1$-cells.

We let $V_1,\dots,V_k$ be the $1$-cells of $H$ and $E_1,\dots,E_k$ be the $(k-1)$-cells of $H$, where $E_i\subset \prod_{j\not=i}V_j$ for $i=1,\dots,k$.
Further, let $S \subseteq V(G)$ be such that $|S \cap V_i| \le (1-\eta) n/t$.
By the good event of Lemma~\ref{lem:goodtuples} assumed above there are $o(n)$ $(k-1)$-sets in $V(\Gamma)$ outside of $S$ that are not $(f,p,\ell)$-good for $S$.
Before we can expand using the coarse partition we need to ensure that we use edges that behave well with respect to the fine cells.

\textbf{Preparing the fine cells.}
The polyad $H$ corresponds to a $(k-1)$-complex $\mathcal{H}(G;H)$ in the multicomplex of the family of partitions $\cP_c^*$.
We denote by $\cH=\mathcal{H}(G;H,\cP_f^*)$ the multicomplex of the family of partitions $\cP_f^*$ restricted
to $\mathcal{H}(G;H)$, i.e.~we keep all the edges of the multicomplex of $\cP_f^*$ that correspond to cells that are contained in cells that correspond to edges of $\mathcal{H}(G;H)$.
For $2 \le j \le k-1$ and any $j$-set $Q$
supported by $\cH$, the number of fine $j$-cells supported by the polyad $\hat{P}(Q;\cP_c^*)$ is
\begin{align}
\label{eq:d_ioverd_if}
\frac{d_j \pm \eps}{d_j^f \pm f}=\left(1 \pm \frac{2\eps}{d_j}\right) \frac{d_j}{d_j^f}\,,
\end{align}
where the equality follows as $f< \eps t_2^{-1}$.
This follows by simple double counting of the $j$-sets supported by $\hat{P}(Q;\cP_c^*)$ that are contained in $\cH$.
This allows us for each $1 \le i \le k-1$ and any $i$-edge of $\cH$ to control the number of $(i+1)$-edges it is contained in.
But we have to carefully select the edges of $\cH$, which we are using.

For this, we denote by $\cH_{\eps_k}=\mathcal{H}_{\eps_k}(G;H,\cP_f^*)$ the $\eps_k$-reduced multicomplex of $G$ with respect to $(H,\cP^f)$, which is the (unique) maximal subcomplex of $\cH$ which has the following properties:
\begin{enumerate}[label=(FR\arabic*)]
\item\label{fr:polyad} For every $k$-edge $\hat{P}(Q;\cP^*_f)$ of $\mathcal{H}_{\eps_k}(G;H,\cP_f^*)$ we have that $G$ is $\big(f_k,p\big)$-regular with respect to $\hat{P}(Q;\cP^*_f)$ and  $d_p\big(G\big|\hat{P}(Q;\cP^*_f)\big) = d_p\big(G\big|H\big)\pm\eps_k$.
\item\label{fr:cells} For each $1 \le i \le k-1$, each $i$-edge of $\mathcal{H}_{\eps_k}(G;H,\cP_f^*)$ is in the boundary of at least
 \begin{align*}
\left(1-k 2^{i+4} \varepsilon_k^{1/k}\right)\frac{t_f}{t} \prod_{j=2}^{i+1} \left( \frac{d_j}{d_j^f} \right)^{\binom{i}{j-1}} \quad\text{ if $i<k-1$,} \\
\text{and }\quad\left(1-k 2^{i+4} \varepsilon_k^{1/k}\right)\frac{t_f}{t} \prod_{j=2}^{k-1} \left( \frac{d_j}{d_j^f} \right)^{\binom{i}{j-1}}\quad\text{ if $i=k-1$}
\end{align*}
$(i+1)$-edges of $\mathcal{H}_{\eps_k}(G;H,\cP_f^*)$ with respect to any other $1$-cell of $H$.
\end{enumerate}
We will show below that by this construction we keep sufficiently many $1$-edges and $(k-1)$-edges.

\begin{claim}
\label{claim:goodfinecells}	
There are at most $6 \eps_k^{1/k} \tfrac{t_f}{t}$ $1$-edges removed from $\mathcal{H}$ to get $\mathcal{H}_{\eps_k}$.
Furthermore, for any $i=1,\dots,k$, there are at least
\[(1-k\,  k!\eps_k^{1/k}) \left(\frac{t_f}{t}\right)^{k-1} \prod_{j=2}^{k-1}  \left(\frac{d_j}{d_j^f}\right)^{\binom{k-1}{j}}\]
$(k-1)$-edges in $\mathcal{H}_{\eps_k}$ that are also contained in $E_i$.
\end{claim}

\begin{claimproof}[Proof of Claim~\ref{claim:goodfinecells}]
Since $H$ is a regular polyad, for at most $3\eps_k ( \tfrac{n}{t} )^k  \prod_{j=2}^{k} d_j^{\binom{k}{j}}$ of the $k$-sets $Q'$ supported on $H$, the fine polyad $\hat{P}(Q';\cP^*_f)$ fails~\ref{fr:polyad}.
As any $k$-edge
in $\cH$ supports at least $\tfrac{3}{4} (\tfrac{n}{t_f})^k \prod_{j=2}^{k} (d_j^f)^{\binom{k}{j}}$ of these $Q'$, we have at most
\begin{align}
\label{eq:number_bad_edges}
4 \eps_k \left(\frac{t_f}{t}\right)^k  \prod_{j=2}^{k} \left(\frac{d_j}{d_j^f}\right)^{\binom{k}{j}}
\end{align}
$k$-edges in $\mathcal{H}$ failing~\ref{fr:polyad}.
We mark all these $k$-edges as \emph{bad}.
Then, for each $i=k-1,\dots,1$, in succession, we mark as \emph{bad} all $i$-edges which are contained in the boundary of at least
\begin{align*}
6\eps_k^{1/k} (k-i) \frac{t_f}{t} \prod_{j=2}^{i+1} \left( \frac{d_j}{d_j^f} \right)^{\binom{i}{j-1}}
\end{align*}
bad $(i+1)$-edges.

Now consider the following construction.
We begin by taking any bad $1$-edge, then any bad $2$-edge containing it, and so on until we obtain a bad $k$-edge together with an order on its vertices.
Clearly we obtain any given bad $k$-edge in at most $k!$ ways by following this process (since a $k$-edge together with an order determines the chosen edges in the process).
If there are more than $6\eps_k^{1/k}\frac{t_f}{t}$ bad $1$-edges, it follows that the number of bad $k$-edges is at least
 \[\tfrac{1}{k!}\prod_{i=0}^{k-1}\left( 5\eps_k^{1/k} (k-i) \frac{t_f}{t}\prod_{j=2}^{i+1}\left( \frac{d_j}{d_j^f} \right)^{\binom{i}{j-1}}\right)= 5^k\eps_k \left(\frac{t_f}{t}\right)^k\prod_{i=2}^k\left( \frac{d_i}{d_i^f} \right)^{\binom{k}{i}}\,, \]
which is a contradiction.

We claim that any edge of $\cP_f^*$ which is not in $\cH_{\eps_k}(G;H,\cP^*_f)$ is either bad or contains a bad edge.
To see this, consider the process of obtaining $\cH_{\eps_k}(G;H,\cP^*_f)$ by successively removing edges which either fail one of~\ref{fr:cells} or~\ref{fr:polyad}, or which contain a removed edge.
Suppose for a contradiction that at some stage in this process we remove an edge which is neither bad nor contains a bad edge; let $e$ be the first such edge removed. Observe that $e$ cannot have been removed for failing~\ref{fr:polyad}, since edges which fail this condition are bad.
Furthermore $e$ cannot have been removed for being unsupported, because all edges previously removed either were bad or contain bad edges, and by assumption $e$ contains no bad edges.
So $e$ was removed for failing~\ref{fr:cells}.
In other words, we have $|e|\le k-1$ and $e$ is neither bad nor contains a bad edge, but nevertheless there are many $(|e|+1)$-edges containing $e$ which either are bad or contain a bad edge.

Suppose that $f$ is a bad edge such that $|f\setminus e|=1$.
If $|f|=1$, then there are at most $6\eps_k^{1/k}\frac{t_f}{t}$ choices of $f$, each of which, by~\eqref{eq:d_ioverd_if}, is contained in at most
\begin{align}
\label{eq:badedgesin1cell}
\prod_{i=2}^{|e|+1} \left( \left( 1 + \frac{2\eps}{d_i} \right) \frac{d_i}{d_i^f} \right)^{\binom{|e|+1}{i}- \binom{|e|}{i}}  \le 2 \prod_{i=2}^{|e|+1} \left(\frac{d_i}{d_i^f}\right)^{\binom{|e|}{i-1}}
\end{align}
edges of $\cP^*_f$ of uniformity $|e|+1$ which contain $e$.
If $|f|>1$, then $f\cap e$ is non-empty and not a bad edge. There are at most $\binom{|e|}{\ell}$ choices of $f\cap e$ with $\ell$ elements, each of which by definition is contained in less than
 \[6\eps_k^{1/k}(k-\ell)\frac{t_f}{t}\prod_{i=2}^{\ell+1}\left( \frac{d_i}{d_i^f} \right)^{\binom{\ell}{i-1}}\]
 bad edges.
 Thus there are at most
 \[\binom{|e|}{\ell}\cdot 6\eps_k^{1/k}(k-\ell)\frac{t_f}{t}\prod_{i=2}^{\ell+1}\left( \frac{d_i}{d_i^f} \right)^{\binom{\ell}{i-1}}\]
 choices of $f$, each of which is contained in at most
 \[2 \prod_{i=2}^{|e|+1}\left(\frac{d_i}{d_i^f}\right)^{\binom{|e|}{i-1}-\binom{\ell}{i-1}}\]
 edges of $\cP^*_f$ of uniformity $|e|+1$ which contain $e$.
 
 Summing up, the number of $(|e|+1)$-edges containing $e$ which are either bad or contain a bad edge is at most
 \begin{align*}
  &6\eps_k^{1/k}\frac{t_f}{t}\cdot 2 \prod_{i=2}^{|e|+1}\left(\frac{d_i}{d_i^f}\right)^{\binom{|e|}{i-1}}\\
  &+\sum_{\ell=1}^{|e|}\binom{|e|}{\ell}\cdot 6\eps_k^{1/k} (k-\ell) \frac{t_f}{t}\left(\prod_{i=2}^{\ell+1}\left(\frac{d_i}{d_i^f}\right)^{\binom{\ell}{i-1}}\right)\cdot 2 \prod_{i=2}^{|e|+1}\left(\frac{d_i}{d_i^f}\right)^{\binom{|e|}{i-1}-\binom{\ell}{i-1}}\\
  &\le k 16\eps_k^{1/k}\frac{t_f}{t}\Bigg(\sum_{\ell=0}^{|e|}\binom{|e|}{\ell}\Bigg)\prod_{i=2}^{|e|+1}\left(\frac{d_i}{d_i^f}\right)^{\binom{|e|}{i-1}}=k 2^{|e|+4}\eps_k^{1/k}\frac{t_f}{t}\prod_{i=2}^{|e|+1}\left(\frac{d_i}{d_i^f}\right)^{\binom{|e|}{i-1}}\,.
 \end{align*}
 This last equation simply states that $e$ does \emph{not} fail~\ref{fr:cells}, which is our desired contradiction.
 In particular, every $1$-edge which is not bad is in $\cH_{\eps_k}(G;H,\cP^*_f)$.

With this at hand we want to estimate the number of $(k-1)$-edges of $\cH_{\eps_k}$ that are not in $E_1$ (the same argument applies to any other).
A bad $(k-1)$-edge is in the boundary of at least
\[6 \varepsilon_k^{1/k} \frac{t_f}{t} \prod_{j=2}^{k-1} \left(\frac{d_j}{d_j^f}\right)^{\binom{k-1}{j-1}}\]
bad $k$-edges
and a $k$-edge supports exactly $k$ different $(k-1)$-edges.
With the bound on the number of bad $k$-edges in~\eqref{eq:number_bad_edges} it then follows that there are at most
\[k \varepsilon_k^{(k-1)/k} \left(\frac{t_f}{t}\right)^{k-1} \prod_{j=2}^{k-1} \left(\frac{d_j}{d_j^f}\right)^{\binom{k-1}{j}}\]
bad $(k-1)$-edges.

More generally, for $i=k-2,\dots,1$, a bad $i$-edge is in the boundary of at least
\[6 \varepsilon_k^{1/k} (k-i) \frac{t_f}{t} \prod_{j=2}^{i+1} \left(\frac{d_j}{d_j^f}\right)^{\binom{i}{j-1}}\]
bad $(i+1)$-edges
and an $(i+1)$-edge supports exactly $i+1$ different $i$-edges.
It follows that there are at most
\[k^{\underline{k-i}} \varepsilon_k^{i/k} \left(\frac{t_f}{t}\right)^{i} \prod_{j=2}^{i+1} \left(\frac{d_j}{d_j^f}\right)^{\binom{i}{j}}\]
bad $i$-edges.
In particular, there are at most $k! \eps_k^{1/k} \frac{t_f}{t}$ $1$-edges removed from $\mathcal{H}$ to get $\mathcal{H}_{\eps_k}$.

Then, with~\eqref{eq:d_ioverd_if}, for $i=1,\dots,k-2$, there are at most
\begin{align*}
2 \cdot k^{\underline{k-i}} \eps_k^{i/k} \left(\frac{t_f}{t}\right)^{i} \prod_{j=2}^{i+1}  \left(\frac{d_j}{d_j^f}\right)^{\binom{i}{j}} \cdot \left(\frac{t_f}{t}\right)^{k-1-i} \prod_{j=2}^{k-1}  \left(\frac{d_j}{d_j^f}\right)^{\binom{k-1}{j}-\binom{i}{j}} \\
\le 2 \cdot k! \eps_k^{i/k} \left(\frac{t_f}{t}\right)^{k-1} \prod_{j=2}^{k-1}  \left(\frac{d_j}{d_j^f}\right)^{\binom{k-1}{j}}
\end{align*}
fine $(k-1)$-edges in $\mathcal{H}$ supported by a bad $i$-edge, but no bad $i'$-edge for $1\le i'\le i-1$.
Combining this with the number of $(k-1)$-edges of $\cH$ that are in $E_1$, which we can derive from~\eqref{eq:d_ioverd_if} and the bound on $\eps$,
we get that there are at least
\[ \left(1- \eps_k - 2 \sum_{i=1}^{k-1} k! \eps_k^{i/k} \right) \left(\frac{t_f}{t}\right)^{k-1} \prod_{j=2}^{k-1}  \left(\frac{d_j}{d_j^f}\right)^{\binom{k-1}{j}} \ge (1- k k! \eps_k^{1/k}) \left(\frac{t_f}{t}\right)^{k-1} \prod_{j=2}^{k-1}  \left(\frac{d_j}{d_j^f}\right)^{\binom{k-1}{j}}\]
$(k-1)$-edges in $\mathcal{H}_{\eps_k}$ that are also contained in $E_1$.
\end{claimproof}

\textbf{Partitioning the vertex sets.}
For $i=1,\dots,k$ we let $\mathcal{C}_i$ be the family of fine $1$-cells in $V_i$ that correspond to $1$-edges of $\mathcal{H}_{\eps_k}(G;H,P_f^*)$.
For building the paths we have to avoid $S$.
To avoid clash of vertices during the expansion in the coarse partition,
we set aside a small fraction of the fine cells that do not overlap with $S$ too much.
For $i=1,\dots,k$ let $\overline{\mathcal{C}_i} \subseteq \mathcal{C}_i$
be selections of fine $1$-cells $C$ with $|C \setminus S| \ge \eps_k^{1/2k} \tfrac{n}{t_f}$ such that
\[\sum_{C\in\overline{\mathcal{C}_i}} |C| = \frac{\nu}{16 k^2} |V_i|\pm \frac{t|V_i|}{t_f}\,.\]
This is possible, because by the condition on $S$ and $\eps_k$ in~\eqref{eq:def_epsk} we have $|V_i \setminus S| \ge \eta \tfrac nt \ge 2 \eps_k^{1/2k} \tfrac nt$ and there can be at most $\frac{t_f}{t} \eps_k^{1/2k} \frac{n}{t_f} = \eps_k^{1/2k} \frac{n}{t}$ vertices of $V_i \setminus S$ in $1$-cells $C \in C_i$ with $|C \setminus S|<\eps_k^{1/2k} \tfrac{n}{t_f}$.
Then for $i=1,\dots ,k$ we define $\overline{V}_i=\cup_{C \in \overline{\mathcal{C}_i}} C\setminus S$ and get
\begin{align}
|\overline{V}_i| \ge \eps_k^{1/2k} \frac{\nu}{16 k^2} |V_i| -\frac{t|V_i|}{t_f}\ge \frac{\eps_k^{1/2k} \nu}{20 k^2} \frac{n}{t}\,.
\end{align}
We will use $\overline{\mathcal{C}_i}$ for the expansion and then in the end use $\mathcal{C}_i$ to reach a $(1-\nu)$-fraction of $E_k$.

We now want to argue that the fine $1$-cells from the $\mathcal{C}_i$ are enough to reach a large fraction of any $E_i$.
For this we let $\cH'=\cH_{\eps_k}'(G;H,\cP_f^*,\mathcal{C}_1,\dots,\mathcal{C}_k)$ be the sub-multicomplex of $\cH_{\eps_k}$ induced by $\cC_j$, $j=1,\dots,k$.
For any $i=1,\dots,k$, by Claim~\ref{claim:goodfinecells}, there are at most
\begin{equation}
\label{eq:goodfineHp}
2 k k! \eps_k^{1/k} \left(\frac{n}{t}\right)^{k-1} \prod_{j=2}^{k-1}  d_j^{\binom{k-1}{j}} \le \nu \frac{\eps_k \nu^{k}}{10^{\ell} 20^k k^{2k} 2} \left(\frac{n}{t}\right)^{k-1} \prod_{j=2}^{k-1}  d_j^{\binom{k-1}{j}}
\end{equation}
$(k-1)$-tuples in $E_i$ that are not contained in a $(k-1)$-cell of $\mathcal{H}_{\eps_k}(G;H,\cP_f^*)$, where the inequality follows from the choice of $\eps_k$ in~\eqref{eq:def_epsk}.
Therefore, for $i=1,\dots,k$ at least $(1-\tfrac \nu4)|E_i|$ tuples from $E_i$ are contained in fine $(k-1)$-cells of $\cH'$.
This justifies that it will be sufficient to restrict the fine expansion to $\cH'$.

For the coarse expansion, similarly to $\cH'$, let $\overline{\cH}=\overline{\cH_{\eps_k}}(G;H,\cP_f^*,\overline{\mathcal{C}_1},\dots,\overline{\mathcal{C}_k})$ be the sub-multicomplex of $\cH_{\eps_k}$ induced by $\overline{\cC_j}$, $j=1,\dots,k$.
Then, let $\overline{E}_1,\dots,\overline{E}_k$ be the $(k-1)$-cells of $H$ restricted to the $1$-cells $\overline{V}_i$, where $\overline{E}_i\subset \prod_{j\not=i}\overline{V}_j$ for $i=1,\dots,k$.
Then, for $i=1,\dots,k$, we have with Lemma~\ref{lem:DCL}
\[ |\overline{E}_i| \ge \frac{1}{2} \prod_{j \not= i} |\overline{V}_i| \prod_{j=2}^{k-1} d_j^{\binom{k-1}{j}} \ge \frac{\eps_k \nu^{k}}{20^k k^{2k} 2}  \left(\frac{n}{t}\right)^{k-1} \prod_{j=2}^{k-1}  (d_j)^{\binom{k-1}{j}},\]
which with~\eqref{eq:goodfineHp} gives that at least 
\begin{equation}
\label{eq:goodfineEip}
(1-\nu 10^{-\ell})|\overline{E}_i| \text{ tuples from } \overline{E}_i  \text{ are contained in fine }(k-1) \text{-cells of } \overline{\cH}.
\end{equation}
This will be essential for the coarse expansion.

\textbf{Preparing for coarse expansion.}
We will construct paths starting in a $(1-\nu)$-fraction of the tuples from $E_k$. In order to do this, we need to know that most of these tuples have high degree into $\overline{C}_k$, and that the tuples we then reach have high degree into $\overline{C}_1$, and so on. The following claim allows us to get this.

\begin{claim}
	\label{claim:highdgreecells}
	Let $\nu' \in (0,1)$ with $\nu' \ge \tfrac14 \eps_k^{1/k}$ and $\nu' \ge 3 \gamma_c \prod_{i=2}^{k-1} d_i^{-\binom{k}{i}}$.
	Let $U_i \subseteq V_i$ for $i=1,\dots,k$ be subsets of size at least $\frac{\eps_k^{1/2k} \nu}{20 
	k^2} \frac{n}{t}$ and for $i=1,\dots,k$ let $F_i$ be the $(k-1)$-edges of $\cH'$ in $\prod_{j \not= i} U_j$.
	Then for any $i_1 \not=i_2$ the number of $(k-1)$-tuples from $F_{i_1}$ which are contained in less than
	\begin{equation}
	\label{eq:highdegree}
	\frac{p}{2}  |U_{i_1}| d \prod_{j=2}^{k-1} d_j^{\binom{k-1}{j-1}}
	\end{equation}
	edges of $G$ that are supported by $\cH_{\eps_k}$ and any $(1-\nu')$-fraction $\hat{F}_{i_2}$ of $F_{i_2}$ is at most a $10 \nu'$-fraction of $F_{i_1}$.
\end{claim}

\begin{claimproof}[Proof of Claim~\ref{claim:highdgreecells}]
W.l.o.g.~let $i_1=1$ and $i_2=2$.
First we note that analogous to~\eqref{eq:goodfineEip} with~\eqref{eq:goodfineHp} we get that $F_1$ is at least a $(1-\tfrac{\nu'}{10})$-fraction of all $(k-1)$-tuples of $\cH$ supported by the respective $1$-cells.
Now fix any $(1-\nu')$-fraction $\hat{F}_2$ of $F_2$.
Next let $\hat{F}_1 \subseteq F_1$ be those $(k-1)$-tuples in $F_1$ which are contained in less than
\[\frac{p}{2} |U_1| d \prod_{i=2}^{k-1} d_i^{\binom{k-1}{i-1}}\]
edges of $G$ that are supported by $\cH$ and $\hat{F}_2$ and assume for a contradiction that
\[|\hat{F}_1| \ge 10 \nu' |F_1| \ge \frac{89}{10} \nu' \prod_{i=2}^k |U_i| \prod_{i=2}^{k-1} d_i^{\binom{k-1}{i}},\]
where the latter inequality holds by Lemma~\ref{lem:DCL} and as $10 (1-\gamma_f)(1-\tfrac{\nu'}{10}) \ge \tfrac{89}{10}$.
By the first part of Lemma~\ref{lem:DCL-variant} at least $(1-\tfrac{1}{100})|\hat{F}_1|$ of these are contained in at least $(1-\tfrac{1}{100})|U_1| \prod_{i=2}^{k-1}d_i^{\binom{k-1}{i-1}}$ $k$-vertex complete $(k-1)$-graphs in ${\cH}$ using vertices from $U_1$.

On the other hand, with Lemma~\ref{lem:DCL}, there are at most
\[\nu'|F_2| \le  \left(1+\frac{1}{100}\right) \nu' \prod_{i\not=2} |U_i| \prod_{i=2}^{k-1}d_i^{\binom{k-1}{i}}\]
$(k-1)$-tuples in $F_2 \setminus \hat{F}_2$.
By the first part of Lemma~\ref{lem:DCL-variant} all but $2\gamma_c |F_2|$ of these are contained in at most $(1+\tfrac{1}{100}) |U_2| \prod_{i=2}^{k-1}d_i^{\binom{k-1}{i-1}}$ $k$-vertex complete $(k-1)$-graphs.
By Claim~\ref{claim:goodfinecells}, the first part of Lemma~\ref{lem:DCL-variant},~\eqref{eq:def_epsk}, and the bound on $\gamma_c$ there are at most \[4 k^2 6^k k! \eps_k^{1/k} \left( \frac{n}{t} \right)^k \prod_{i=2}^{k-1} d_i^{\binom{k}{i}} + \gamma_c k^2 6^k k! \eps_k^{1/k} \left( \frac{n}{t} \right)^k \prod_{i=2}^{k-1} d_i^{\binom{k-1}{i}} \le \frac{\nu'}{2} \prod_{i=1}^k |U_i| \prod_{i=2}^{k-1} d_i^{\binom{k}{i}} \]
$k$-vertex complete $(k-1)$-graphs supported in $\cH$ but not in $\cH_{\eps_k}$.
So with $2\gamma_c (1+\gamma_c) \le \frac{\nu'}{2} \prod_{i=2}^{k-1} d_i^{\binom{k}{i}}$ this gives us at least
\begin{align*}
\left(1-\frac{1}{100}\right)^2|\hat{F}_1||U_1| \prod_{i=2}^{k-1}d_i^{\binom{k-1}{i-1}} -\left( 1+ \left(1+\frac{1}{100}\right)^2\right) \nu' \prod_{i=1}^{k} |U_i| \prod_{i=2}^{k-1}d_i^{\binom{k}{i}} \ge \frac{3}{4}|\hat{F}_1| |U_1| \prod_{i=2}^{k-1}d_i^{\binom{k-1}{i-1}}
\end{align*}
$k$-vertex complete $(k-1)$-graphs in $\cH_{\eps_k}$ that are also in $\hat{F_2}$.

By definition of $\hat{F}_1$, they support at most
\[ |\hat{F}_1| \cdot \frac{p}{2} |U_1| d \prod_{i=2}^{k-1} d_i^{\binom{k-1}{i-1}}\]
edges of $G$.
This gives a $p$-density of at most $\tfrac d2 (\tfrac 34)^{-1}=\frac{2}{3}d$ and, therefore, by $(\varepsilon_k,1)$-regularity of $H$ (recall that the relative $p$-density of $G$ with respect to $H$ is at least $d$), we have
\[\frac{3}{4}  |\hat{F}_1| |U_1|  \prod_{i=2}^{k-1} d_i^{\binom{k-1}{i-1}}  \le \varepsilon_k \left( \frac{n}{t} \right)^{k} \prod_{i=2}^{k-1} d_i^{\binom{k}{i}},\]
which with~\eqref{eq:def_epsk}  and $\nu' \ge \tfrac14 \eps_k^{1/k}$ implies that
\[|\hat{F}_1| < 2 \varepsilon_k \left(\frac{20k^2}{\eps_k^{1/2k} \nu}\right)^{k} \prod_{i=2}^k |U_i| \prod_{i=2}^{k-1} d_i^{\binom{k-1}{i}}
\le 2 \eps_k^{1/k} \prod_{i=2}^k |U_i| \prod_{i=2}^{k-1} d_i^{\binom{k-1}{i}}
\le 8 \nu' \prod_{i=2}^k |U_i| \prod_{i=2}^{k-1} d_i^{\binom{k-1}{i}}\]
and we have the desired contradiction.
\end{claimproof}

We let $\ell_0 = \ell- k+1$. As discussed, there are two stages to building our paths. Given a $(k-1)$-tuple $\tpl{x}$, we first look at all the ways to add $\ell_0$ vertices, all contained in the $\overline{V}_i$, to get tight paths. We then use the fine partition to complete these by adding $k-1$ vertices to get to a $(1-\tfrac \nu4)$-fraction of most fine $(k-1)$-cells in $E_1$, which then gives a $(1-\nu)$-fraction of $E_1$.

We define the sequence $r_i$ for $i=0,\dots,\ell$ with $r_0=k$ and $r_{i+1} \equiv r_i+1 \pmod{k}$ for $i=0,\dots,\ell-1$.
As $\ell \equiv 1 \pmod{k}$, we have $r_{\ell}=1$ and $r_{\ell_0} = 2$.

Let $F_{\ell_0}\subset \overline{E}_2$ denote those $(k-1)$-edges which are contained in $(k-1)$-edges of $\overline{\cH}$. That is, every $(k-1)$-edge of $F_{\ell_0}$ has vertices in $\overline{V}_1,\overline{V}_3,\dots,\overline{V}_k$ and is in a fine $(k-1)$-cell which (and all of whose supporting $j$-cells) satisfies~\ref{fr:cells}.

Let $\nu_0=\nu$, for $i=1,\dots,\ell_0-1$ let $\nu_i=\nu_{i-1}/10$, and observe that with~\eqref{eq:def_epsk} and the bounds on $\nu$ and $\gamma_c$ we have that $\nu_i$ satisfies the requirements of Claim~\ref{claim:highdgreecells} for $i=1,\dots,\ell_0-1$.
We apply Claim~\ref{claim:highdgreecells} to obtain that a $(1-\nu_{\ell_0+2k-4})$-fraction $F_{\ell_0-1}$ of the edges of $\overline{E}_{r_{\ell_0-1}}$ have high degree as in~\eqref{eq:highdegree} into $\overline{V}_{r_{\ell_0+2k-4}}$ with respect to $F_{\ell_0}$.
Repeating this for $i=\ell_0-2,\dots,2$ we get that a $(1-\nu_i)$-fraction $F_{i}$ of the edges of $\overline{E}_{r_{i}}$ have high degree as in~\eqref{eq:highdegree} into $\overline{V}_{r_i}$ with respect to $F_{i+1}$. Finally, repeating the same procedure except replacing $\overline{E}_k$ with $E_k$, we
arrive at a $(1-\nu)$-fraction $F_0$ of the tuples from $E_k$. This set $F_0$ is the set in the lemma statement from which we can construct paths, and we have verified that it is sufficiently large for the lemma statement. We now need to justify that we can indeed construct paths from any tuple in $F_0$ of length $\ell$ to most of $E_1$ avoiding any given small $S'$.

\textbf{Coarse expansion.}
We fix any tuple $\tpl{x} \in F_0$ and any set $S'$ of size at most $s$. We let $P$ be the set of tight paths starting at $\tpl{x}=(x_1,\dots,x_{k-1})$ which can be constructed as follows.

For each $1\le j\le \ell_0$ in succession, we pick a vertex $x_{k+j-1}$ such that $\{x_j,\dots,x_{k+j-1}\}$ is an edge of $G$, and $\{x_{j+1},\dots,x_{k+j-1}\}$ is an edge of $F_j$. In addition, we insist that $x_{k+j-1}$ is not in $S'$, nor equal to $x_i$ for any $i<k+j-1$.

Because of our choice of $F_{j-1}$ to satisfy~\eqref{eq:highdegree}, at any given step $j$, the number of choices for $x_{k+j-1}$ is at least
\[\frac{p}{2}\big|\overline{V}_{r_j}\big|d\prod_{i=2}^{k-1}d_i^{\binom{k-1}{i-1}}-s-k-\ell_0\ge\frac{p}{4}\big|\overline{V}_{r_j}\big|d\prod_{i=2}^{k-1}d_i^{\binom{k-1}{i-1}}\,.\]
It follows that
\[|P|\ge\prod_{j=1}^{\ell_0} \left(\frac{p}{4}\big|\overline{V}_{r_j}\big|d\prod_{i=2}^{k-1}d_i^{\binom{k-1}{i-1}} \right) \ge \left( \frac{1}{4} pd \frac{\eps_k^{1/k} \nu n}{20 k^2t} \prod_{i=2}^{k-1} d_i^{\binom{k-1}{i-1}}  \right)^{\ell_0} \ge (\mu p n)^{\ell_0}\]
paths starting in $\tpl{x}$ and ending with tuples in $F_{\ell_0}\subset \overline{E}_{2}$, where the last inequality follows from~\eqref{eq:delta}.
By the good event of Lemma~\ref{lem:goodexpansion}, since $\ell_0 \ge \frac{k-1}{\gamma}$ we get~that for all $\tpl{x}$ there are at least 
\[\frac{\mu^{2\ell_0}}{8(2\ell_0)!} n^{k-1} \ge \delta \prod_{j\in\{1,3,\dots,k\}}|\overline{V}_j| \prod_{i=2}^{k-2}d_i^{\binom{k-1}{i}}\]
different end-tuples from the end-tuples above, where the lower bound follows from~\eqref{eq:delta}.

\textbf{Expansion in the fine partition.}
We pick a $(k-1)$-cell $\hat{C}_0$ of $\overline{\cH}$ such that at least $\delta|\hat{C}_0|$ of the $(k-1)$-edges in $\hat{C}_0$ are end-tuples of paths in $P$, which is possible by averaging. Let $R'_0$ denote the subset of $\hat{C}_0$ which are end $(k-1)$-tuples of paths in $P$.

We now consider all the $(k-1)$-cells $\hat{C}_{k-1}$ in $\cH'$, whose $1$-cells are contained in 
$\cC_2 \setminus \overline{\cC}_2, \dots,\cC_k \setminus \overline{\cC}_k$
respectively, which we can obtain by the following procedure.
For each $1\le j\le k-1$ in succession, we pick a $(k-1)$-cell $\hat{C}_{j}$ in $\cH'$ such that $\hat{C}_{j-1}$ and $\hat{C}_j$ are in the boundary of some $k$-edge of $\cH'$ and the new $1$-cell is in $\cC_{j+1} \setminus \overline{\cC}_{j+1}$.

For any such choice of $\hat{C}_{k-1}$, we do the following. Let $m=\eps_k^{1/2k}\tfrac{n}{t_f}$. We choose $X_0,\dots,X_{k-2}$ subsets of the $1$-edges of $\hat{C}_0$, each of size $m$, which contain a maximum number of the sets $R'_0$, and let the contained sets be $R_0$. By averaging, we have
\[|R_0|\ge\tfrac12\delta m^{k-1}\prod_{i=2}^{k-1}d_i^{\binom{k-1}{i}}\,.\]
We choose any $X_{k-1},\dots,X_{2k-3}$ subsets of the $1$-edges of $\hat{C}_{k-1}$ each of size $m$. To complete the setup for Lemma~\ref{lem:finconnect}, we put for each $2\le i\le k-1$ an $i$-graph consisting of the restrictions of the $i$-cells in the union of the $\hat{C}_j$ to the $X_j$. We obtain the required $(\tpl{d}_f,\sqrt{f},1)$-regularity from Lemma~\ref{lem:RRL}; and we let $G_k$ of Lemma~\ref{lem:finconnect} be the supported subgraph of $G$, which as previously observed is $\big(\tfrac{d}{2},f_k,p\big)$-regular with respect to each of the required polyads. Thus Lemma~\ref{lem:finconnect} which returns a set $R_{k-1}\subset \hat{C}_{k-1}[X_{k-1},\dots,X_{2k-3}]$ of size at least $(1-\delta)\big|\hat{C}_{k-1}[X_{k-1},\dots,X_{2k-3}]\big|$ of $(k-1)$-tuples which, together with some tuple of $R_0$, make a tight path in $G$. Apart from the at most $kn^{k-1}$ of these tuples which share a vertex with $\tpl{x}$, these are end-tuples of tight paths of length $\ell$ from $\tpl{x}$ whose internal vertices avoid $S\cup S'$. By averaging, we conclude that at least $(1-2\delta)\big|\hat{C}_{k-1}\big|$ of the edges of $\hat{C}_{k-1}$ 
are ends of tight paths of length $\ell$ from $\tpl{x}$ whose internal vertices avoid $S\cup S'$.

\begin{claim}\label{cl:connpart:valid}
There are at least
\[\big(1-\tfrac{\nu}{2k}\big)\big(\tfrac{t_f}{t}\big)^{k-1}\prod_{i=2}^{k-1}\left( \frac{d_i}{d_i^f}\right)^{\binom{k-1}{i}}\]
valid choices of $\hat{C}_{k-1}$.
\end{claim}

Assuming this claim, by Lemma~\ref{lem:DCL} we conclude that there are at least
\[\big(1-\tfrac{\nu}{2k}\big)\big(\tfrac{t_f}{t}\big)^{k-1}\prod_{i=2}^{k-1} (\tfrac{d_i}{d_i^f})^{\binom{k-1}{i}}(1-2 \delta)(1-\gamma_f) \big(\tfrac{n}{t_f}\big)^{k-1}\prod_{i=2}^{k-1}(d_i^f)^{\binom{k-1}{i}}\ge \big(1-\tfrac{\nu}{2}\big)\big(\tfrac{n}{t}\big)^{k-1}\prod_{i=2}^{k-1}d_i^{\binom{k-1}{i}} \]
$(k-1)$-tuples in $E_1$ which are ends of tight paths from $\tpl{x}$ whose internal vertices avoid $S\cup S'$, as required.

\begin{claimproof}[Proof of Claim~\ref{cl:connpart:valid}]
Recall that for each $i\in[k]$ we have
\[\sum_{C\in \overline{\mathcal{C}}_i}|C|\le \tfrac{\nu}{16k^2}|V_i|+t|V_i|/t_f\,.\]
Thus the number of fine $1$-cells in $\overline{\mathcal{C}}_i$ is at most $\tfrac{\nu}{8k^2}\cdot \tfrac{t_f}{t}$ for each $i$, and so the number of $(k-1)$-cells of $\cH$ which cannot be $\hat{C}_{k-1}$ due to being supported by a $1$-cell in some $\overline{\mathcal{C}}_i$ is at most
\[2 k\tfrac{\nu}{8k^2}\big(\tfrac{t_f}{t}\big)^{k-1}\prod_{i=2}^{k-1}\big(\tfrac{d_i}{d^f_i}\big)^{\binom{k-1}{i}}\,.\]

We consider constructing the $\hat{C}_j$ in order $j=1,\dots,k-1$, ignoring the restriction of not using $1$-cells in the $\overline{\mathcal{C}}_i$, and at each step $j$ keep track of the number of ways to construct $\hat{C}_j$ which will lead to different choices for $\hat{C}_{k-1}$.

Given $j\ge1$, suppose we have fixed $\hat{C}_{j-1}$. This means that we have already fixed the $(j-1)$-cell supporting $\hat{C}_{k-1}$ on $V_2,\dots,V_j$, and when we choose $\hat{C}_j$ we will fix the $i$-cells supporting $\hat{C}_{k-1}$ for each $1\le i\le j$ that are on the vertex set $V_{j+1}$ and some $i-1$ of $V_2,\dots,V_j$. Suppose that we have a particular choice $X$ of all these $i$-cells with $1\le i\le j$ that are on the vertex set $V_{j+1}$ and some $i-1$ of $V_2,\dots,V_j$, which is consistent with the choice of $\hat{C}_{j-1}$ (i.e. each $i$-cell in $X$ has boundary whose $(i-1)$-cell not on $V_{j+1}$ is in the support of $\hat{C}_{j-1}$). Consider the number of $k$-polyads in $\cH$ on $V_1,\dots,V_k$ which are consistent with both $\hat{C}_{j-1}$ and $X$. The number of these is at most
\[2 \prod_{i=2}^{k-1}\big(\tfrac{d_i}{d^f_i}\big)^{\binom{k-1}{i-1}-\binom{j-1}{i-1}}\,.\]
Suppose that more than
\[k2^{k+4}\eps_k^{1/k}\big(\tfrac{t_f}{t}\big)\prod_{i=2}^{k-1}\big(\tfrac{d_i}{d^f_i}\big)^{\binom{j-1}{i-1}}\]
choices of $X$ are not consistent with any $k$-edge of $\cH'$ whose boundary contains $\hat{C}_{j-1}$. Then in particular $\hat{C}_{j-1}$ is contained in less than
\[\big(1-k2^{k+4}\eps_k^{1/k}\big)\big(\tfrac{t_f}{t}\big)\prod_{i=2}^{k-1}\big(\tfrac{d_i}{d^f_i}\big)^{\binom{k-1}{i-1}}\]
$k$-edges of $\cH'$, which is a contradiction to~\ref{fr:cells}.

We conclude that the number of ways to construct $\hat{C}_{k-1}$, ignoring the restriction to avoid the $\overline{\mathcal{C}}_i$, is at least
\[\prod_{j=1}^{k-1}\Bigg(\big(1-k2^{k+5}\eps_k^{1/k}\big)\big(\tfrac{t_f}{t}\big)\prod_{i=2}^{k-1}\big(\tfrac{d_i}{d^f_i}\big)^{\binom{j-1}{i-1}}\Bigg)\ge \big(1-k^22^{k+5}\eps_k^{1/k}\big)\big(\tfrac{t_f}{t}\big)^{k-1}\prod_{i=2}^{k-1}\big(\tfrac{d_i}{d^f_i}\big)^{\binom{k-1}{i}}\,.\]

Finally the number of valid choices of $\hat{C}_{k-1}$ is at least
\[\big(1-\tfrac{\nu }{4k}-k^22^{k+5}\eps_k^{1/k}\big)\big(\tfrac{t_f}{t}\big)^{k-1}\prod_{i=2}^{k-1}\big(\tfrac{d_i}{d^f_i}\big)^{\binom{k-1}{i}}\ge\big(1-\tfrac{\nu }{2k}\big)\big(\tfrac{t_f}{t}\big)^{k-1}\prod_{i=2}^{k-1}\big(\tfrac{d_i}{d^f_i}\big)^{\binom{k-1}{i}}\,,\]
as required.
\end{claimproof}
\end{proof}

\section{Proof of reservoir lemma}
\label{sec:proofreservoir}

In this section we prove Lemma~\ref{lem:respath}.
Before giving the details we outline the strategy of the proof.
We fix $G$ and let $R \subseteq V(G)$ be a set of size $r=|R| \le \nu n$.
For every $u \in R$ we need a reservoir path $P_u$ with reservoir set $\{u\}$ on a constant number of vertices with end-tuples $\tpl{v}_u$ and $\tpl{w}_u$; this is a tight path with end-tuples $\tpl{v}_u$ and $\tpl{w}_u$ and vertex set $V(P_u)$ such that there also exists a tight path with the same end-tuples and vertex set $V(P_u) \setminus \{ u \}$.

To build $\Pres$ we begin with an arbitrary $(k-1)$-tuple $\tpl{v}=\tpl{w}$, which is a reservoir path with an empty reservoir set, and call this $P_0$. 
Assume we have built a reservoir path $P_{i-1}$ with reservoir set $R' \subseteq R$ of size $i-1$ and end-tuples $\tpl{v}_{i-1}$ and $\tpl{w}_{i-1}$  such that $V(P_{i-1})$ does not intersect $R \setminus R'$.
Then, for some $u \in R \setminus R'$, we construct a reservoir path $P_u$ with end-tuples $\tpl{v}_u$ and $\tpl{w}_u$ that is disjoint from $P_{i-1}$.
If $i$ is odd we connect $\tpl{w}_{i-1}$ to $\tpl{v}_u$ by a tight path (using Lemma~\ref{lem:connecting}) and let $\tpl{w}_i=\tpl{w}_u$ and $\tpl{v}_i=\tpl{v}$. If $i-1$ is even we connect $\tpl{v}_{i-1}$ to $\tpl{w}_u$ by a tight path and let $\tpl{v}_i=\tpl{v}_u$ and $\tpl{w}_i=\tpl{w}$.
In both cases we obtain a reservoir path $P_i$ with reservoir set $R' \cup \{ u \}$, end-tuples $\tpl{v}_i$ and $\tpl{w}_i$, and continue.
By alternating between the endpoints we ensure that the end-tuples are always $(\eps',p,\ell')$-good for $V(\Pres')$.

Finally, let us sketch how we construct $P_u$, a picture of which (for $k=5$) is in Figure~\ref{fig:reservoir}. We begin by finding a $(2k-1)$-vertex tight path with $u$ its central vertex; this gives the spikes $\tpl{u}$ and $\tpl{x}_1$ in the figure. We look at all the ways to fill in the upper and lower spike paths in the figure. Using Lemma~\ref{lem:goodexpansion}, we see that from each we can get to a positive density of end-tuples. In particular, we can get to a positive density of each of two vertex-disjoint coarse $(k-1)$-cells in the regular partition, and two applications of Lemma~\ref{lem:finconnect} gives us the tuple $\tpl{v}$ connecting the paths, completing the spikes. We then use Lemma~\ref{lem:connecting} repeatedly to create the paths between pairs of spikes. The only point where we need to be a bit careful is to ensure that when creating the upper and lower spike paths, and when connecting them, we do not reuse vertices; for this purpose we randomly split the vertex set into three parts and use one for each of the upper spike path, the lower spike path, and the connection.

\begin{proof}[Proof of Lemma~\ref{lem:respath}]
Let $\gamma>0$, $k \ge 3$ and $\ell'$ be integers.
Then let $\ell$ be the smallest multiple of $k-1$ which is both larger than $\ell'$ and sufficiently large for Lemma~\ref{lem:connecting} with input $k$ and $\gamma$. Let $\delta=10^{-4\ell}/(32(2\ell)!k!)$.

Set $c = \tfrac{\ell}{k-1}\big(k-1+\ell)+2k-1+\ell$, and set $s=4c$. Let $0<\eps' \le \tfrac14\gamma$ and $0<d \le \tfrac{1}{8}\gamma$ be given.

Next, we let $t_0$ be large enough for Lemma~\ref{lem:goodconnected} with input $k$, $\gamma$, and $d$ and for Lemma~\ref{lem:connecting} with input as above, $d$, $\eta=\tfrac12$, and $\eps'$.
Now let $ 0< \nu_\sublem{lem:connecting}<\tfrac14 \gamma$ and $0 < \eps_k < (\gamma 2^{-k-4})^k$ be small enough for Lemma~\ref{lem:connecting} with input as above, and for $\sqrt{\eps_k}$ to play the role of $f'_k$ in Lemma~\ref{lem:finconnect} with input $k,\delta,d$. Let $\nu=\tfrac{1}{10c} \nu_\sublem{lem:connecting}$. In addition let $\eps,f,f_k:\mathbb{N}\to(0,1]$ tend to zero sufficiently fast for Lemma~\ref{lem:connecting} and such that $\eps(t)$ is small enough for Lemma~\ref{lem:RRL} with input $k$, $\alpha=\tfrac1{100}$ and density $d_0=\tfrac{1}{t}$, and such that for any $d_0$, if $1/t<d_0$ then $\sqrt{\eps(t)}$ is small enough to play the role of $f'$ in Lemma~\ref{lem:finconnect}, and be such that Lemma~\ref{lem:goodconnected} is applicable with inputs as above and minimum degree $\tfrac12+\gamma$, $\eps_k$, $10c\nu$. Finally let $\eta_\sublem{lem:goodconnected}$ be given by Lemma~\ref{lem:goodconnected}.

Given $n$, let $p \ge n^{-1+\gamma}$.
Suppose that $\Gamma=G^{(k)}(n,p)$ is in the good events of Lemma~\ref{lem:connecting} with input as above and Lemma~\ref{lem:goodtuples} with input $\tfrac12\eps'$ and $k$, and Lemma~\ref{lem:goodexpansion} with input $\gamma$, $k$, both $\ell$ and $\ell+1-k$, and $\mu=\tfrac1{100}$, with Lemma~\ref{lem:upperreg}, that $\Gamma$ and all of its subgraphs are $(\eta_{\sublem{lem:goodconnected}},p)$-upper regular.
Suppose $G \subseteq \Gamma$ with $\delta_{k-1}(G) \ge (\frac{1}{2} + \gamma) pn$, that $(\cP^*_c,\cP^*_f)$ is a $\big(t_0,t_1,t_2,\eps_k,\eps(t_1),f_k(t_1),f(t_2),p\big)$-strengthened pair for $G$, and that $t$ is the number of $1$-cells in $\cP^*_c$.
Let $\cR=\cR_{\eps_k,d}(G;\cP^*_c,\cP^*_f)$ be the $(\eps_k,d)$-reduced multicomplex of $G$ and $S$ be the union of the vertices that are not contained in $1$-cells of $\cR$. 
Then, by Lemma~\ref{lem:goodconnected}, $\cR$ contains at least $(1-4 \eps_k^{1/k})t$ $1$-edges and every induced subcomplex on at least $(1-10c\nu)t$ $1$-cells is tightly linked.
We let $S$ be the set of vertices of $V(G)$ that are not in $1$-edges of $\cR$ and note $|S| \le 4 \eps_k^{1/k} n \le \tfrac14 n$.
Finally let $R \subseteq V(G)$ with $|R| \le \nu n$.

Our goal is to construct a reservoir path $\Pres$ in $G$ with reservoir set $R$ and ends $\tpl{v}$ and $\tpl{w}$, such that $|V(\Pres)| \le c |R|$.
Moreover, we require that $\tpl{v}$ and $\tpl{w}$ are $(\eps',p,\ell')$-good for $V(\Pres) \cup S$.
For the construction of the path $\Pres$ we proceed as outlined above.
We start with an arbitrary $(k-1)$-tuple outside $R$ that is $(\tfrac12\eps',p,\ell)$-good for $S \cup R$ (which exists by the good event of Lemma~\ref{lem:goodtuples}), denote it by $\tpl{v}_0$, and set $\tpl{w}_0=\rev{\tpl{v}_0}$. Let this tuple be $P_0$, and let $S_{-1}:=S\cup R$.

\begin{figure}[htb]
	\begin{center} 
	\begin{tikzpicture}[ultra thick,scale=0.5, every node/.style={scale=0.5}]
	
	\node[circle,fill,draw,label=left:{\Huge \bf $u$}] (a) at (-10,1){};		
	\draw[rounded corners] (-9.5,-1.5) rectangle (-13.5,-0.5);	
	\draw[rounded corners] (9.5,0.5) rectangle (13.5,-0.5);	
	
	\foreach \a in {0,1,2,3}
	{
		\node[circle,fill,draw] (u\a) at (-10-\a,-1){};		
		\node[circle,fill,draw] (v\a) at (10+\a,0){};		
	}		
	
	\foreach \a in {0,...,5}
	{
		\draw[rounded corners] ({-8+\a*3},5.5) rectangle ({-7+\a*3},1.5);
		\draw[rounded corners] ({-8+\a*3},-5.5) rectangle ({-7+\a*3},-1.5);
		
		\foreach \b in {0,1,2,3}		
		{
			\node[circle,fill,draw] (x\a\b) at (-7.5+\a*3,5-\b){};
			\node[circle,fill,draw] (y\a\b) at (-7.5+\a*3,-5+\b){};
		}
	}
	
	\node at (-12,0){\Huge \bf $\tpl{u}$};
	\node at (11.5,1){\Huge \bf $\tpl{v}$};
	\node at (-6.5,0){\Huge \bf $P'_1$};
	\node at (-3.5,0){\Huge \bf $P'_2$};
	\node at (3.3,0){\Huge \bf $P'_{\ell^\ast-1}$};
	\node at (6.5,0){\Huge \bf $P'_{\ell^\ast}$};
	\node at (0,0){\Huge \bf $\dots$};	
	
	\node at (-7.5,6){\Huge \bf $\tpl{x}_1$};			
	\node at (-4.5,6){\Huge \bf $\tpl{x}_2$};		
	\node at (-1.5,6){\Huge \bf $\tpl{x}_3$};	
	\node at (1.5,6){\Huge \bf $\dots$};	
	\node at (4.5,6){\Huge \bf $\tpl{x}_{\ell^\ast-1}$};
	\node at (7.5,6){\Huge \bf $\tpl{x}_{\ell^\ast}$};
	
	\node at (7.5,-6){\Huge \bf $\tpl{y}_{\ell^\ast}$};
	\node at (4.5,-6){\Huge \bf $\tpl{y}_{\ell^\ast-1}$};
	\node at (1.5,-6){\Huge \bf $\dots$};	
	\node at (-1.5,-6){\Huge \bf $\tpl{y}_3$};
	\node at (-4.5,-6){\Huge \bf $\tpl{y}_2$};
	\node at (-7.5,-6){\Huge \bf $\tpl{y}_1$};

	\draw[line width=2pt] (u3)--(u2)--(u1)--(u0);
	\draw[line width=2pt] (v3)--(v2)--(v1)--(v0);
	
	\foreach \a in {0,...,5}
	{
		\draw[line width=2pt] (x\a0)--(x\a1)--(x\a2)--(x\a3)--(y\a3)--(y\a2)--(y\a1)--(y\a0);
	}
	
	\draw[line width=2pt] (x00)--(x10) (x20)--(x30) (x40)--(x50);
	\draw[line width=2pt] (y10)--(y20) (y30)-- (y40);
	
	\draw[line width=2pt] (y00)--(y10) (y20)--(y30) (y40)--(y50);
	\draw[line width=2pt] (x10)--(x20) (x30)-- (x40);
	
	\draw[line width=2pt] (u0)--(-10,-4.5) .. controls (-10,-5) ..(-9.5,-5)--(y00);
	\draw[line width=2pt] (v0)--(10,-4.5) .. controls (10,-5) .. (9.5,-5)--(y50);		
	
	\draw[line width=2pt] (u0) -- (a);
	\draw[line width=2pt] (a)--(-10,4.5) .. controls (-10,5) ..(-9.5,5)--(x00);
	\draw[line width=2pt] (v0)--(10,4.5) .. controls (10,5) .. (9.5,5)--(x50);

	\end{tikzpicture}
	\end{center}
	\caption{Reservoir structure in the case $k=5$ with $\ell^\ast = \ell/(k-1)$ for one vertex $u$ with two tight paths that both have end-tuples $\tpl{u}$ and $\tpl{v}$, where one is using all vertices and the other all but $u$.}
	\label{fig:reservoir}
\end{figure}
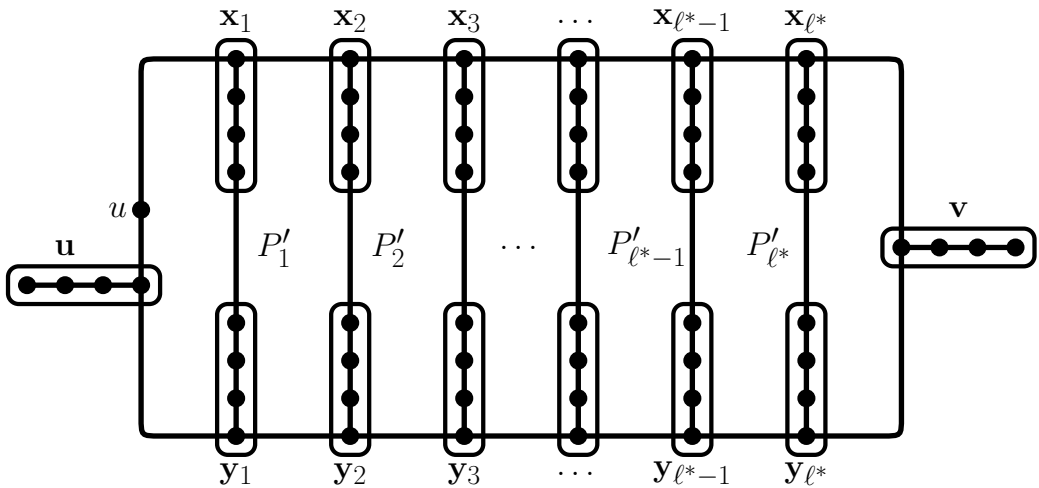

Suppose that for some $i \ge1$ we have constructed a reservoir path $P_{i-1}$ whose ends $\tpl{v}_{i-1}$ and $\tpl{w}_{i-1}$ are both $(\tfrac12\eps',p,\ell)$-good for $S_{i-2}$, with $|V(P_{i-1})|\le c(i-1)$, whose reservoir set is some $i$ vertices of $R$ and which does not intersect $R$ outside its reservoir set. Let $S_{i-1}:=S_{i-2}\cup V(P_{i-1})$.

We first choose any $u \in R\setminus S_{i-1}$ and construct a reservoir path $P_u$ with reservoir set $\{ u \}$ (see Figure~\ref{fig:reservoir}) disjoint from $S_{i-1} \setminus \{ u \}$.

Let $S^\ast$ be the set of vertices that are contained in $(k-1)$-tuples that are not $\big(\tfrac12\eps' , p, \ell\big)$-good for $S_{i-1}$. Note that we have $|S^\ast|=o(n)$ by the good event of Lemma~\ref{lem:goodtuples} assumed above. Let $S':=S_{i-1}\cup S^\ast$. We will construct $P_u$ disjoint from $S'\setminus \{u\}$. This automatically means that any $(k-1)$-tuple of vertices in $P_u$ is $\big(\tfrac12\eps',p,\ell\big)$-good for $S_{i-1}$.

We choose vertices $u_1,\dots,u_{k-2}$ from $V(G) \setminus (S^\ast \cup S_{i-1})$ such that the tuple $(u,u_1,\dots,u_{k-2})$ is $(\tfrac12 \eps',p,\ell_s+k-1)$-good for $S'$. This is possible since (however we choose $u_1,\dots,u_{k-3}$) when we come to choose $u_{k-2}$ we have at least $\tfrac12n$ vertices to choose from, and by the good event of Lemma~\ref{lem:goodtuples} at most $o(n)$ of these can give a $(k-1)$-tuple which is not $(\tfrac12 \eps',p,\ell_s+k-1)$-good for $S'$.

By the minimum degree of $G$, the $(k-1)$-tuple $(u,u_1,\dots,u_{k-2})$ is contained in at least $\tfrac12pn$ edges of $G$. Since $(u,u_1,\dots,u_{k-2})$ is $(\tfrac12\eps',p,\ell_s+k-1)$-good for $S'$, at most $p|S'|+\tfrac12\eps'pn$ of these edges go to vertices of $S'$, and at most another $\tfrac12\eps'pn$ of these are in $(k-1)$-tuples with any $k-2$ vertices of $(u,u_1,\dots,u_{k-2})$ that are not $(u,u_1,\dots,u_{k-2})$ is $(\tfrac12\eps',p,\ell_s+k-2)$-good for $S'$. Pick two of these vertices not in $S'$ that are in good tuples, and let them be $u_{k-1}$ and $x_{1,1}$.

We now, for each $2\le j\le k-1$ in succession, choose a vertex $x_{1,j}$ such that $\{u_{k-j-1}, \dots, u_1,u,\allowbreak x_{1,1},\allowbreak\dots,x_{1,j}\}$ is an edge of $G$, such that $x_{1,j}$ is not in $S'\cup\{u_1,\dots,u_k\}$, and such that if $j\le k-2$ then $\{u_{k-j-2},\dots, u_1,u,x_{1,1},\allowbreak\dots,x_{1,j}\}$ is $(\tfrac12\eps',p,\ell+k-j-1)$-good for $S'$, while if $j=k-1$ then $\{x_{1,1},\dots,x_{1,j}\}$ is $(\tfrac12\eps',p,\ell+k-j-1)$-good for $S'$. This is possible for each $j$ by the same argument as above.

At this point, we have constructed the top left part of Figure~\ref{fig:reservoir}: the spikes $\tpl{u}=(u_1,\dots,u_{k-1})$ and $\tpl{x}_1=(x_{1,1},\dots,x_{1,k-1})$. Our next goal is to construct the remaining spikes.

To begin with, consider constructing spike paths starting from $\tpl{x}_1$ of type $1$, with $\ell+1-k$ vertices. We can do this greedily, adding at each step $j$ one more vertex and one more edge. We insist on choosing our new vertex outside $S'$, and we insist on all the $(k-1)$-subsets of the $j$th edge being $(\tfrac12,p,\ell-j)$-good for $S'$. This is possible for each $j\le \ell+1-k$ by the same argument as above, and critically at each step we have a choice of at least $\tfrac{1}{10} pn$ vertices that satisfy these conditions. The same statement is true for spike paths of length $\ell$ starting from $\tpl{u}$.

To avoid clashes between the spike paths we construct, let $Z_1,Z_2,Z_3$ be subsets of $V(G)$ obtained by putting each vertex of $G$ independently into one of the three sets with equal probability $\tfrac13$. By the Chernoff bound and the union bound, with probability at least $1-n^k\exp\big(-\tfrac{1}{1600}p n\big)$, the following holds. When we construct spike paths from $\tpl{x}$ satisfying the above conditions and in addition with all new vertices in $Z_1$ greedily, at each step we have at least $\tfrac1{100} p n$ choices. Similarly, when we construct spike paths from $\tpl{u}$ satisfying the above conditions and in addition with all new vertices in $Z_2$ greedily, at each step we have at least $\tfrac1{100} p n$ choices. Finally, for each cluster $V'\in\cP_c$ such that $|V'\setminus S'|>\eps'|V'|$ and each $1\le j\le 3$, we have $|(V'\setminus S')\cap Z_j|\ge\tfrac14|V' \setminus S'|$.
Suppose that our choice of $Z_1,Z_2,Z_3$ is such that all these statements hold.

By Lemma~\ref{lem:goodexpansion}, there is a set $X$ of $(k-1)$-tuples which are end-tuples of spike-paths starting from $\tpl{x}_1$ of type $1$, with $\ell+1-k$ vertices other than those in $\tpl{x}_1$, none of whose vertices are in $S'$ and all of whose vertices outside $\tpl{x}_1$ are in $Z_1$, such that $|X|\ge \tfrac{10^{-4\ell-4+4k}}{8(2\ell+2-2k)!}n^{k-1}\ge\delta n^{k-1}$. Similarly, there is a set $U$ of $(k-1)$-tuples which are end-tuples of spike-paths starting from $\tpl{u}$ of type $1$, with $\ell$ vertices, none of whose vertices are in $S'$ and all of whose vertices outside $\tpl{u}$ are in $Z_2$, such that $|U|\ge \tfrac{10^{-4\ell}}{8(2\ell)!}n^{k-1}\ge\delta n^{k-1}$.

We now aim to find two disjoint $(k-1)$-cells $C_x$ and $C_u$ of $\cR$ with the following properties. At least $2k!\delta|C_x|$ of the $(k-1)$-tuples in $X$ are orderings of edges of $C_x$; and at least $2k!\delta|C_u|$ of the $(k-1)$-tuples in $U$ are orderings of edges of $C_u$; and for each cluster $V'$ supporting either $C_u$ or $C_x$, we have $|V'\setminus S'|\ge\tfrac12|V'|$.

Observe that since $|S'|\le 2c|R|$ and $|R|\le\nu n$ we have $|S'|\le 2c\nu n$, and consequently there are at most $5c\nu t-k$ clusters $V'$ of $\cR$ with $|V'\setminus S'|<\tfrac12|V'|$. The total number of $k$-sets intersecting $S'$ or $1$-cells with too many elements in $S'$ is at most $7c\nu n^k$, while by choice of $\nu$ the number of $k$-sets which are not supported by the $(k-1)$-cells of $\cR$ is at most $\nu n^k$. By choice of $\delta$, and since $|X|\ge4\delta\binom{n}{k-1}k!$, at least $3\delta\binom{n}{k-1}k!$ of the $(k-1)$-tuples in $X$ are on $(k-1)$-sets not intersecting $S'$, and by averaging the desired $C_x$ exists. A similar calculation, this time removing additionally elements of $U$ which lie in the clusters of $C_x$, gives $C_u$. Let $X''$ denote the tuples in $X$ which are orderings of edges of $C_x$. Choose an ordering $(V_{k-1},\dots,V_{1})$ of the clusters of $C_x$ which is consistent with most tuples of $X''$, and let $X'$ be the consistently ordered tuples of $X''$. Thus a spike path with end in $X'$ has its final vertex in $V_1$. We obtain $|X'|\ge2\delta|C_x|$. Similarly, we choose an ordering $(V_{2k-1},\dots,V_{3k-3})$ of the clusters of $C_u$ and let $U'$ be the consistently ordered tuples of $U$ in $C_u$, obtaining $|U'|\ge 2\delta|C_u|$.

Now by construction of $\cR$ there is a tight link between $C_x$ and $C_u$ with the given orderings that does not use any cluster $V'$ with $|V'\setminus S'|<\tfrac12|V'|$, because $|S'|<2c\nu n$ and hence less than $5c\nu t$ clusters are more than half covered by $S'$. Let the clusters witnessing this tight link be $V_k,\dots,V_{2k-2}$, and let $C'$ be the $(k-1)$-cell in the tight link whose clusters are $V_k,\dots,V_{2k-2}$. Thus for each $1\le j\le k-1$ the set $\{V_j,\dots,V_{j+k-1}\}$ and $\{V_{2k-1-j},\dots,V_k,V_{2k-1},\dots,V_{2k-2+j}\}$ are $k$-edges of $\cR$.

We now choose subsets $V'_1,\dots,V'_{3k-3}$ of $V_1,\dots,V_{3k-3}$ respectively, each of size $\tfrac{n}{10t}$ and disjoint from $S'$, with $V'_1,\dots,V'_{k-1}$ in $Z_1$, with $V'_k,\dots,V'_{2k-2}$ in $Z_3$, and with $V'_{2k-1},\dots,V'_{3k-3}$ in $Z_2$, and such that $\big|X'\cap C_x[V'_1,\dots,V'_{k-1}]\big|\ge2\delta\big| C_x[V'_1,\dots,V'_{k-1}]\big|$ and $\big|U'\cap C_u[V'_{2k-1},\dots,V'_{3k-3}]\big|\ge2\delta\big| C_u[V'_{2k-1},\dots,V'_{3k-3}]\big|$. Note that it is possible to find sets of the given sizes by choice of the $V_i$ and definition of the $Z_j$; while the condition about $X'$ and about $U'$ is satisfied on average and hence a choice exists. By Lemma~\ref{lem:RRL} all the $j$-cells with $2\le j\le k-1$ between these chosen sets are $(d_j,\sqrt{\eps},1)$-regular, and by definition of regularity it follows that the supported $k$-polyads corresponding to $k$-edges of $\cR$ are $(d,\sqrt{\eps_k},p)$-regular.

By Lemma~\ref{lem:finconnect}, with $R_0=X'\cap C_x[V'_1,\dots,V'_{k-1}]$, there is a set $R\subset C'$ with $|R|\ge(1-\delta)|C'|$ such that for each $(v_1,\dots,v_{k-1})\in R$
there is a tight path from a tuple of $C_x[V_1,\dots,V_{k-1}]$ to $(v_1,\dots,v_{k-1})$. Applying Lemma~\ref{lem:finconnect} again, with $R_0:=R$, there is a set $R'\subset C_u$ with $|R'|\ge(1-\delta)|C_u|$ such that for each $(y_{\ell/(k-1),1},\dots,y_{\ell/(k-1),k-1})\in R'$ there is a tight path from $R$ to $(y_{\ell/(k-1),1},\dots,y_{\ell/(k-1),k-1})$. In particular, we can choose $\tpl{y}_{\ell/(k-1)}:=(y_{\ell/(k-1),1},\dots,y_{\ell/(k-1),k-1})\in U' \cap R'$, and obtain the corresponding $\tpl{v}:=(v_1,\dots,v_{k-1})\in R$ which in turn gives us $\tpl{x}_{\ell/(k-1)}:=(x_{\ell/(k-1),1},\dots,x_{\ell/(k-1),k-1})\in X'$. This structure is the three right-hand-most spikes of Figure~\ref{fig:reservoir}. By definition of $X'$ and $U'$, and since $Z_1,Z_2,Z_3$ are by construction disjoint, there exist vertex-disjoint spike-paths completing the spikes of Figure~\ref{fig:reservoir}, none of whose vertices are in $S'$. That is, we find $(k-1)$-tuples $\tpl{x}_2,\dots,\tpl{x}_{\ell/(k-1)-1}$ and $\tpl{y}_1,\dots,\tpl{y}_{\ell/(k-1)-1}$ such that $\tpl{x}_1,\dots,\tpl{x}_{\ell/(k-1)}$ form a spike path, and $\tpl{y}_1,\dots,\tpl{y}_{\ell/(k-1)}$ form a spike path.

Next, for $j=1,\dots,\ell/(k-1)$, we use Lemma~\ref{lem:connecting} to connect the tuple $\tpl{x}_j$ to the tuple $\rev{\tpl{y}_j}$ by a tight path $P'_j$ of length $\ell$ with internal vertices not in $S \cup S_{i-1}$ and any of the previously chosen vertices.
Finally, we denote the whole structure by $P_u$ and note $|V(P_u)| = \tfrac{\ell}{k-1}\big(k-1+\ell)+2k-1$.
Observe (see Figure~\ref{fig:reservoir}) that there are two tight paths with end-tuples $\tpl{u}$ and $\tpl{v}$; one with vertex set $V(P_u)$ and one with vertex set $V(P_u) \setminus \{u\}$, i.e.~$P_u$ is a reservoir path with reservoir set $\{ u \}$.

To finish the step, we use Lemma~\ref{lem:connecting} to connect $\tpl{u}$ to one of the ends of $P_{i-1}$, either $\tpl{u}_{i-1}$ if $i$ is odd, or $\tpl{w}_{i-1}$ if $i$ is even. Repeating until $i=|R|$ proves the lemma.
\end{proof}

\appendix

\section{Regularity lemmas and properties}
\label{sec:proofreg}

We first prove Lemma~\ref{lem:ssshrl}. This we do in two steps: first, we prove the special case that $p_i=1$ for each $1\le i\le s$ and the $G_i$ are edge-disjoint (the \emph{dense disjoint case}), and then we deduce from this special case the general case. Note that in the dense case the assumption of upper regularity is trivially satisfied with $p_i=\eta=1$.

To prove the dense disjoint case of Lemma~\ref{lem:ssshrl}, we use a standard approach, borrowed from~\cite{AFKS}. That is, we begin with the input family of partitions $\cQ^*$, and iteratively apply the Strong Hypergraph Regularity Lemma of R\"odl and Schacht~\cite{RSch}, obtaining a collection of families of partitions $\cP_1^*$, $\cP_2^*$ and so on, where $\cP_1^*$ refines $\cQ^*$ and for each $j\ge2$ the family $\cP_j^*$ refines $\cP_{j-1}^*$. We choose parameters for these applications of the Regularity Lemma such that for each $j\ge1$ the pair $(\cP_j^*,\cP_{j+1}^*)$ satisfies the regularity properties of being a strengthened pair for each $G_i$. Then for each $j\ge1$, one of the following two things occurs. First, the density property of being a strengthened pair holds for each $G_i$. Second, there is some $G_i$ for which the density property does not hold. We will define an energy $\cE_j$ of the family of partitions $\cP_j^*$, and see that in the second case $\cE_{j+1}$ is significantly larger than $\cE_j$. Since we will see $\cE_j$ is bounded above by $s$, we conclude that the first case must occur for some bounded $j$, and the lemma follows. We now quote the Strong Hypergraph Regularity Lemma from~\cite{RSch}, and give the details.

\begin{lemma}[{\cite[Lemma 23]{RSch}}]\label{lem:shrl}
 Let $k\ge2$ be a fixed integer. For all positive integers $t_0$ and $s$, and all $\eps_k>0$ and functions $\eps:\mathbb{N}\to(0,1]$, there are integers $t_1$ and $n_0$ such that the following holds for all $n\ge n_0$ which are divisible by $t_1!$. Let $V$ be a vertex set of size $n$, let $G_1,\dots,G_s$ be edge-disjoint $k$-uniform hypergraphs on $V$, and suppose $\cQ^*$ is a $(1,t_0,\eta)$-equitable family of partitions on $V$. Then there exists a family of partitions $\cP^*$ refining $\cQ^*$ such that $\cP^*$ is $\big(t_0,t_1,\eps(t_1)\big)$-equitable, and for each $1\le i\le s$, the hypergraph $G_i$ is $(\eps_k,1)$-regular with respect to $\cP^*$.
\end{lemma}

Note that in~\cite{RSch} this lemma is stated for $k\ge3$; the case $k=2$ is the (much older) Szemer\'edi Regularity Lemma, in which the `families of partitions' are simply vertex partitions and the function $\eps$ plays no r\^ole.

\begin{proof}[Proof of Lemma~\ref{lem:ssshrl}, dense disjoint case]
 Given $k\ge2$, $q$, $t_0$, $s$ integers and $\eps_k>0$, and functions $f,f_k,\eps:\mathbb{N}\to(0,1]$, we define sequences of numbers $t_1,t_2,\dots$ and $n_1,n_2,\dots$ as follows. For each $j\ge1$, let $t_j$ and $n_j$ be returned by Lemma~\ref{lem:shrl} with input $k$, $t_{j-1}$, $s$, $\min\big(\eps_k,f_k(t_0)\big)$ and the function $\min\big(\eps,f\big)$.
 
 Let Lemma~\ref{lem:ssshrl} return the parameters $\eta=1$, $T=t_{s\eps_k^{-4}+2}$, and $n_0=\max(n_1,\dots,n_{s\eps_k^{-4}+2})$.
 Given an initial family of partitions $\cQ^*$ which is $(1,t_0,\eta)$-equitable and edge-disjoint $k$-uniform hypergraphs $G_1,\dots,G_s$, we proceed as follows.
 
 We apply Lemma~\ref{lem:shrl}, with input $k,t_0,s,\min\big(\eps_k,f(t_0)\big)$ and $\eps$, with input family of partitions $\cQ^*$, to the hypergraphs $G_1,\dots,G_s$. Let $\cP_1^*$ be the returned family of partitions.
 
 We now, for each $j\ge1$ successively such that the conditions are met, apply Lemma~\ref{lem:shrl}, with input $k$, $t_j$, $s$, $\min\big(\eps_k,f_k(t_j)\big)$ and $\min(\eps,f)$, with input family of partitions $\cP_j^*$, to the hypergraphs $G_1,\dots,G_s$. Let $\cP_{j+1}^*$ be the returned family of partitions.
 
 For each $j\ge1$ such that $\cP^*_j$ exists, we define
 \[\cE_j:=\binom{n}{k}^{-1}\sum_{i=1}^s\sum_{Q\in\cross_k(\cP_j)}d_1\big(G_i\big|\hat{P}(Q,\cP_j^*)\big)^2\,.\]
 Observe that if $j\ge1$ and $\cP^*_{j+1}$ exists, writing temporarily $D_{j,i}(Q):=d_1\big(G_i\big|\hat{P}(Q,\cP_{j}^*)\big)$ and $D_{j+1,i}(Q):=d_1\big(G_i\big|\hat{P}(Q,\cP_{j+1}^*)\big)$, we have
 \begin{align*}
  \cE_{j+1}-\cE_j&\ge\tbinom{n}{k}^{-1}\sum_{i=1}^s\sum_{Q\in\cross_k(\cP_j)}D_{j+1,i}(Q)^2-D_{j,i}(Q)^2\\
  &=\tbinom{n}{k}^{-1}\sum_{i=1}^s\sum_{Q\in\cross_k(\cP_j)}\big(D_{j,i}(Q)-(D_{j,i}(Q)-D_{j+1,i}(Q))\big)^2-D_{j,i}(Q)^2\\
  &=\tbinom{n}{k}^{-1}\sum_{i=1}^s\sum_{Q\in\cross_k(\cP_j)}-2D_{j,i}(Q)\big(D_{j,i}(Q)-D_{j+1,i}(Q)\big)+\big(D_{j,i}(Q)-D_{j+1,i}(Q)\big)^2\\
  &=\tbinom{n}{k}^{-1}\sum_{i=1}^s\sum_{Q\in\cross_k(\cP_j)}\big(D_{j,i}(Q)-D_{j+1,i}(Q)\big)^2\,,
 \end{align*}
 where the inequality comes from the fact that some $Q$ may be in $\cross_k(\cP_{j+1})$ but not $\cross_k(\cP_j)$, and the final equality is since (by definition of density) the sum of $D_{j,i}(Q)-D_{j+1,i}(Q)$ over any polyad of $\cP^*_j$ is zero, and $-2D_{j,i}(Q)$ is constant on any such polyad. From this we observe that $\cE_{j+1}-\cE_j$ is always nonnegative, and furthermore if there is some $1\le i\le s$ and some $\eps_k^2\binom{n}{k}$ choices of $Q$ such that $D_{j+1,i}(Q)\neq D_{j,i}(Q)\pm\eps_k$, then $\cE_{j+1}-\cE_j\ge\eps_k^4$.

 If for some $1\le j\le s\eps_k^{-4}+1$ the pair $(\cP_j^*,\cP_{j+1}^*)$ is a $(t_0,t_j,t_{j+1},\eps_k,\eps(t_j),f_k(t_j),f(t_{j+1}),1)$-strengthened pair for $G_i$ for each $1\le i\le s$, then by choice of $T$ we are done. It follows that for each $1\le j\le s\eps_k^{-4}+1$ the pair $(\cP_j^*,\cP_{j+1}^*)$ is not such a strengthened pair. By construction, the conditions~\ref{S:refine}--\ref{S:freg} in the definition of a strengthened pair are satisfied, and we conclude that~\ref{S:dens} fails, i.e.\ for some $1\le i\le s$ there are $\eps_k^2\binom{n}{k}$ elements $Q\in\cross_k(\cP_j)$ such that
 \[d_1\big(G_i\big|\hat{P}(Q,\cP_j^*)\big)\neq d_1\big(G_i\big|\hat{P}(Q,\cP_{j+1}^*)\big)\pm\eps_k\,.\]
 and so by our observation above $\cE_{j+1}-\cE_j\ge\eps_k^4$.
 
 Summing over $j$, we obtain $\cE_{s\eps_k^{-4}+1}-\cE_1\ge (s\eps_k^{-4}+1)\eps_k^4>s$, which by definition of $\cE_{s\eps_k^{-4}+1}$ is not possible, completing the proof.
\end{proof}

Next, we deduce the general case of Lemma~\ref{lem:ssshrl}, using the Weak Regularity Lemma of Conlon, Fox and Zhao~\cite{CFZ}. This follows the approach in~\cite{ADS}. Specifically, for each $G_i$ on vertex set $V$ we create a dense model $G''_i$ on $V$ by first using the Weak Regularity Lemma, which returns a dense model $G'_i$ of $G_i$ with weighted edges, and then randomly picking a $k$-edge into $G''_i$ with probability proportional to the weight of that edge in $G'_i$. The conclusion of the Weak Regularity Lemma, together with a simple application of the Chernoff bound, tell us that density and regularity properties with respect to $G''_i$ of any family of partitions $\cP^*$ on $V$ with sufficiently large parts carry over (with a small loss of parameters) to $G_i$. In particular, we can apply the dense disjoint case of Lemma~\ref{lem:ssshrl} to the $G''_i$ and the resulting strengthened pair is also a strengthened pair for the $G_i$. We now quote the Weak Hypergraph Regularity Lemma of~\cite{CFZ} from~\cite{ADS} (where a simplified statement which is all we need is given) and give the details. We need a couple of definitions.

 Let $V$ be a vertex set and let $k\ge2$. Let $g,h$ be two functions from $\binom{V}{k}$ to $\mathbb{R}_{\ge0}$, which we think of as weighted hypergraphs. Given any collection $F_1,\dots,F_k$ of $(k-1)$-uniform hypergraphs on $V$, let $S$ be the collection of $k$-sets in $V$ whose $(k-1)$-subsets can be labelled using each label from $1$ to $i$ exactly once, such that the label $i$ subset is in $F_i$ (we say the edges of $S$ are \emph{rainbow} for the $F_i$). If for any choice of the $F_i$ we have
 \[\Big|\sum_{e\in S}(g(e)-h(e))\Big|\le\gamma |V|^k,\]
 then $(g,h)$ is a \emph{$\gamma$-discrepancy pair}. In addition, given $\eta>0$, if for any choice of the $F_i$ we have
 \[\sum_{e\in S}\big(g(e)-1\big)\le \eta |V|^k,\]
 then we say $g$ is upper $\eta$-regular. Note that this definition is not quite the same as the previously defined $(\eta,p)$-upper regular; however if we have an $(\eta,p)$-upper regular hypergraph $G$, and we define a function $g$ by setting $g(e)=p^{-1}$ for edges $e$ of $G$, and $g(e)=0$ otherwise, then $g$ is upper $\eta$-regular.
 
\begin{theorem}[{\cite[Theorem~19]{ADS}, simplified from~\cite[Theorem 2.16]{CFZ}}]\label{thm:wrl}\label{thm:CFZweak}
 For any $k\ge2$ and $\gamma>0$ and $g:\binom{V}{k}\to\mathbb{R}_{\ge0}$ which is upper $\eta$-regular with $\eta\le 2^{-80k/\gamma^2}$, there exists $\tilde{g}:\binom{V}{k}\to[0,1]$ such that $(g,\tilde{g})$ form a $\gamma$-discrepancy pair.
\end{theorem}

The following proof is very similar to the proof of~\cite[Lemma~23, general case]{ADS} and we copy it from there, making the appropriate modifications, for completeness.

\begin{proof}[Proof of Lemma~\ref{lem:ssshrl}, general case]
 Given $k\ge2$, $q$, $t_0$, $s$ integers and $\eps_k>0$, and functions $f,f_k,\eps:\mathbb{N}\to(0,1]$, let $T$ and $n_0$ be returned by the dense disjoint case of Lemma~\ref{lem:ssshrl} for input as above but with $\tfrac1{2s}\eps_k$ and $\tfrac1{2s}f$ replacing $\eps_k$ and $f$. Without loss of generality, we can assume $\eps,f,f_k$ are decreasing functions.

We let $\gamma=\tfrac{1}{100s^4}\eps_k^2\eps(T)^2f(T)^2f_k(T)^2T^{-2^k}$, and $\eta=2^{-160k/\gamma^2}$, and let Lemma~\ref{lem:ssshrl} return $\eta$, $T$ and $n_0$.

Given an initial family of partitions $\cQ^*$ which is $(1,t_0,\eta)$-equitable and $k$-uniform hypergraphs $G_1,\dots,G_s$, where $G_i$ is $(\eta,p_i)$-upper regular for each $i$, we proceed as follows. 
First, for each $1\le i\le s$, let $g_i:\binom{[n]}{k}\to\mathbb{R}_{\ge0}$ be defined by $g_i(e)=p_i^{-1}$ if $e\in G_i$, and $g_i(e)=0$ otherwise. Observe that each $g_i$ is upper $\eta$-regular.

Applying Theorem~\ref{thm:CFZweak}, with input $\gamma$ separately to each $g_i$, we obtain functions $\tilde{g}_i:\binom{[n]}{k}\to[0,1]$, such that $(g_i,\tilde{g}_i)$ is a $\gamma$-discrepancy pair for each $i$. It follows that $(s^{-1}g_i,s^{-1}\tilde{g}_i)$ is also a $\gamma$-discrepancy pair for each $i$.
 
 We now create pairwise disjoint unweighted $k$-graphs $G''_i$ by, for each $e\in\binom{[n]}{k}$ independently, choosing to put $e$ into either exactly one of the $G''_i$, or into none of them, choosing to put $e$ in $G''_i$ with probability $s^{-1}\tilde{g}_i(e)$. Since $0\le \tilde{g}_i(e)\le 1$ for each $i$, we have $0\le \sum_{i\in[s]}\tilde{g}_i(e)\le s$, so that the distribution we just described is as required a probability distribution. By construction, the $G''_i$ are edge-disjoint.
 
  We claim that a.a.s.~$(s^{-1}\tilde{g}_i,G''_i)$ is a $\gamma$-discrepancy pair for each $i$ (where we temporarily abuse notation by equating $G''_i$ and the characteristic function of its edges). Indeed, suppose $i$ and unweighted $(k-1)$-graphs $F_1,\dots,F_k$ on $[n]$ are fixed before the sampling of the $G''_i$. The expected number of edges of $G''_i$ which are rainbow for $F_1,\dots,F_k$ is exactly equal to the sum of $g'_i(e)$ over $e$ rainbow for $F_1,\dots,F_k$. By the Chernoff bound, the probability of an additive error of $\gamma n^k$ is $o(2^{-kn^{k-1}})$. In other words, a given $F_1,\dots,F_k$ and $i$ witness the failure of our claim with probability $o(2^{-kn^{k-1}})$. Taking the union bound over the at most $s2^{kn^{k-1}}$ choices of $F_1,\dots,F_k$ and $i$, our claim fails with probability $o(1)$ as desired.
 
 We now apply the dense disjoint case of Lemma~\ref{lem:ssshrl} to the $G''_i$, with inputs as above. We obtain integers $t_1,t_2\le T$ and families of partitions $\cP_c^*$ and $\cP_f^*$, both refining $\cQ^*$, such that for each $1\le i\le s$ the pair $(\cP_c^*,\cP_f^*)$ is a $(t_0,t_1,t_2,\tfrac1{2s}\eps_k,\tfrac1{2s}\eps(t_1),f_k(t_1),f(t_2),1)$-strengthened pair for $G''_i$. We claim that this is the required strengthened pair for each $G_i$.

To see this, we need to check~\ref{S:creg},~\ref{S:freg} and~\ref{S:dens} hold. Expanding out the definition of regularity with respect to a family of partitions, each of these three statements boils down to a collection of claims of the following form. Given a $k$-polyad $\hat{P}$ (of either the coarse or fine partition) and a set of $(k-1)$-graphs $F_1,\dots,F_k$ which are subgraphs of $\hat{P}$, let $S$ be the set of $k$-edges rainbow for $F_1,\dots,F_k$. If $|S|$ is not too small, then $\tfrac{|E(G_i)\cap S|}{p_i|S|}$ is close to the relative density $d_{p_i}(G_i|\hat{P})$. Note that for each such condition, the corresponding statement for $G''_i$, namely that $\tfrac{|E(G''_i)\cap S|}{|S|}$ is close to $d_1(G''_i|\hat{P})$, is guaranteed by definition of a strengthened pair.

It is therefore enough to show the following: for any $(k-1)$-graphs $F_1,\dots,F_k$, letting $S$ be the set of $k$-edges rainbow for $F_1,\dots,F_k$, we have
 \[\big|E(G_i)\cap S\big|=p_is\big|E(G''_i)\cap S|\pm 2p_i s\gamma n^k\,.\]
 Note that this statement also implies $d_{p_i}(G_i|\hat{P})$ is close to $s d_1(G''_i|\hat{P})$ by taking the $F_j$ to be the $(k-1)$-partite subgraphs of $\hat{P}$, and the error term $2p_is \gamma n^k$ is genuinely small compared to the main term because of the requirement that $S$ is not too small in the definition of `regular' and by choice of $\gamma$.
 
 This last equation is immediate since $(s^{-1}g_i,s^{-1}\tilde{g}_i)$ and $(s^{-1}\tilde{g}_i,G''_i)$ are both $\gamma$-discrepancy pairs.
\end{proof}

In this second proof we could allow each $G_i$ to be a weighted $k$-uniform hypergraph (i.e.\ a function from $\binom{[n]}{k}$ to $\mathbb{R}_{\ge0}$) without any change; thus Lemma~\ref{lem:ssshrl} applies in the weighted setting of~\cite{ADS}. We will not need this strengthening here, but note it for future reference.

We now prove, as promised, that a strengthened partition contains few irregular polyads.

\begin{proposition}\label{prop:fewirreg}
	Given $k\in\mathbb{N}$, suppose that $t_0\in\mathbb{N}$ is sufficiently large. Given any constants $\delta,d,\eps_k>0$, any function $\eps:\mathbb{N}\to(0,1]$ which tends to zero sufficiently fast, any $t_1,t_2\in\mathbb{N}$, any $0<f_k\le\eps_k^2$ and any $f>0$, there exists $\eta>0$ such that the following holds for any sufficiently large $n$ and any $p>0$. Suppose $G$ is an $n$-vertex hypergraph. Suppose that $(\cP^*_c,\cP^*_f)$ is a $(t_0,t_1,t_2,\eps_k,\eps,f_k,f,p)$-strengthened pair for $G$.
 Suppose that $\cP^*_c$ has $t$ clusters and density vector ${\mathbf{d}}=(d_{k-1},
  \dots, d_2)$.
 
 Then the number of irregular polyads of $\cP^*_c$ for $G$ is at most
 \[4\eps_k\binom{t}{k}\prod_{i=2}^{k-1}d_i^{-\binom{k}{i}}\,.\]
\end{proposition}
\begin{proof}
 We require $t_0$ to be sufficiently large that $3\big(1-\tfrac{1}{100}\big)^{-1}\tfrac{t_0^k}{k!}\le 4\binom{t_0}{k}$ (and so the same holds for each $t\ge t_0$). We also require $\eps(t_1)$ to be small enough, and $n/t_1$ large enough, for Lemma~\ref{lem:DCL}, with $\alpha=1$, $d_0=t_1^{-1}$, and $\gamma=\tfrac1{100}$.

 Recall that a polyad $\hat{P}(Q;\cP^*_c)$ is irregular with respect to $G$ if either $G$ is not $(\eps_k,p)$-regular with respect to $\hat{P}(Q;\cP^*_c)$, or for more than an $\eps_k$-fraction of the $k$-sets $Q'$ supported on $\hat{P}(Q;\cP^*_c)$, $G$ is not $(f_k,p)$-regular with respect to $\hat{P}(Q';\cP^*_f)$, or for more than an $\eps_k$-fraction of the $k$-sets $Q'$ supported on $\hat{P}(Q;\cP^*_c)$, we have $d_p\big(G\big|\hat{P}(Q';\cP^*_f)\big)\neq d_p\big(G\big|\hat{P}(Q;\cP^*_c)\big)\pm\eps_k$.
 
 By definition and~\ref{S:creg}, there are at most $\eps_k\binom{n}{k}$ sets $Q$ in $\cross_k(\cP_c)$ such that $G$ is not $(\eps_k,p)$-regular with respect to $\hat{P}(Q;\cP^*_c)$. Since by Lemma~\ref{lem:DCL} for each $Q\in\cross_k(\cP_c)$ the number of $k$-sets supported by $\hat{P}(Q;\cP^*_c)$ is at least
 \[(1-\tfrac{1}{100})\big(\tfrac{n}{t}\big)^k\prod_{i=2}^{k-1}d_i^{\binom{k}{i}},\]
 it follows that there are at most
 \[\eps_k\binom{n}{k}(1-\tfrac{1}{100})^{-1}\big(\tfrac{t}{n}\big)^k\prod_{i=2}^{k-1}d_i^{-\binom{k}{i}}\le(1-\tfrac{1}{100})^{-1}\eps_k\tfrac{t^k}{k!}\prod_{i=2}^{k-1}d_i^{-\binom{k}{i}}\]
 $k$-polyads in $\cP^*_c$ with respect to which $G$ is not $(\eps_k,p)$-regular.
 
 Suppose that $\mathbf{d}^f=(d^f_{k-1},\dots,d^f_2)$ is the density vector of $\cP^*_f$, and $t_f$ is the number of its clusters. By definition and~\ref{S:freg}, there are at most $f_k\binom{n}{k}$ sets in $\cross_k(\cP_f)$ such that $G$ is not $(f_k,p)$-regular with respect to $\hat{P}(Q;\cP^*_f)$, and by~\ref{S:dens}, there are at most $\eps_k^2\binom{n}{k}$ sets $Q$ in $\cross_k(\cP_c)$ such that $d_p\big(G\big|\cP^*_c(Q)\big)\neq d_p\big(G\big|\cP^*_f(Q)\big)\pm\eps_k$. By choice of $f_k$, there are in total at most $2\eps_k^2\binom{n}{k}$ sets $Q$ in $\cross_k(\cP_f)$ which fail either condition, and so at most $2\eps_k^2\binom{n}{k}$ sets $Q$ in $\cross_k(\cP_c)$ which fail either condition. Now a polyad $\hat{P}(Q;\cP^*_c)$ which is irregular, but with respect to which $G$ is $(\eps_k,p)$-regular, by Lemma~\ref{lem:DCL} supports at least
 \[\eps_k(1-\tfrac{1}{100})\big(\tfrac{n}{t}\big)^k\prod_{i=2}^{k-1}d_i^{\binom{k}{i}}\]
 $k$-sets which fail one of these two conditions. It follows that the number of polyads $\hat{P}(Q;\cP^*_c)$ which are irregular, but with respect to which $G$ is $(\eps_k,p)$-regular, is at most
 \[2\eps_k^2\binom{n}{k}\cdot \eps_k^{-1}(1-\tfrac{1}{100})^{-1}\big(\tfrac{t}{n}\big)^k\prod_{i=2}^{k-1}d_i^{-\binom{k}{i}}\le 2\eps_k(1-\tfrac{1}{100})^{-1}\tfrac{t^k}{k!}\prod_{i=2}^{k-1}d_i^{-\binom{k}{i}}\,.\]
 Putting these together, we see that the total number of irregular $k$-polyads in $\cP^*_c$ is at most
 \[ 3\eps_k(1-\tfrac{1}{100})^{-1}\tfrac{t^k}{k!}\prod_{i=2}^{k-1}d_i^{-\binom{k}{i}}\le 4\eps_k\binom{t}{k}\prod_{i=2}^{k-1}d_i^{-\binom{k}{i}}\,.\]
\end{proof}

Building on this, we next prove Lemma~\ref{lem:goodconnected}.
\begin{proof}[Proof of Lemma~\ref{lem:goodconnected}]
 Given $k$ and $d$, we require $t_0>\tfrac{4k}{d}$, and we then choose $0<\gamma<\tfrac{1}{100}$ small enough so that
 \[\frac{\delta-\tfrac{2k}{t_0}-(1+\gamma)^4\cdot 2^{k+1}\eps_k^{1/k}-d}{(1+\gamma)^4}>\delta-2d-2^{k+2}\eps_k^{1/k}\,.\]
 
 We additionally require $t_0$ to be large enough, and $\eps(t_1)$ small enough, for Proposition~\ref{prop:fewirreg}, so that $\cP^*_c$ supports at most
 \[4\eps_k\binom{t}{k}\prod_{i=2}^{k-1}d_i^{-\binom{k}{i}}\]
 irregular $k$-polyads, and also $\eps(t_1)$ small enough, and $\tfrac{n}{t_1}$ large enough, for Lemma~\ref{lem:DCL} with $d_0=t_1^{-1}$ and $\gamma$. For convenience in this proof, we define $d_k=1$.
 
 We require $\eta$ to be sufficiently small compared to all the products of $d_i$ and $\tfrac{1}{t_1}$ which appear in the following proof: concretely, any
 $\eta\le 100t_1^{-2k2^k}$
 suffices.
 
 Our first aim is to show that $\cR_{\eps_k}(G;\cP^*_c,\cP^*_f)$ contains many $1$-edges. To start with, we mark the $k$-edges of the multicomplex $\cP^*_c$ which are irregular as \emph{bad}. We then, for each $i=k-1,\dots,1$ in succession, mark as bad all $i$-edges which are contained in at least 
 \[4\eps_k^{1/k}t\prod_{j=2}^{i+1}d_{j}^{-\binom{i}{j-1}}\]
 bad $(i+1)$-edges. Now consider the following construction. We begin by taking any bad $1$-edge, then any bad $2$-edge containing it, and so on until we obtain a bad $k$-edge together with an order on its vertices. Clearly we obtain any given bad $k$-edge in at most $k!$ ways by following this process (since a $k$-edge together with an order determines the chosen edges in the process). If there exist at least $4\eps_k^{1/k}t$ bad $1$-edges, it follows that the number of bad $k$-edges is at least
 \[\tfrac{1}{k!}\prod_{i=0}^{k-1}\Big( 4\eps_k^{1/k}t\prod_{j=2}^{i+1}d_j^{-\binom{i}{j-1}}\Big)\ge 4\eps_k\tfrac{t^k}{k!}\prod_{i=2}^kd_i^{-\binom{k}{i}}> 4\eps_k\binom{t}{k}\prod_{i=2}^kd_i^{-\binom{k}{i}}\,, \]
 which is in contradiction to Proposition~\ref{prop:fewirreg}. 
 
 We claim that any edge of $\cP^*_c$ which is not in $\cR_{\eps_k}(G;\cP^*_c,\cP^*_f)$ is either bad or contains a bad edge. To see this, consider the process of obtaining $\cR_{\eps_k}(G;\cP^*_c,\cP^*_f)$ by successively removing edges which either fail one of~\ref{rg:k} or~\ref{rg:lower}, or which contain a removed edge. Suppose for a contradiction that at some stage in this process we remove an edge which is neither bad nor contains a bad edge; let $e$ be the first such edge removed. Observe that $e$ cannot have been removed for failing~\ref{rg:k}, since edges which fail this condition are bad. Furthermore $e$ cannot have been removed for being unsupported, because all edges previously removed either were bad or contain bad edges, and by assumption $e$ contains no bad edges. So $e$ was removed for failing~\ref{rg:lower}. In other words, we have $|e|\le k-1$ and $e$ is neither bad nor contains a bad edge, but nevertheless there are many $(|e|+1)$-edges containing $e$ which either are bad or contain a bad edge.
 
 Suppose that $f$ is a bad edge such that $|f\setminus e|=1$. If $|f|=1$, then there are at most $4\eps_k^{1/k}t$ choices of $f$, each of which is contained in $\prod_{i=2}^{|e|+1}d_i^{-\binom{|e|}{i-1}}$ edges of $\cP^*_c$ of uniformity $|e|+1$ which contain $e$. If $|f|>1$, then $f\cap e$ is non-empty and not a bad edge. There are at most $\binom{|e|}{\ell}$ choices of $f\cap e$ with $\ell$ elements, each of which by definition is contained in less than
 \[4\eps_k^{1/k}t\prod_{i=2}^{\ell+1}d_{i}^{-\binom{\ell}{i-1}}\]
 bad edges. Thus, there are at most
 \[\binom{|e|}{\ell}\cdot 4\eps_k^{1/k}t\prod_{i=2}^{\ell+1}d_{i}^{-\binom{\ell}{i-1}}\]
 choices of $f$, each of which is contained in
 \[\prod_{i=2}^{|e|+1}d_i^{-\sum_{j=1}^{i-1}\binom{\ell}{i-j-1}\binom{|e|-\ell}{j}}\]
 edges of $\cP^*_c$ of uniformity $|e|+1$ which contain $e$.
 This holds as $\cP^*_f$ is $(t_0,t_1,\eps)$-equitable partition.
 
 Summing up, the number of $(|e|+1)$-edges containing $e$ which are either bad or contain a bad edge is at most
 \begin{align*}
  &4\eps_k^{1/k}t\cdot \prod_{i=2}^{|e|+1}d_i^{-\binom{|e|}{i-1}}+\sum_{\ell=1}^{|e|}\binom{|e|}{\ell}\cdot 4\eps_k^{1/k}t\Big(\prod_{i=2}^{\ell+1}d_{i}^{-\binom{\ell}{i-1}}\Big)\cdot \prod_{i=2}^{|e|+1}d_i^{-\sum_{j=1}^{i-1}\binom{\ell}{i-j-1}\binom{|e|-\ell}{j}}\\
  &=4\eps_k^{1/k}t\Bigg(\sum_{\ell=0}^{|e|}\binom{|e|}{\ell}\Bigg)\prod_{i=2}^{|e|+1}d_i^{-\binom{|e|}{i-1}}=2^{|e|+2}\eps_k^{1/k}t\prod_{i=2}^{|e|+1}d_i^{-\binom{|e|}{i-1}}\,.
 \end{align*}
 But this last equation simply states that $e$ does \emph{not} fail~\ref{rg:lower}, which is our desired contradiction. In particular, every $1$-edge which is not bad is in $\cR_{\eps_k}(G;\cP^*_c,\cP^*_f)$, and so also in $\cR_{\eps_k,d}(G;\cP^*_c,\cP^*_f)$. There are at least $\big(1-4\eps_k^{1/k}\big)t$ such $1$-edges.
 
 Next, we prove that every $(k-1)$-edge of $\cR_{\eps_k,d}(G;\cP^*_c,\cP^*_f)$ is contained in sufficiently many $k$-edges.
 
 By Lemma~\ref{lem:DCL}, given any $(k-1)$-edge $E$ of $\cR_{\eps_k,d}(G;\cP^*_c,\cP^*_f)$, we know that $E$ contains $e$ $(k-1)$-sets, where $e$ is in the range
 \[ (1 \pm \gamma) (\tfrac nk)^{k-1} \prod_{i=2}^{k-2} d_i^{\binom{k-1}{i}} \]
 $(k-1)$-sets; we need the lower bound to justify some applications of $(\eta,p)$-upper regularity. First, consider the edges of $G$ which contain a $(k-1)$-set in $E$ and another vertex from the clusters of $E$. These edges are all rainbow for $E$ together with $k-1$ copies of the complete $(k-1)$-graph on the union of the clusters of $E$, and any such edge contains at most two sets in $E$. The total number of $k$-sets which are rainbow for $E$ together with the $k-1$ complete graphs on the clusters of $E$ is at least $\tfrac{1}{2}e\cdot (\tfrac{n}{t_1}-1)>\eta n^k$, 
 by choice of $\eta$ and since every cluster has at least $\tfrac{n}{t_1}$ vertices.
 It is also at most $(k-1)\tfrac{n}{t_0}e$, since every cluster has at most $\tfrac{n}{t_0}$ vertices, and so by $(\eta,p)$-upper regularity of $G$, the total number of edges of $G$ which contain a $(k-1)$-set in $E$ and another vertex from the clusters of $E$ is at most $p\cdot (k-1)\tfrac{n}{t_0}e + p \eta n^k<kp\tfrac{n}{t_0}e$.
 
 Since each $(k-1)$-set of $E$ is contained in at least $\delta p n$ edges of $G$, it follows that the total number of edges of $G$ which consist of an $(k-1)$-set of $E$ together with a vertex not in the clusters of $E$ is at least
 \[\big(\delta-\tfrac{2k}{t_0}\big)epn\,.\]
 Now these edges of $G$ are partitioned according to the $k$-polyad of $\cP^*_c$ containing them. By Lemma~\ref{lem:DCL}, each of these $k$-polyads supports \[\big(1\pm \gamma\big)^3 e \tfrac{n}{t} \prod_{i=2}^{k-1}d_i^{\binom{k-1}{i-1}}\]
 $k$-sets, and hence by $(\eta,p)$-upper regularity of $G$, at most
 \[\big(1+ \gamma\big)^4 e p \tfrac{n}{t} \prod_{i=2}^{k-1}d_i^{\binom{k-1}{i-1}}\]
 edges of $G$. Since $E$ is in $\cR_{\eps_k,d}(G;\cP^*_c,\cP^*_f)$, at most
 \[2^{k+1}\eps_k^{1/k}t\prod_{i=2}^{k-1}d_i^{-\binom{k-1}{i-1}}\]
 of these polyads are not in $\cR_{\eps_k}(G;\cP^*_c,\cP^*_f)$, so the total number of edges of $G$ in such $k$-polyads of $\cP^*_c$ containing $E$ is at most
 \[2^{k+1}\eps_k^{1/k}t\Bigg(\prod_{i=2}^{k-1}d_i^{-\binom{k-1}{i-1}}\Bigg)\cdot \big(1+ \gamma\big)^4 e p \tfrac{n}{t}\prod_{i=2}^{k-1}d_i^{\binom{k-1}{i-1}}=(1+\gamma)^4\cdot 2^{k+1}\eps_k^{1/k} epn\,.\] 
  By definition, the number of edges of $G$ in $k$-polyads of $\cP^*_c$ containing $E$ with respect to which the $p$-density of $G$ is less than $d$ is at most $depn$. It follows that all the remaining edges of $G$ which contain a $k$-set of $E$ together with a vertex outside the clusters of $E$ are contained in edges of $\cR_{\eps_k,d}(G;\cP^*_c,\cP^*_f)$, and so there are at least
 \[\frac{\big(\delta-\tfrac{2k}{t_0}\big)epn-(1+\gamma)^4\cdot 2^{k+1}\eps_k^{1/k} epn-depn}{\big(1+ \gamma\big)^4 e p \tfrac{n}{t} \prod_{i=2}^{k-1}d_i^{\binom{k-1}{i-1}}}=\frac{\delta-\tfrac{2k}{t_0}-(1+\gamma)^4\cdot 2^{k+1}\eps_k^{1/k}-d}{(1+\gamma)^4}t\prod_{i=2}^{k-1}d_i^{-\binom{k-1}{i-1}}\]
 such edges, as desired.
 
  Finally, we need to show that if $\delta>\tfrac12+2d+2^{k+2}\eps_k^{1/k}+\nu$ then any induced subcomplex $\cR' \subseteq \cR$ on at least $(1-\nu)t$ $1$-edges is tightly linked. In other words, we need to show that if $\mathbf{u}$ on the vertices $(u_1,\dots,u_{k-1})$ and $\mathbf{v}$ on the vertices $(v_1,\dots,v_{k-1})$ are any two ordered $(k-1)$-edges of $\cR' \subseteq \cR= \cR_{\eps_k,d}(G;\cP^*_c,\cP^*_f)$, there is a tight link in $\cR'$ from $\mathbf{u}$ to $\mathbf{v}$. This proof does not require any further regularity theory; simply the properties of $\cR$ we already deduced.
  
  The critical observation we need is the following. If $f$ is any $(k-1)$-edge in $\cR'$, then it is an edge in $\cR$ and so is in at least
  \[\big(\delta-2d-2^{k+2}\eps_k^{1/k}\big)t\prod_{i=2}^{k-1}d_i^{-\binom{k-1}{i-1}}>\big(\tfrac12+\nu)t\prod_{i=2}^{k-1}d_i^{-\binom{k-1}{i-1}}\]
  $k$-edges of $\cR$. Of these, at most
  \[\nu t\prod_{i=2}^{k-1}d_i^{-\binom{k-1}{i-1}}\]
  use one of the $1$-edges not in $\cR'$, and so $f$ is in strictly more than
  \[\tfrac12t\prod_{i=2}^{k-1}d_i^{-\binom{k-1}{i-1}}\]
  $k$-edges of $\cR'$.
  
  We build up the desired tight link vertex by vertex. For each $1\le j\le k-1$, we will choose a $1$-cell $w_j$
   of $\cR'$, disjoint from $\{u_1,\dots,u_{k-1},v_1,\dots,v_{k-1}\}$. At the $j$th step, we in addition choose for each $S\subset [j-1]$ with $S \not= \emptyset$
   an $(|S|+1)$-cell
   $w_{S\cup\{j\}}$ of $\cR'$ supported by $w_S$ and the cells $w_{S'\cup\{j\}}$ for $S'\subset S$ with $|S'|=|S|-1$. Finally we choose two $k$-edges $e_{j,u}$ and $e_{j,v}$ of $\cR$, where $e_{j,u}$ has underlying $1$-cells $u_j,\dots,u_{k-1},w_1,\dots,w_j$ and $e_{j,v}$ has underlying $1$-cells $v_j,\dots,v_{k-1},w_1,\dots,w_j$ and where both $e_{j,u}$ and $e_{j,v}$ have $w_{[j]}$ in their $(k-j)$-times-iterated boundaries. We insist additionally that $\partial e_{1,u}$ contains the $(k-1)$-cell underlying $\mathbf{u}$, and $\partial e_{1,v}$ contains the $(k-1)$-cell underlying $\mathbf{v}$, and that for each $2\le j\le k-1$ we have $\big(\partial e_{j-1,u}\big)\cap\big(\partial e_{j,u}\big),\big(\partial e_{j-1,v}\big)\cap\big(\partial e_{j,v}\big)\neq\emptyset$.
  
  For $j=1$, consider $\tpl{u}$ and $\tpl{v}$ as $(k-1)$-edges of $\cR'$. We need to find a $1$-cell $w_1$ of $\cR'$ such that there are $k$-edges of $\cR'$ whose boundaries contain $\tpl{u}$ and $\tpl{v}$ respectively, and both of which use the $1$-edge $w_1$. If no such edges exist, then for one of $\tpl{u}$ and $\tpl{v}$ (without loss of generality, suppose it is $\tpl{u}$) there are at most $\tfrac12t$ $1$-edges which support edges of $\cR'$ using $\tpl{u}$. Thus $\tpl{u}$ is in at most
  \[\tfrac12t\prod_{i=2}^{k-1}d_i^{-\binom{k-1}{i-1}}\]
  $k$-edges of $\cR'$,
   which is a contradiction since $\tpl{u}$ is an edge of $\cR'$.
  
  Given $2\le j\le k-1$, suppose $w_S$ has been constructed for each $S\subset[j-1]$ with $S\neq\emptyset$. Consider the collection $W^*$ of $j$-edges in $\cR'$ whose boundary contains $w_{[j-1]}$. Any $k$-edge of $\cR'$ whose $(k+1-j)$-times-iterated boundary contains $w_{[j-1]}$ has $(k-j)$-times-iterated boundary containing some $w^*\in W^*$. In particular, consider the $(k-1)$-edges $e^*_u$ in $\partial e_{u,j-1}$ which does not use $u_{j-1}$, and $e^*_v$ in $\partial e_{v,j-1}$ which does not use $v_{j-1}$. Both these edges have $w_{[j-1]}$ in their $(k-j)$-times iterated boundary. What we want is some $w^*\in W^*$ such that there are $k$-edges $e_{u,j}$ whose boundary contains $e^*_u$, and $e_{v,j}$ whose boundary contains $e^*_u$, both of whose $(k-j)$-times-iterated boundaries contain $w^*$. Assuming such a $w^*$ exists, we can then let $w_{[j]}=w^*$ and for each $S \subsetneq [j-1]$ with $S \neq \emptyset$, we let $w_{S \cup \{j\}}$ be the unique $(|S|+1)$-edge on the vertices $\big\{w_i: i \in S \cup \{j\}\big\}$ that is in the $(j-|S|-1)$-iterated boundary of $w_{[j]}$.
  
  Suppose for a contradiction that no such $w^*$ exists. Then each $w^*\in W^*$ can be assigned to at most one of $u$ and $v$, according to whether it is in an edge with $e^*_u$ or with $e^*_v$. We have
  \[|W^*|\le t\prod_{i=2}^{j}d_i^{-\binom{j-1}{i-1}}\]
  and hence for one of $u$ and $v$, the number of elements of $W^*$ assigned to it is at most
  \[\tfrac12t\prod_{i=2}^{j}d_i^{-\binom{j-1}{i-1}}\,.\]
  Suppose without loss of generality this is $u$. For any given $w^*\in W^*$, the number of $k$-edges of $\cR$ whose boundary contains $e^*_u$ and whose $(k-j)$-times iterated boundary contains $w^*$ is at most
  \[\prod_{i=2}^{k-1}d_i^{-\binom{k-1}{i-1}+\binom{j-1}{i-1}}\]
  where we use the convention $\binom{j-1}{i-1}=0$ if $i>j$. Hence the number of $k$-edges of $\cR'$ whose boundary contains $e^*_u$ is at most
  \[\tfrac12 t\prod_{i=2}^{j}d_i^{-\binom{j-1}{i-1}}\cdot\prod_{i=2}^{k-1}d_i^{-\binom{k-1}{i-1}+\binom{j-1}{i-1}}=\tfrac12 t\prod_{i=2}^{k-1}d_i^{-\binom{k-1}{i-1}}\,.\]
  But $e^*_u$ is a $(k-1)$-edge of $\cR'$, so this is a contradiction. 
\end{proof}

To prove the remaining lemmas we first state the Dense Counting Lemma with parts of the same size for more general graphs.

\begin{lemma}[{\cite[Theorem 6.5]{KRS}}]
	\label{lem:DCL-general}
	For all integers $k\ge2$ and constants $\gamma,d_0>0$, there exists $\eps>0$ such that the following holds. Let $\mathbf{d}=(d_{k-1},\dots,d_2)$ be a vector of real numbers with $d_i\ge d_0$ for each $2\le i\le k-1$, let $G$ be a $k$-partite $(k-1)$-complex with parts of size $m\ge\eps^{-1}$ which is $(d_{k-1},\dots,d_2,\eps,1)$-regular, and let $H$ be any $k$-partite $(k-1)$-complex on at most $2k$ vertices with fixed partition classes.
	Then the number of copies of $H$ in $G$, with the vertices of the $i$th class of $H$ embedded into the $i$th class of $G$ for $i=1,\dots,k$, is
	\[\big(1\pm\gamma\big)m^{v(H)}\prod_{i=2}^{k-1}d_i^{e_i(H)}\, ,\]
	where $e_i(H)$ is the number of edges of size $i$ in $H$.
\end{lemma}

\begin{proof}[Proof of Lemma~\ref{lem:DCL}]
	Let $k \ge 2$ be an integer and $\alpha,\gamma,d_0>0$.
	Let $\eps>0$ be small enough got Lemma~\ref{lem:RRL} on input $k$, $\alpha$, and $d_0$.
	Additionally assume that $\sqrt{\eps}$ is small enough for Lemma~\ref{lem:DCL-general} on input $k$, $\gamma$, and $d_0$.
	Let $\mathbf{d}=(d_{k-1},\dots,d_2)$ be a vector of real numbers with $d_i\ge d_0$ for each $2\le i\le k$, and let $G$ be a $k$-partite $(k-1)$-complex with parts $V_1,\dots,V_k$ of size $m\ge \alpha^{-1} \eps^{-1}$ which is $(\mathbf{d},\eps,1)$-regular.
	
	Let $V_i' \subseteq V_i$ be of size $|V_i'| \ge \alpha |V_i|$ for $i=1,\dots,k$.
	Our goal for he first part is to estimate the number of copies of the $k$-vertex complete $(k-1)$-complex in $G[V_1' , \dots , V_k']$.
	We denote this number by $E$.
	Note that for any $V_i'' \subseteq V_i'$ each of size $\alpha m$ the induced subcomplex $G[V_1'' , \dots , V_k'']$ is $(\mathbf{d},\sqrt{\eps},1)$-regular by Lemma~\ref{lem:RRL}.
	Therefore, by Lemma~\ref{lem:DCL-general}, there number of copies of the $k$-vertex complete $(k-1)$-complex in $G[V_1'' , \dots , V_k'']$ is
	\[ (1 \pm \gamma) (\alpha m)^{k} \prod_{i=2}^{k-1} d_i^{\binom{k}{i}} \, . \]
	
	We choose sets $V_i'' \subseteq V_i$ of size $\alpha m$ uniformly at random and note that the expected number of copies of the $k$-vertex complete $(k-1)$-complex in $G[V_1'' , \dots , V_k'']$ is $E \prod_{i=1}^k \tfrac{\alpha m}{|V_i'|}$.
	With the estimate from above it follows that
	\[ E \prod_{i=1}^k \tfrac{\alpha m}{|V_i'|}= (1 \pm \gamma) (\alpha m)^{k} \prod_{i=2}^{k-1} d_i^{\binom{k}{i}} \]
	and rearranging this for $E$ finishes the proof of the first part.
	The second part follows analogously.
\end{proof}

\begin{proof}[Proof of Lemma~\ref{lem:DCL-variant}]
	We will prove the first statement, then deduce the second, and note that the third also follows.
	Let $k \ge 2$ be an integer and $\alpha,\gamma, d_0>0$.
	Let $0<\eps<\tfrac \gamma4$ be small enough for Lemma~\ref{lem:DCL} on input $k$, $\alpha$, $\gamma_0 = \tfrac{1}{36} \gamma^3$, and $d_0$.
	Additionally, assume that $\eps$ is small enough for Lemma~\ref{lem:DCL} with the same input and $k-1$ instead of $k$ and also small enough for the generalisation of Lemma~\ref{lem:DCL} for counting two copies of the $k$-vertex complete $(k-1)$-complex that overlap on the first $k-1$ vertices\footnote{This follows from Lemma~\ref{lem:DCL-general} by exactly the same argument as in our proof of Lemma~\ref{lem:DCL}.} with the same input as above.
	Let $\mathbf{d}=(d_{k-1},\dots,d_2)$ be a vector of real numbers with $d_i\ge d_0$ for each $2\le i\le k$, and let $G$ be a $k$-partite $(k-1)$-complex with parts $V_1,\dots,V_k$ of size $m\ge\alpha^{-1}\eps^{-1}$ which is $(\mathbf{d},\eps,1)$-regular.
	
	Let $V_i' \subseteq V_i$ be of size $|V_i'| \ge \alpha |V_i|$ for $i=1,\dots,k$ and $G'=G[V_1', \dots , V_k']$.
	For a random $(k-1)$-tuple $e$ from $G[V_1', \dots , V_{k-1}']$ we denote by $X_e$ the number of copies of the $k$-vertex complete $(k-1)$-complex in $G'$ that $e$ is contained in.
	With Lemma~\ref{lem:DCL} we can count the total number of copies of the $k$-vertex complete $(k-1)$-complex in $G'$ and the number of choices for $e$ to obtain
	\[\mathbb{E}[X_e] = \frac{(1\pm \gamma_0)}{(1 \mp \gamma_0)} \prod_{i=1}^{k}|V_i'| \prod_{i=2}^{k-1} d_i^{\binom{k}{i}} \cdot \prod_{i=1}^{k-1} |V_i'|^{-1} \prod_{i=2}^{k-1} d_i^{-\binom{k-1}{i}} = (1 \pm 3 \gamma_0) |V_k'| \prod_{i=2}^{k-1} d_i^{\binom{k-1}{i-1}}\,.\]
	Similarly, we can count the number of two copies of the $k$-vertex complete $(k-1)$-complex that overlap on the first $k-1$ vertices to get
	\[\mathbb{E}[X_e^2] = \frac{(1\pm \gamma_0)}{(1 \mp \gamma_0)} |V_k'| \prod_{i=1}^k |V_i'| \prod_{i=2}^{k-1} d_i^{\binom{k}{i}+\binom{k-1}{i-1}} \cdot \prod_{i=1}^{k-2}|V_i'| \prod_{i=2}^{k-1} d_i^{-\binom{k-1}{i}} = (1 \pm 3 \gamma_0) |V_k'|^2 \prod_{i=2}^{k-1} d_i^{2\binom{k-1}{i-1}}\,.\]
	This implies
	\[V[X_e] = \mathbb{E}[X_e^2] - \mathbb{E}[X_e]^2 \le  9 \gamma_0 |V_k'|^2 \prod_{i=2}^{k-1} d_i^{2\binom{k-1}{i-1}}\]
	and, with Chebyshev's inequality, we infer
	\[\mathbb{P}\left[ \left| X_e - \EE[X] \right| \ge \frac{\gamma}{2} |V_k'| \prod_{i=2}^{k-1} d_i^{\binom{k-1}{i-1}} \right] \le \frac{9 \gamma_0}{ (\tfrac \gamma2)^2} \le \gamma.\]
	Thus, at least a $(1-\gamma)$-fraction of the $(k-1)$-tuples in $G[V_1', \dots , V_{k-1}']$ is contained in
	\[\big(1\pm \tfrac \gamma2\big) (1 \pm 3 \gamma_0) |V_k'| \prod_{i=2}^{k-1} d_i^{\binom{k-1}{i-1}} = (1 \pm \gamma) |V_k'| \prod_{i=2}^{k-1} d_i^{\binom{k-1}{i-1}} \]
	copies of the $k$-vertex complete $(k-1)$-complex in $G'$.
	
	For the counting statement, let $S$ denote the $\gamma$-fraction of edges $e$ with largest $X_e$.
	Suppose that $\mathbb{E}\big[X_e|e\in S\big]\ge 2\mathbb{E}[X_e]$, as otherwise we get
	\[\sum_{e \in S} X_e \le |S| \cdot \mathbb{E}\big[X_e|e\in S\big] < 2 |S| \cdot \mathbb{E}[X_e] \le \tfrac52 \gamma \prod_{i=1}^k |V_i'| \prod_{i=2}^{k} d_i^{\binom{k}{i-1}}\, .\]
	Then we have
	\[\mathbb{E}[X_e^2]=\gamma \mathbb{E}\big[X_e^2|e\in S\big]+(1-\gamma)\mathbb{E}\big[X_e^2|e\not\in S\big]\ge \gamma \mathbb{E}\big[X_e|e\in S\big]^2+(1-\gamma)\mathbb{E}\big[X_e|e\not\in S\big]^2\]
	where the inequality is Jensen's inequality (since the second moment function is convex). Now we have $\mathbb{E}\big[X_e|e\in S\big]=\mathbb{E}\big[X_e\big]+c$ for some $c\ge\mathbb{E}\big[X_e\big]$, and $\mathbb{E}\big[X_e| e \not\in S\big]=\mathbb{E}\big[X_e\big]-\tfrac{\gamma}{1-\gamma}c$. Plugging these in we get
	\[\mathbb{E}[X_e^2]\ge (1+\gamma)\mathbb{E}[X_e]^2\] which is in contradiction to the above bound on $V[X_e]$.
	
	For the third statement we apply the first statement to the first $k-1$ parts $V_1,\dots,V_{k-1}$, which is possible by the choice of $\eps$ above.
	We find that a $(1-\gamma)$-fraction of the $(k-2)$-tuples in $G[V_1' , \dots , V_{k-2}']$ is contained in
	\[\big(1\pm \tfrac \gamma2\big)(1 \pm 3 \gamma_0) |V_{k-1}'| \prod_{i=2}^{k-2} d_i^{\binom{k-2}{i-1}}\]
	copies of the $(k-1)$-vertex $(k-2)$-complex in $G'$ together with a vertex from $V_{k-1}'$.
	Then, with $(d_{k-1},\eps,1)$-regularity, we get that a $(1 \pm \eps) d_{k-1}$-fraction of these give copies of the $(k-1)$-vertex $(k-1)$-complex in $G'$.
	The statement follows as $(1 \pm \tfrac \gamma2)(1 \pm 3 \gamma_0)(1 \pm \eps) = (1 \pm \gamma)$.
\end{proof}

\begin{proof}[Proof of Lemma~\ref{lem:MDL}]
Let $k\ge 3$ be an integer $\gamma,\delta, d_0>0$ and w.l.o.g.~assume $\delta \ge 2\gamma$.
We choose $0 <\gamma_0 < \gamma^2 \delta 2^{-k}$ and $0<\eps<\tfrac 12 \gamma_0 d_0^{10k2^{10k}}$.
Let $\mathbf{d}=(d_{k-1},\dots,d_2)$ be a vector of real numbers with $d_i\ge d_0$ for each $2\le i\le k-1$, and let $G$ be a $k$-partite $(k-1)$-complex with parts $V_1,\dots,V_k$ of size $m\ge\eps^{-1}$ which is $(\mathbf{d},\eps,1)$-regular.
Given integers $a,b,c$ such that $a+b-c=k$, let $A$ be part of an $a$-cell, $B$ be part of a $b$-cell of size $|B| \ge \gamma m^b \prod_{i=2}^{k-2}d_i^{\binom{b}{i}}$, and $C$ be part of a $c$-cell such that the tuples from $C$ have degree $(\delta\pm\gamma) m^{a-c} \prod_{i=2}^{k-1}d_i^{\binom{a}{i}-\binom{c}{i}}$ into $A$ and every tuple of $B$ is contained in a tuple of $C$.

We denote by $F_0$ the $k$-vertex $(k-1)$-complex obtained from the union of a complete $a$-vertex $a$-complex with a complete $b$-vertex $b$-complex identified on $c$ vertices.
We fix the canonical labelling $v_1,\dots,v_k$ of the vertices from $F_0$ such that the $v_i$ vertices of the $a$- and $b$-vertex graph constructed above correspond to the clusters $V_i$ of $G$ in $A$ and $B$, respectively.
Then each tuple from $B$ is contained in $(\delta\pm\gamma) m^{a-c} \prod_{i=2}^{k-1}d_i^{\binom{a}{i}-\binom{c}{i}}$ copies of $F_0$ supported by $A$ and $B$ and, therefore, there are 
\[|B| (\delta\pm\gamma) m^{a-c} \prod_{i=2}^{k-1}d_i^{\binom{a}{i}-\binom{c}{i}} = \eta m^k \prod_{i=2}^{k-1}d_i^{e_i(F_0)} \]
copies of $F_0$ in $G$ supported by $A$ and $B$, where $\eta = (\delta\pm\gamma) |B| m^{-b} \prod_{i=2}^{k-2}d_i^{-\binom{b}{i}} \ge \delta^2 \gamma_0$.

Starting from $F_0$ and adding edges of size at most $k-1$ that are supported we obtain a sequence $F_0,\dots,F_t$ of $k$-vertex $(k-1)$-complexes, where $F_t$ is the complete $k$-vertex $(k-1)$-complex.
Note that $t=2^k - 2 - e(F_0)$.
We show by induction on $j$ that for each $0\le j\le t$, we have
\[ \big(1 \pm j \gamma_0\big)\eta m^k \prod_{i=2}^{k-1}d_i^{e_i(F_j)} \]
copies of $F_j$ in $G$ supported by $A$ and $B$. The $j=t$ case, and the choice of $\eta$ and as $\gamma_0 \le \gamma^2 \delta/t$, proves the lemma.

For a given $1\le j\le t$, let $e$ be the edge in $E(F_j) \setminus E(F_{j-1})$. Let $F'$ be the subgraph of $F_{j-1}$ induced on $V(F)\setminus e$, and denote by $X$ the $(k-|e|)$-tuples in $G$ from the clusters $V_i$ such that $v_i \not\in e$ that are copies of $F'$.
For each $x \in X$, let $F_{j-1}(x)$ and $F_j(x)$ be the number of copies of $F_{j-1}$ and $F_j$, respectively, in $G$ supported by $A$ and $B$ that contain $x$.
Then, by assumption, we have that
\[ \sum_{x \in X} F_{j-1}(x) =(1\pm(j-1)\gamma_0)\eta m^k \prod_{i=2}^{k-1}d_i^{e_i(F_{j-1})} \, . \]
We let $X'\subseteq X$ be those tuples, for which $F_{j-1}(x)  \ge \eps m^{|e|} $.
Then for any fixed $x \in X'$ we get from $(\mathbf{d},\eps,1)$-regularity of $G$ that $F_j(x)= (d_{|e|}\pm\eps)F_{j-1}(x)$. For any $x\in X$, we additionally have $F_j(x)\le F_{j-1}(x)$.

Observe that we have
\[\sum_{x\in X\setminus X'}F_j(x)\le\sum_{x\in X\setminus X'}F_{j-1}(x)\le \eps m^k\,,\]
where we use the bound $|X\setminus X'|\le m^{k-|e|}$ and the definition of $X'$ for the final inequality. This gives
\[\sum_{x\in X'}F_{j-1}(x)=(1\pm\sqrt{\eps})\sum_{x\in X}F_{j-1}(x)\,.\]

We then have
\begin{align*}
 \sum_{x\in X}F_j(x)&=\sum_{x\in X'}F_j(x)+\sum_{x\in X\setminus X'}F_j(x)=\Bigg(\sum_{x\in X'}(d_{|e|}\pm\eps)F_{j-1}(x)\Bigg)\pm \eps m^k\\
 &=\Bigg((d_{|e|}\pm\eps)(1\pm\sqrt{\eps})\sum_{x\in X}F_{j-1}(x)\Bigg)\pm \eps m^k=\big(1\pm\tfrac12\gamma_0\big)d_{|e|}\sum_{x\in X}F_{j-1}(x)\,,
\end{align*}
as desired.
\end{proof}


\bibliographystyle{amsplain} 
\bibliography{TightResil}

\end{document}